\documentclass[onefignum,onetabnum]{siamonline220329}
\usepackage{gmmpreamble}

\usepackage{lineno}

\ifpdf
\hypersetup{
  pdftitle={Gaussian Mixture Taylor approximations for risk measure of PDEs}
  pdfauthor={D. Luo, J. Chen, P. Chen, O. Ghattas}
}
\fi

\usepackage{lipsum}
\usepackage{amsfonts}
\usepackage{graphicx}
\usepackage{epstopdf}
\usepackage{algorithmic}
\ifpdf
  \DeclareGraphicsExtensions{.eps,.pdf,.png,.jpg}
\else
  \DeclareGraphicsExtensions{.eps}
\fi

\usepackage{enumitem}
\setlist[enumerate]{leftmargin=.5in}
\setlist[itemize]{leftmargin=.5in}

\newsiamremark{remark}{Remark}
\newsiamremark{hypothesis}{Hypothesis}
\crefname{hypothesis}{Hypothesis}{Hypotheses}
\newsiamthm{claim}{Claim}

\headers{Gaussian Mixture Taylor approximations for risk measure of PDEs}{D. Luo, J. Chen, P. Chen, O. Ghattas}

\title{
Gaussian mixture Taylor approximations of risk measures constrained by PDEs with Gaussian random field inputs
\thanks{
\funding{This work support by the US National Science Foundation under awards DMS-2245674 and DMS-2245111.}}
}

\author{Dingcheng Luo\thanks{Oden Institute for Computational Engineering and Sciences, The University of Texas at Austin, Austin, TX
  (\email{dc.luo@utexas.edu}, \email{joshuawchen@utexas.edu}).}
\and Joshua Chen\footnotemark[2]
\and Peng Chen\thanks{School of Computational Science and Engineering, Georgia Institute of Technology, Atlanta, GA
  (\email{pchen402@gatech.edu})}
\and Omar Ghattas\thanks{Walker Department of Mechanical Engineering and Oden Institute for Computational Engineering and Sciences, The University of Texas at Austin, Austin, TX (\email{omar@oden.utexas.edu})}
}

\usepackage{amsopn}

\begin{document} 
\maketitle
\begin{abstract}
This work considers the computation of risk measures for quantities of interest governed by PDEs with Gaussian random field parameters using Taylor approximations. 
While efficient, Taylor approximations are local to the point of expansion, 
and hence may degrade in accuracy when the variances of the input parameters are large. 
To address this challenge, we approximate the underlying Gaussian measure by a mixture of Gaussians with reduced variance in a dominant direction of parameter space. 
Taylor approximations are constructed at the means of each Gaussian mixture component, which are then combined to approximate the risk measures. 
The formulation is presented in the setting of infinite-dimensional Gaussian random parameters for risk measures including the mean, variance, and conditional value-at-risk. 
We also provide detailed analysis of the approximations errors arising from two sources: the Gaussian mixture approximation and the Taylor approximations.
Numerical experiments are conducted for a semilinear advection-diffusion-reaction equation with a random diffusion coefficient field and for the Helmholtz equation with a random wave speed field. 
For these examples, the proposed approximation strategy can achieve less than $1\%$ relative error in estimating CVaR with only $\mathcal{O}(10)$ state PDE solves, which is comparable to a standard Monte Carlo estimate with $\mathcal{O}(10^4)$ samples, 
thus achieving significant reduction in computational cost.
The proposed method can therefore serve as a way to rapidly and accurately estimate risk measures under limited computational budgets.
\end{abstract}




\section{Introduction}
Computational models described by partial differential equations (PDEs) have become indispensable tools for the analysis of physical systems. 
The PDE models are frequently used to make predictions about \textit{quantites of interest} (QoIs) of the systems to guide decision-making. 
However, the ubiquity of uncertainty in such systems, e.g. arising from uncertain model parameters, such as initial/boundary conditions, material coefficients, and sources, demands that predictions of the QoIs are equipped with a measure of their uncertainty.
This can be quantified using risk measures, which are statistical quantities that summarize the probability distribution of the QoIs subject to the uncertainty in the model parameters, and are typically defined in terms of expectations over the distribution. 
Risk measures play important roles in risk-informed decision-making as well as engineering applications such as risk-based design optimization, where the goal is to minimize the risk measure of a QoI with respect to the design variables, potentially subject to constraints on certain failure probabilities. 

In this work, we consider the estimation of risk measures given PDEs with uncertain parameters that are Gaussian random fields. 
As a single evaluation of the QoI requires solving the underlying PDE, sample-based estimates may be computationally expensive as it can require solving a large number of PDEs, especially for risk measures that emphasize tail events, such as the conditional value-at-risk (CVaR). 
Moreover, since the uncertain parameters are formally infinite dimensional (arbitrarily high dimensional upon discretization), 
deterministic quadrature methods can suffer from the curse of dimensionality.

Taylor approximations of the parameter-to-QoI map have commonly been used in uncertainty quantification (UQ), as in perturbation methods for computing moments, and the first-order and second-order reliability method (FORM/SORM) for computing failure probabilities \cite{KiureghianLinHwang87,Rackwitz01, MelchersBeck18}. 
By exploiting local derivative information, these methods are often able to produce accurate estimates of the risk measures at the cost of a single evaluation of the QoI and its derivatives. 
For this reason, they have been widely applied to risk/reliability-based design optimization, \cite{TuChoiPark99, ChiralaksanakulMahadevan04, YounChoi04}, and more recently, to PDE-modeled systems with random parameter fields \cite{AlexanderianPetraStadlerEtAl17, ChenVillaGhattas19, ChenGhattas21},
since the risk measures or failure probabilities typically need to be evaluated numerous times during the optimization process.
However, since Taylor approximations only are guaranteed accuracy locally near the point of expansion, the accuracy of the risk measure estimates may degrade for distributions with large variances or functions with high degrees of nonlinearity where the Taylor truncation errors dominates.

Other approaches for estimating risk measures subject to high-dimensional random parameters 
include sparse grid quadrature or collocation \cite{BungartzGriebel04, NobileTemponeWebster08a} and quasi-Monte Carlo methods \cite{GrahamKuoNuyensEtAl11}. 
For sufficiently smooth parameter-to-QoI maps, these have been shown to have favorable convergence properties compared to the standard $\mathcal{O}(N^{-1/2})$ convergence rate of Monte Carlo methods. 
However, these methods may still require a large number of function evaluations to reach the asymptotic regime. 
To reduce the computational complexity, \cite{CastrillonCandasNobileTempone21} combines sparse grids with Taylor approximations for computing expectations of QoIs of PDEs with random geometries, making use of sparse grids for large-variance directions of the input space and taking advantage of Taylor approximations for low-variance input directions. 
Alternatively, the efficiency of Monte Carlo methods can be improved through variance reduction strategies using lower-fidelity evaluations, as in multilevel and multifidelity Monte Carlo methods \cite{NgWillcox14, Giles15, PeherstorferWillcoxGunzburger16, AliUllmannHinze17, GeraciEldredIaccarino17}.

In this work, we develop and analyze the use of \textit{Gaussian mixture Taylor approximations} as a method of approximating the risk measures of PDEs with relatively small computational costs. 
Building on the Taylor approximation for its computational efficiency,
we seek to address the issues with larger variances. 
To this end, we consider approximating the underlying Gaussian measure by a Gaussian mixture with reduced variance along a dominant direction in parameter space. This procedure is studied in \cite{VittaldevRussell16}, and has also been used to improve the accuracy of UQ methods such as Kalman filters \cite{DeMarsBishopJah13}, polynomial chaos expansion \cite{VittaldevRussellLinares16}, 
and Taylor approximations \cite{FossaArmellinDelandeEtAl22}, 
for the propagation of uncertainty in orbital uncertain propagation for spacecrafts. 
Following the strategy in \cite{VittaldevRussell16}, the mixture approximation of the random field can be derived from a Gaussian mixture approximation of the standard Gaussian distribution in 1D. 
In our case, Taylor approximations are then constructed about the mean  of each mixture component with reduced variance.
Estimates of the risk measures are then computed by combining the contributions from the Taylor approximations at each mixture component.
The dominant direction can be selected as the dominant eigenvector of the covariance operator, which is the direction of largest variance. 
Alternatively, as we propose in this work, we can obtain this from the dominant eigenvector of covariance-preconditioned Hessian of the parameter-to-QoI map, which additionally accounts for the nonlinearity of the mapping. 

The method in our work is analogous to that of \cite{FossaArmellinDelandeEtAl22}, which considers the propagation of uncertainty for a Gaussian distribution subject to a nonlinear transformation. The authors also use a Gaussian mixture approximation of the distribution that is recursively created based on one-dimensional decompositions along dominant directions. Subsequently, linear Taylor approximations about the mixture component means are used to propagate the uncertainty to the outputs. In \cite{FossaArmellinDelandeEtAl22}, the decomposition directions are selected to maximize a nonlinearity index, which is computed based on the change in the Jacobian along each input direction. 
While \cite{FossaArmellinDelandeEtAl22} largely focuses on the propagation of uncertainty subject to low-dimensional random vectors for applications in orbital mechanics, 
we instead look to extend the formulation to formally infinite-dimensional parameter spaces.
Moreover, we focus on the computational challenges that arise in the setting of mappings governed by PDEs, and investigate the accuracy of our approximation scheme given a small number of function evaluations, due to the large computational costs of the PDE solves.

As such, the main contributions in this work are as follows. 
We derive expressions for the mean, standard deviation, and CVaR when using Gaussian mixtures with linear and quadratic Taylor approximations. While no simple analytic expression of the CVaR is available when using the quadratic approximation, we present a rapid sampling-based approach using a low-rank approximation of the Hessian operator.
We also analyze the approximation error of the Gaussian measure by the mixture when using the approach considered by \cite{VittaldevRussell16}. This analysis is of broader interest beyond our methodology, 
since the mixture decomposition strategy has been combined with several other UQ methods, as mentioned above.
We then present bounds for the combined approximation errors of the mean and CVaR risk measures in terms of both the mixture and Taylor approximations.
Finally, we investigate the performance of our method using two numerical examples; a semilinear advection-diffusion-reaction (ADR) equation with a random diffusion coefficient field and a scattering problem governed by the Helmholtz equation with a random wave speed field. In particular, our results show that the Gaussian mixture Taylor approximation is able to efficiently and effectively improve upon a single Taylor approximation for estimating the risk measures.

The remainder of the paper is organized as follows. Section~\ref{sec:forward_uq} presents the problem setting and reviews the use of Taylor approximation for UQ of PDEs with random parameter fields. In Section~\ref{sec:taylorgm}, we introduce the computation of the mean, variance, and CVaR using Taylor approximations with Gaussian mixtures. The mixture construction strategy of \cite{VittaldevRussell16} is then presented in Section~\ref{sec:gm} in the language of general Gaussian measures. Section~\ref{sec:analysis} provides an analysis of the mixture approximation error and the combined mixture Taylor approximation errors. Numerical results are then shown in Section~\ref{sec:numerical_results} before concluding with some remarks in Section~\ref{sec:conclusion}.

\section{Forward uncertainty quantification of PDEs}\label{sec:forward_uq}
In this work, we consider the problem of evaluating risk measures for systems governed by PDEs with uncertain parameter fields. We let $m$ denote an uncertain parameter belonging to a separable Hilbert space $\cM$ with inner product $\linner \cdot, \cdot \rinner_{\cM}$. We assume that $m$ has an associated distribution $\nu$, meaning $m$ is an $\cM$-valued random variable $m : (\Omega, \Sigma, \bP) \rightarrow (\cM, \cF, \nu)$. Here, $(\Omega, \Sigma, \mathbb{P})$ is the underlying probability space, $\cF = \cB(\cM)$ is the Borel sigma-algebra on $\cM$, and $\nu$ is therefore the pushforward measure of $\bP$ on $(\cM, \cB(\cM))$. 
In this paper, we focus on the case where $\nu = \cN(\bar{m}, \cC)$ is a Gaussian measure such that $m$ is a Gaussian random field over a bounded spatial domain $\cD \subset \bR^{d}$, where $d$ is the spatial dimension.
The Gaussian measure is defined by its mean $\bar{m} \in \cM = L^2(\cD)$ and a covariance operator $\cC \in \cL(\cM, \cM)$, a self-adjoint, positive trace-class operator, where $\cL(\cM, \cM)$ denotes the space of bounded linear operators from $\cM$ to itself.

As a prototypical example, we will consider the class of Gaussian random fields with 
covariance operators given by inverses of elliptic PDE operators, $\cC = \cA^{-\beta}$, 
where $\cA = \delta -\gamma \Delta$ is defined on the domain $H^2(\cD)$ with either homogeneous Neumann or Robin boundary conditions. 
For $\beta > d/2$, this is a well-defined covariance operator that gives rise to $\tau$-Holder continuous samples for any 
$\tau < \min\{1, \beta-d/2\}$. 
The spatial covariance structure is additionally determined by the parameters $\gamma, \delta > 0$, 
such that the random fields have pointwise variances of $\sigma_x = \sqrt{\gamma \delta}$ and
correlation lengths of $l_x = \sqrt{8(\beta-d/2)\gamma/\delta}$ (i.e.\ the correlation between the values at two points $x_1, x_2 \in \cD$ where $|x_1 - x_2| = l_x$ is approximately 0.1) 
This definition of the covariance operator coincides with that of the Whittle--Mat\'ern covariance model and the explicit link between the two is discussed in \cite{LindgrenRueLindstroem11}. 
In this work, we will focus on the case where $\beta = 2$, which corresponds to a widely used class of covariance operators,
but make the note that the methodologies presented are applicable to general covariance operators provided one can numerically apply the operator $\cC$ and its (numerical) inverse. 

The state variable $u$, belonging to a separable Hilbert space $\cU$, depends on the uncertain parameter field $m$ by a PDE, which we write as
\begin{equation} 
	\cR(u, m) = 0, 
\end{equation}
where $\cR : \cU \times \cM \rightarrow \cV'$ is a differential operator as a map from the product space $\cU \times \cM$ to the dual $\cV'$ of the test space $\cV$, another separable Hilbert space. 
In particular, we identify $\cV'$ with $\cV$ using the Riesz map.
We assume that the PDE is well-posed, admitting a solution operator $u = u(m)$. 
Computationally, the function spaces need to be discretized (e.g. by finite element bases), in which case $\cM$, $\cU$ and $\cV$ become finite but high-dimensional vector spaces, and $m \sim \cN(\bar{m}, \cC)$ is a high-dimensional Gaussian random vector. 

In many applications, specific quantities, such as pointwise observations and integral quantities, are often of primary interest. 
We let $q : \cM \times \cU \rightarrow \bR$ denote such a quantity of interest. Since the state depends on $m$ through the PDE, we can also consider the parameter-to-QoI map 
$Q : \cM \ni m \rightarrow q(u(m), m) \in \bR$, which is assumed to be $k$ times continuously Fr\'echet differentiable for some $k > 2$.
The QoI is thus a random variable distributed as the pushforward of $\nu$ under $Q$.
We are interested in quantifying the uncertainty of $Q$ through risk measures, which are essentially statistical quantities of the distribution of $Q$. Simple examples include 
the expectation $\bE_{\nu}[Q]$ 
and standard deviation $\var_{\nu}^{1/2}[Q] = \left(\bE_{\nu}[Q^2] - \bE_{\nu}[Q]^2\right)^{1/2}$. 
Additionally, we consider a more complex risk measure, the conditional value-at-risk (CVaR). The CVaR is of importance for optimization under uncertainty and reliability-based engineering design (\cite{RockafellarUryasev00, KouriSurowiec16, ChaudhuriKramerNortonEtAl21}) due to its favorable mathematical properties, including coherence as a risk measure (as in \cite{ArtznerDelbaenEberEtAl99}) and conservativeness (see \cite{ChaudhuriKramerNortonEtAl21}). 
The CVaR (or superquantile) is defined for a particular percentile $\alpha \in [0,1]$ as 
\begin{equation}\label{eq:cvar_min}
	\CVaR_{\alpha}[Q] = \min_{t \in \bR} t + \frac{1}{1-\alpha} \bE_{\nu}[(Q - t)^+].
\end{equation}
When $Q$ is continously distributed, this can be interpreted as the conditional expectation of $Q$ given that it exceeds its $\alpha$-value-at-risk ($\alpha$-VaR, or $\alpha$-quantile) \cite{ShapiroDentchevaRuszczynski09}. 
In particular, the $\alpha$-VaR is given by 
\begin{equation}
	\VaR_{\alpha}[Q] = \inf \{t \in \bR  : \bP[Q \leq t] \geq \alpha \} = F_Q^{-1}(\alpha),
\end{equation}
where $F_Q$ is the cumulative distribution function of $Q$.
Then, the $\alpha$-CVaR is
\begin{equation}\label{eq:cvar}
	\CVaR_{\alpha}[Q] = \VaR_{\alpha}[Q] + \frac{1}{1-\alpha} \bE_{\nu}[(Q - \VaR_{\alpha}[Q])^+],
\end{equation}
where $(\cdot)^+ = \max(\cdot, 0)$ is the maximum function. It follows from the definition that 
$\CVaR_{0}[Q] = \bE_{\nu}[Q]$ is the mean and
$\CVaR_{1}[Q] = \| Q \|_{L^{\infty}(\nu)}$ is the worst-case value,
while $\alpha \in (0,1)$ interpolates between the two extremes.

\subsection{Taylor approximations}
Given that the parameter-to-QoI map is sufficiently smooth, one can use Taylor approximations to compute moments of $Q$. 
In this section, we follow the presentation in \cite{ChenVillaGhattas19}. We adopt a Taylor approximation about the mean $\bar{m}$, which yields
\begin{equation}
	Q_{\lin}(m) = Q(\bar{m}) + \linner D Q(\bar{m}), m - \bar{m} \rinner_{\cM}
\end{equation}
in the linear case, and 
\begin{equation}
	Q_{\qua}(m) = Q(\bar{m}) + \linner D Q(\bar{m}), m - \bar{m} \rinner_{\cM} 
	+ \frac{1}{2} \linner D^2 Q(\bar{m}) (m - \bar{m}), m - \bar{m} \rinner_{\cM}
\end{equation}
in the quadratic case, where $D Q(\bar{m}) \in \cM$ and $D^2 Q(\bar{m}) \in \cL(\cM, \cM)$ are the gradient and Hessian operator of the parameter-to-QoI map evaluated at the mean $\bar{m} \in \cM$. 
When $m \sim \cN(\bar{m}, \cC)$ is Gaussian, we have analytic expressions for the means
\begin{align}
	\bE_{\nu}[Q_{\lin}] &= Q(\bar{m}), \label{eq:lin_exp}\\
	\bE_{\nu}[Q_{\qua}] &= Q(\bar{m})  + \frac{1}{2} \tr(\cH), \label{eq:quad_exp}
\end{align}
and variances 
\begin{align}
	\var_{\nu}[Q_{\lin}] &= \linner g, \cC g \rinner, \label{eq:lin_var}\\
	\var_{\nu}[Q_{\qua}] &= \linner g, \cC g \rinner + \frac{1}{2} \tr(\cH^2), \label{eq:_quad_var}
\end{align}
of the linear and quadratic approximations. Here, $g := DQ(\bar{m})$ and $\cH := \cC^{1/2} D^2 Q(\bar{m}) \cC^{1/2}$ denote the gradient and covariance-preconditioned Hessian at the mean $\bar{m}$. 
As discussed in \cite{ChenVillaGhattas19}, the trace can be approximated by first solving 
\begin{equation} \label{eq:hessian_ghep}
	D^2 Q(\bar{m}) \phi_j = \lambda_j \cC^{-1} \phi_j, \qquad j = 1, \dots, r_{\cH}
\end{equation}
for the eigenvalues $\lambda_j \in \bR$ (given in descending order, i.e. $\lambda_1 \geq \lambda_2 \geq \dots \geq \lambda_{r_{\cH}} \geq 0$)
and eigenvectors $\phi_j \in \cM$, which we note are orthonormal under the $\cC^{-1}$-inner product, i.e.\,$\langle \phi_i, \cC^{-1} \phi_j \rangle_{\cM} = \delta_{ij}$.
The traces are then estimated by a truncated sum over the first $r_{\cH}$ eigenvalues,
\[ \tr(\cH) \approx \sum_{j=1}^{r_{\cH}} \lambda_j, \quad \tr(\cH^2) \approx \sum_{j=1}^{r_{\cH}} \lambda_j^2. \]
Thus, the mean and variance of $Q$ can be estimated by linear and quadratic Taylor approximations without sampling $m \sim \nu$. 

When the covariance-preconditioned Hessian has a fast-decaying spectrum, 
the truncated sum is accurate even when using a low rank $r_{\cH}$, and can be efficiently computed using randomized methods with $\cO(r_{\cH})$ Hessian actions.
Specifically, we consider the randomized algorithm of \cite{SaibabaLeeKitanidis15} (Algorithm 6). This is a matrix-free approach, requiring only the action of the Hessian on $2(r_{\cH} + n_{\mathrm{OS}})$ randomly drawn vectors, where $n_{\mathrm{OS}}$ is an oversampling factor typically taken to be $\approx 20$.
The fast decay of eigenvalues have been numerically observed in Hessians of parameter-to-QoI maps for many problems of interest, and have even been mathematically proven in some. Examples include problems that arise in ice sheet flows \cite{IsaacPetraStadlerEtAl15a},
ocean dynamics \cite{KalmikovHeimbach14},
viscous and turbulent flows \cite{YangStadlerMoserEtAl11,ChenVillaGhattas19}, 
porous flows and poroelasticity \cite{AlexanderianPetraStadlerEtAl17a,ChenGhattas21,AlghamdiHesseChenEtAl21},
acoustic and electromagnetic scattering \cite{Bui-ThanhGhattas12a, Bui-ThanhGhattas12, Bui-ThanhGhattas13,ChenHabermanGhattas21}, 
and seismic wave propagation \cite{Bui-ThanhBursteddeGhattasEtAl12,MartinWilcoxBursteddeEtAl12,Bui-ThanhGhattasMartinEtAl13}.

The accuracy of the Taylor approximation for computing expectations is analyzed in \cite{AlexanderianPetraStadlerEtAl17}. When the parameter-to-QoI map has a globally bounded third-order derivative, one has 
\[ \bE_{\nu}[|Q - Q_{\qua}] \leq \sqrt{3} K \tr(\cC)^{3/2}, \]
where $K \geq 0$ is the bound on the third-order derivative such that 
\[
|D^3 Q(m) (\widetilde{m}_1, \widetilde{m}_2, \widetilde{m}_3)| \leq K \|\widetilde{m}_1\|_{\cM} \|\widetilde{m}_2\|_{\cM} \|\widetilde{m}_3\|_{\cM} 
\quad \forall m, \widetilde{m}_1, \widetilde{m}_2, \widetilde{m}_3 \in \cM,
\]
in which we are viewing the third-order derivative as a mapping from $m \in \cM$ to a bounded trilinear functional $D^3 Q(m)(\cdot, \cdot, \cdot)$.
This suggests that the quadratic approximation is accurate when $K$ is small (i.e.\ truncation error is small) and when $\tr(\cC)$ is small (i.e.\ variances are small). Conversely, one anticipates larger approximation errors of the expectation for highly nonlinear functions and distributions with large variances. This agrees with the well-known fact that Taylor approximations posses only local accuracy guarantees. The accuracy of the expectation therefore depends on the amount of probability mass in regions far from the expansion point where Taylor truncation errors are large.

\subsection{Derivative computation using an adjoint method}
The gradient $D Q(\bar{m})$ and Hessian action $D^2 Q(\bar{m}) \hat{m}$ in a direction $\hat{m} \in \cM$ can be efficiently computed via an adjoint method when $Q(m) = q(u(m),m)$ depends on the PDE $\cR(u,m) = 0$. Expressions for the derivatives can be derived using the Lagrangian formalism (see, for example, \cite{GhattasWillcox21}). We begin by defining the Lagrangian
\[ \cL(u,m,v) := q(u,m) + \linner \cR(u,m), v \rinner, \]
where $v \in \cV$ is the adjoint variable. For notational compactness, $\linner \cdot, \cdot \rinner$ (without the subscript) will be used to denote the inner product on the Hilbert space ($\cV$ or $\cM$) implied by the arguments. The gradient is then derived using partial derivatives of the Lagrangian to obtain the system $\partial_v \cL = 0$, $\partial_u \cL = 0$ and $DQ = \partial_m \cL$. 
This can be written out in weak form as
\begin{align}
	\label{eq:pde_state_eq}
	\linner \cR(u,\bar{m}), \tilde{v} \rinner &= 0 
	\qquad &\forall \tilde{v} \in \cV, \\
	\label{eq:pde_adjoint_eq}
	\linner \partial_u \cR(u,\bar{m})^* v, \tilde{u} \rinner &= - \linner \partial_u q(u,\bar{m}), \tilde{u} \rinner
	\qquad &\forall \tilde{u} \in \cV, \\
	\label{eq:pde_gradient_eq}
	\linner DQ(\bar{m}), \tilde{m} \rinner &= \linner \partial_{m} q(u,\bar{m}), \tilde{m} \rinner + \linner \partial_m \cR(u, \bar{m}) \tilde{m}, v \rinner 
	\qquad &\forall \tilde{m} \in \cM,
\end{align}	
where $(\cdot)^*$ denotes the adjoint of an operator.
Thus, in addition to the solving state PDE \eqref{eq:pde_state_eq}, the gradient $g$ can be obtained by solving the linear adjoint PDE \eqref{eq:pde_adjoint_eq} with operator $\partial_u \cR^* : \cV \rightarrow \cV$ for the adjoint variable $v$, and then evaluating the gradient form \eqref{eq:pde_gradient_eq}. Additionally, the Hessian action in the direction $\hat{m}$ can be derived by defining a Lagrangian,
\[ \cL_{H}(u,m,v,\hat{u},\hat{m},\hat{v}):= 
	\linner \partial_{u}\cL(u, m, v), \hat{u} \rinner
	+ \linner \partial_{v}\cL(u, m, v), \hat{v} \rinner
	+ \linner \partial_{m}\cL(u, m, v), \hat{m} \rinner,
\]
where $\hat{u}, \hat{v} \in \cV$ are the incremental state and adjoint variables, respectively. Given the solutions to the state and adjoint PDEs, $u$ and $v$, the Hessian action is then given by the system $\partial_{v} \cL_{H} = 0$, $\partial_{u} \cL_{H} = 0$, and $\partial_{mm}^{2} q \hat{m} = \partial_{m} \cL_{H}$ taken at $m = \bar{m}$. This has the following form,
\begin{align}
	\label{eq:pde_incr_state_eq}
	\linner \partial_u \cR \hat{u}, \tilde{v} \rinner &= - \linner \partial_m {\cR} \hat{m}, \tilde{v} \rinner, \\
	\label{eq:pde_incr_adjoint_eq}
	\linner \partial_u \cR^* v, \tilde{u} \rinner 
		&= - \linner \linner \partial^2_{uu} \cR \hat{u} + \partial^2_{mu} \cR \hat{m}, \tilde{u}  \rinner, v \rinner 
		- \linner \partial^2_{uu} q \hat{u} + \partial^2_{mu} q \hat{m}, \tilde{u} \rinner , \\
	\label{eq:pde_hessian_eq}
	\linner D^2{Q} \hat{m}, \tilde{m} \rinner 
		&= 
		\linner \partial_{m}\cR \tilde{m}, \hat{v} \rinner
		+ \linner \linner \partial^2_{um} \cR \hat{m} + \partial^2_{mm} \cR \hat{m}, \tilde{m}  \rinner, v \rinner 
		+ \linner \partial^2_{um} Q \hat{m} + \partial^2_{mm} Q \hat{m}, \tilde{m} \rinner ,
\end{align}	
for all $\tilde{u}, \tilde{v} \in \cV$ and $\tilde{m} \in \cM$, where we have implicitly assumed the forms are evaluated at $(u(\bar{m}),\bar{m})$ to further simplify the notation. 
Therefore, applying Hessian action requires the solution of two additional linear PDEs, namely the incremental state \eqref{eq:pde_incr_state_eq} and incremental adjoint \eqref{eq:pde_incr_adjoint_eq} equations for $\hat{u}$ and $\hat{v}$, respectively. These involve the linearized state operator $\partial_u \cR: \cV \rightarrow \cV$ and its adjoint $\partial_u \cR^*$. Moreover, the linear operators $\partial_u \cR$ and $\partial_u \cR^*$ remain unchanged when applying the Hessian to different directions, say, $\hat{m}_1$ and $\hat{m}_2$.
This allows one to amortize the costs of applying the Hessian by reusing factorizations or preconditioners constructed for the $\partial_u \cR$ and $\partial_u \cR^*$ operators across successive Hessian actions.

In Table~\ref{tab:pde_solves}, we summarize the number linear PDE solves required to (i) evaluate the QoI, (ii) construct a linear Taylor approximation, and (iii) construct a quadratic Taylor approximation using the randomized generalized eigenvalue solver of \cite{SaibabaLeeKitanidis15}, broken down into those involved in the state, adjoint, incremental state and incremental adjoint PDEs.
Moreover, we identify the total number of unique linear systems solved. 
When using a direct solver for these linear systems, the matrix factors for $\partial_u \cR$ or $\partial_u \cR^*$ can be reused for all of the adjoint and incremental PDEs, 
where asymptotically ($n \gg r_{\cH}$), the back-substitution step has negligible cost compared to the matrix factorization step.
In this case, constructing the quadratic Taylor approximation costs only one additional linear PDE solve compared to the linear Taylor approximation, which itself costs only one additional linear PDE solve compared to evaluating the QoI. 
\begin{table}[tbhp]
	\centering
	\small
	\caption{Linear PDE solves required to (i) evaluate the QoI, (ii) construct a linear Taylor approximation, 
	and (iii) construct a (low-rank) quadratic Taylor approximation via the randomized method (Algorithm 6 of \cite{SaibabaLeeKitanidis15}), assuming that the state PDE requires solving $n_L$ linear PDEs (e.g. Newton or Picard iterations). Here, $r_{\cH}$ denotes the rank used for the Hessian eigenvalue problem 
	and $n_{\mathrm{OS}}$ is the oversampling factor for the randomized eigenvalue solver (typically $n_{\mathrm{OS}} = \cO(10)$).}
	\label{tab:pde_solves}
	\begin{tabular}{| l | c | c | c | c | c | c | }
		\hline
		& State & Adj. & Incr. state. & Incr. adj. & Total & Total unique \\
		\hline
	Evaluate QoI  & $n_L$ & - & - & - & $n_L$ & $n_L$ \\ 
		\hline
	Lin. Taylor & $n_L$ & $1$ & - & - & $n_L + 1$ & $n_L + 1$ \\ 
		\hline
	Quad. Taylor & $n_L$ & $1$ & $2(r_{\cH} + n_{\mathrm{OS}})$ & $2(r_{\cH} + n_{\mathrm{OS}})$ & $n_L + 1 + 4(r_{\cH} + n_{\mathrm{OS}})$ & $n_L + 2$\\
		\hline
	\end{tabular}
\end{table}

\section{Gaussian mixture Taylor approximations}\label{sec:taylorgm}
We next propose an approach to improve the approximation accuracy of Taylor approximations for Gaussian distributions with larger variance. We first approximate the underlying Gaussian distribution $\nu$ by a Gaussian mixture $\nu_{\mix} = \sum_{i=1}^{N_{\mix}} w_i \nu_i$, where the components $\nu_i = \cN(\bar{m}_i, \cC_{i})$ are Gaussians with smaller variance. Under the Gaussian mixture approximation, we have
\[
\bE_{\nu}[Q]  \approx \bE_{\nu_{\mix}}[Q] = \sum_{i=1}^{N_{\mix}} w_i \bE_{\nu_i}[Q].
\]
We can construct individual Taylor approximations about the mixture component means, $\bar{m}_i$, 
\begin{align}
	Q_{\lin, i}(m) &= Q(\bar{m}_i) + \linner D Q(\bar{m}_i), m - \bar{m}_i \rinner_{\cM}, \\
	Q_{\qua, i}(m) &= Q(\bar{m}_i) + \linner D Q(\bar{m}_i), m - \bar{m}_i \rinner_{\cM} 
		+ \frac{1}{2} \linner D^2 Q(\bar{m}_i) (m - \bar{m}_i), m - \bar{m}_i \rinner_{\cM}.
\end{align}
We then use the Taylor approximations to compute the component expectations $\bE_{\nu_i}[Q]$. This yields the linear and quadratic approximations of the expectation
\begin{align}
	\widehat{Q}_{\lin, \gm} &:= \sum_{i=1}^{N_{\mix}} w_i \bE_{\nu_i}[Q_{\lin,i}] 
		= \sum_{i=1}^{N_{\mix}} w_i Q(\bar{m}_i), \\
	\widehat{Q}_{\qua, \gm} &:= \sum_{i=1}^{N_{\mix}} w_i \bE_{\nu_i}[Q_{\qua,i}] =
		\sum_{i=1}^{N_{\mix}} w_i ( Q(\bar{m}_i) + \frac{1}{2} \tr(\cH_i) ),
\end{align} 
where $\cH_i = \cC_i^{1/2} D^2 Q(\bar{m}_i) \cC_i^{1/2}$. Taylor expansions about the mixture means $\bar{m}_i$ are expected to be more accurate since $\nu_i$ have smaller variances. We will refer to this approach as the \textit{Gaussian mixture Taylor approximation}, or more simply, the \textit{mixture Taylor approximation}.

We can also derive a similar approximation for the variance, and by extension, the standard deviations. We approximate the expectation terms in the variance expression by
\[
\var_{\nu}[Q] = \bE_{\nu}[Q^2] - \bE_{\nu}[Q]^2 
\approx \bE_{\nu_{\mix}}[Q^2] - \bE_{\nu_{\mix}}[Q]^2 
= \sum_{i=1}^{N_{\mix}} w_i \bE_{\nu_i}[Q^2] - \left(\sum_{i=1}^{N_{\mix}} w_i \bE_{\nu_i}[Q] \right)^{\!\!\!2}.
\]
Subsequently applying Taylor approximations to compute $\bE_{\nu_i}[Q^2]$ and $\bE_{\nu_i}[Q]$ yields the variance approximations 
\begin{align}
\widehat{V}_{\lin,\gm} &= \sum_{i=1}^{N_{\mix}} w_i \bE_{\nu_i}[Q_{\lin,i}^2] - \widehat{Q}_{\lin, \gm}^2, \\
\widehat{V}_{\qua,\gm} &= \sum_{i=1}^{N_{\mix}} w_i \bE_{\nu_i}[Q_{\qua,i}^2] - \widehat{Q}_{\qua, \gm}^2,
\end{align}
for the linear and quadratic approximations, respectively.
Here, the second moment terms are given by 
\begin{align}
\bE_{\nu_i}[Q_{\lin,i}^2] &= Q(\bar{m}_i)^2 + \linner g_i, \cC_i g_i\rinner_{\cM}, \\
\bE_{\nu_i}[Q_{\qua,i}^2] &= (Q(\bar{m}_i) + \frac{1}{2}\tr(\cH_i))^2 
+ \linner g_i, \cC_i g_i \rinner_{\cM} + \frac{1}{2} \tr(\cH_i^2),
\end{align}
where we have used the notation $g_i := D Q(\bar{m}_i)$. Approximations for the standard deviations can then be obtained by taking the square root of the variance.

\subsection{Gaussian mixture Taylor approximations for CVaR}
\subsubsection{CVaR of Taylor approximations}
We also derive approximations for computing CVaR using Taylor approximations and Gaussian mixtures. We begin with the linear case, for which there is a closed-form expression of the CVaR when $\nu$ is Gaussian. Due to the Gaussianity of $m$, the linear Taylor approximation $Q_{\lin}(m)$ is also Gaussian with mean 
$\mu_{Q_{\lin}} = \bE_{\nu}[Q_{\lin}]$ 
and standard deviation 
$\sigma_{\lin} = \var_{\nu}[Q_{\lin}]^{1/2}$ 
given by \eqref{eq:lin_exp} and \eqref{eq:lin_var}. 
Using the CVaR for a Gaussian random variable (e.g. \cite{RockafellarUryasev00}), we have the CVaR for the linear approximation,
\begin{equation}
(\widehat{C}_{\alpha})_{\lin} := \mu_{Q_{\lin}} + \frac{1}{1-\alpha} \frac{\sigma_{Q_{\lin}}}{\sqrt{2\pi}}
	\exp \left( - \frac{(\VaR_{\alpha} - \mu_{Q_{\lin}})^2}{2 \sigma_{Q_{\lin}}^2} \right),
\end{equation}
where the $\alpha$-quantile $\VaR_{\alpha}[Q_{\lin}]$ can be found by the inverse CDF of the normally distributed random variable $\cN(\mu_{Q_{\lin}}, \sigma_{Q_{\lin}}^2)$.

Unlike the linear Taylor approximations, we do not have an analytic representation for 
the quadratic case. Instead, we can evaluate the CVaR using the quadratic approximation by a sample average approximation. 
This is the approach taken in \cite{ChenGhattas21}, where the authors use a low-rank approximation of the Hessian
to rapidly sample the quadratic approximation. We write the low-rank quadratic approximation as
\begin{equation}
	Q_{\qua, \lr}(m) = Q(\bar{m}) + \linner g, m - \bar{m} \rinner_{\cM}
		+ \frac{1}{2} \sum_{j=1}^{r_{\cH}} \lambda_j \linner \cC^{-1} \phi_j, m - \bar{m} \rinner_{\cM}^2,
\end{equation}
where $\lambda_j, \phi_j$ are the generalized eigenpairs from \eqref{eq:hessian_ghep}. The expectation and variance of this approximation match those computed using the analytic expressions \eqref{eq:quad_exp} and \eqref{eq:_quad_var} for the quadratic Taylor approximation when the trace is computed by the truncated sum to rank $r_{\cH}$.

For the Whittle--Mat\'ern Gaussian random fields under consideration (with $\beta = 2)$, samples for $m$ are drawn by solving the stochastic PDE associated with the elliptic operator, $\cA(m - \bar{m}) = \cW$, where $\cW$ is Gaussian white noise (see Definition 6, \cite{LindgrenRueLindstroem11}). 
Although the PDE involving $\cA$ can be efficiently solved using fast elliptic solvers, the operations still scale with the discretization dimension of $m$. Here, we present an new sampling scheme for the low-rank quadratic approximation 
that completely avoids the PDE solve. 
Specifically, $Q_{\qua, \lr}$ is distributionally equivalent to
\begin{equation}\label{eq:quadratic_sample}
Q_{\qua, \lr} \equaldist Q(\bar{m}) 
+ \left(\linner g, \cC g \rinner_{\cM} - \sum_{j=1}^{r_{\cH}} \linner g, \phi_j \rinner_{\cM} \right)^{\!\!1/2} \!\!\! y_0
+ \sum_{j=1}^{r_{\cH}} 
	\left(\linner g , \phi_j \rinner_{\cM} y_j + \lambda_j y_j^2 \right), 
\end{equation}
where $y_0, \dots, y_r \sim \cN(0,1)$ are i.i.d. standard Gaussian random variables (see Appendix \ref{sec:appendix_sampling}). Here, we only need to draw samples 
of $y = [y_0, y_1, \dots, y_{r_{\cH}}] \in \bR^{r_{\cH}}$,
with the inner products computed a single time after solving the generalized eigenvalue problems. The cost of sampling the quadratic approximation by \eqref{eq:quadratic_sample} is negligible compared to the cost of solving the state PDE.

We can then use a standard sampling-based estimator to compute the CVaR by the quadratic approximation. For example, we can use the estimator from \cite{HongHuLiu14},
\begin{equation}
	(\widehat{C}_{\alpha})_{\qua} := \min_{t \in \bR} t + \frac{1}{M} \sum_{j=1}^{M} (Q_{\qua}^{(j)} - t)^+, 
\end{equation}
where $Q_{\qua}^{(j)}$, $j = 1,\dots, M$, are i.i.d. samples drawn from the low-rank quadratic approximation. The minimization over $t$ is one-dimensional and can be performed using a standard search algorithm without additional PDE solves. As sampling is inexpensive, we can use a large number of samples to compute a quadratic CVaR estimate to make the sampling error negligible.

\subsubsection{With Gaussian mixtures}
We can also obtain CVaR estimators using the Gaussian mixture approximation $\nu_{\mix} = \sum w_i \nu_i$. That is, we replace the expectation in \eqref{eq:cvar_min} by the expectation with respect to the Gaussian mixture, 
\[ 
\CVaR_{\alpha}[Q] \approx
\min_{t \in \bR} t + \frac{1}{1-\alpha}\sum_{i=1}^{N_{\mix}} w_i  \bE_{\nu_i}[(Q - t)^{+}].
\]
Taking Taylor approximations at each component mean then yields the Gaussian mixture Taylor CVaR estimators. In the linear case, this becomes
\begin{equation}\label{eq:cvar_lin_mix_min}
 (\widehat{C}_{\alpha})_{\lin, \gm} := \min_{t \in \bR} t 
 	+ \frac{1}{1 - \alpha} \sum_{i=1}^{N_{\mix}} w_i \bE_{\nu_i}[(Q_{\lin,i} - t)^+].
\end{equation}
For each mixture component, the corresponding linear approximation $Q_{\lin,i}$ is normally distributed with mean 
$\mu_{Q_{\lin,i}} = Q(\bar{m}_i)$ and variance $\sigma_{Q_{\lin,i}}^2 = \linner g_i, \cC_i g_i \rinner$. 
Therefore, our combined mixture Taylor approximation is equivalent to approximating the distribution of $Q$ by the Gaussian mixture 
\[
Q_{\mix,\lin} \sim \sum_{i=1}^{N_{\mix}} w_i \cN(\mu_{Q_{\lin,i}}, \sigma_{Q_{\lin,i}}^2).
\]
The minimizer $t^*$ of \eqref{eq:cvar_lin_mix_min}, which is the VaR of $ Q_{\mix,\lin}$, 
satisfies 
\[
\sum_{i=1}^{N_{\mix}} w_i \bP[Q_{\lin,i} > t^*] = 1 - \alpha, 
\]
where $\bP[Q_{\lin,i} > t^*]$ can be expressed using the CDF of $\cN(\mu_{Q_{\lin,i}}, \sigma_{Q_{\lin,i}}^2)$. Thus, $t^*$ can be obtained from a simple one-dimensional search. The overall CVaR estimate \eqref{eq:cvar_lin_mix_min} is then given by  
\begin{equation}\label{eq:cvar_lin_mix}
(\widehat{C}_{\alpha})_{\lin, \gm} = t^* 
+ \frac{1}{1-\alpha}\sum_{i=1}^{N_\mix}w_i \left( 
\frac{\sigma_{Q_{\lin,i}}}{\sqrt{2\pi}} 
	\exp\left(-\frac{(t^* - \mu_{Q_{\lin,i}})^2}{2\sigma_{Q_{\lin,i}}^2}\right)
		+ (\mu_{Q_{\lin,i}} - t^*) \bP[Q_{\lin,i} > t^*]
\right),
\end{equation}
where we have utilized the fact that for a normally distributed random variable $X \sim \cN(\mu_X, \sigma_X^2)$, 
\[
	\bE[(X - \gamma)^+] = \frac{\sigma_X}{\sqrt{2\pi}} 
		\exp\left(-\frac{(\gamma - \mu_X)^2}{2\sigma_X^2}\right)
		+ (\mu_X - \gamma) \bP[X > \gamma].
\]
This latter expression can be obtained by a direct evaluation of the integral.

Similarly, we can use quadratic approximations for the mixture components,
\[
\CVaR_{\alpha}[Q] \approx \min_{t \in \bR} t 
	+ \frac{1}{1-\alpha} \sum_{i=1}^{N_{\mix}} w_i \bE_{\nu_i}[(Q_{\qua,i} - t)^+].
\]
Since the quadratic approximation does not admit elegant, closed-form expressions, we can instead use the sampling based estimate
\begin{equation}\label{eq:cvar_quaDix}
(\widehat{C}_{\alpha})_{\qua,\mathrm{GM}}
= \min_{t \in \bR} t 
	+ \frac{1}{1-\alpha} \sum_{i=1}^{N_{\mix}} w_i \; \frac{1}{M_i} \sum_{j=1}^{M_i} (Q_{\qua,i}^{(j)} - t)^+,
\end{equation}
where $Q_{\qua,i}^{(j)}$, $j=1,\dots,M_i$, are i.i.d. samples drawn from the low-rank quadratic approximation for component $\nu_i$. As previously noted, sampling of the low-rank quadratic approximation has negligible computational cost, so we can afford to use large sample sizes to make the sampling error negligible.

\section{Constructing the Gaussian mixture approximation}\label{sec:gm}
\subsection{Approximation of a one-dimensional Gaussian}
Following \cite{VittaldevRussell16}, we now describe a strategy for constructing the mixture approximation 
$\nu_{\mix} = \sum_{i=1}^{N_{\mix}} w_i \cN(\bar{m}_i, \cC_i)$ 
based on a one-dimensional decomposition along a dominant direction. 
This approach is originally presented for finite dimensional Gaussians in $\bR^n$. 
In this section, we present it for generic Gaussian measures to include the case of Gaussian random fields. 
In particular, we derive the expressions for the component means $\bar{m}_i$, covariances $\cC_i$, and its inverse $\cC_i^{-1}$, 
which are used in the Taylor approximation-based risk measure estimates.

In \cite{VittaldevRussell16}, the authors approximate standard normal distribution $\nu_0 = \cN(0, 1)$ with density $\pi_0$, 
by Gaussian mixture $\nu_{0, \mix} = \sum_{i=1}^{N_{\mix}} w_i \cN(\mu_i, \sigma_i^2)$ with density $\pi_{0,\mix}$ by minimizing the $L^q(\bR)$ misfit of the densities for some $q \geq 1$. 
To simplify the minimization problem, authors define a fixed standard deviation by the rule $\sigma_i = N_{\mix}^{-p}$ with predefined $p \in (0,1)$. Moreover, the mixture is assumed to be symmetric about $x=0$. Mixture approximations are then obtained by solving 
\begin{align}
	\min_{w_i, \mu_i} \int_{\bR} | \pi_0(x) - \pi_{0, \mix}(x) |^q dx, 
	\quad \text{s.t. } \sum_{i=1}^{N_{\mix}} w_i = 1
\end{align}
for different mixture sizes $N_{\mix}$. 
In practice $q=2$ is typically used to simplify the optimization problem.
In our work, we directly make use of the results for Gaussian mixtures up to 39 components computed in the univariate splitting library of \cite{VittaldevRussell16}. 
Figure~\ref{fig:univariate_splitting} shows some examples of the mixture approximations 
of \cite{VittaldevRussell16} for the standard normal distribution in 1D.

\begin{figure}[htpb!]
	\centering
	\includegraphics[width=0.3\textwidth]{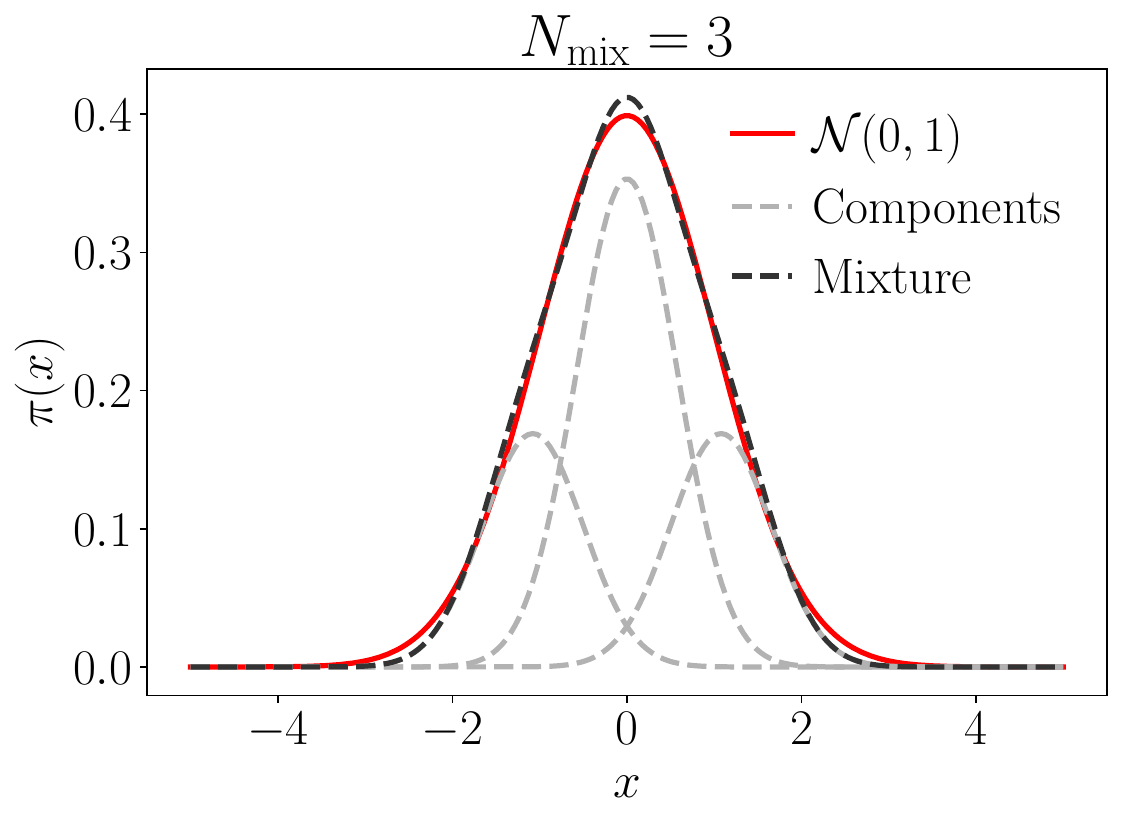}
	\includegraphics[width=0.3\textwidth]{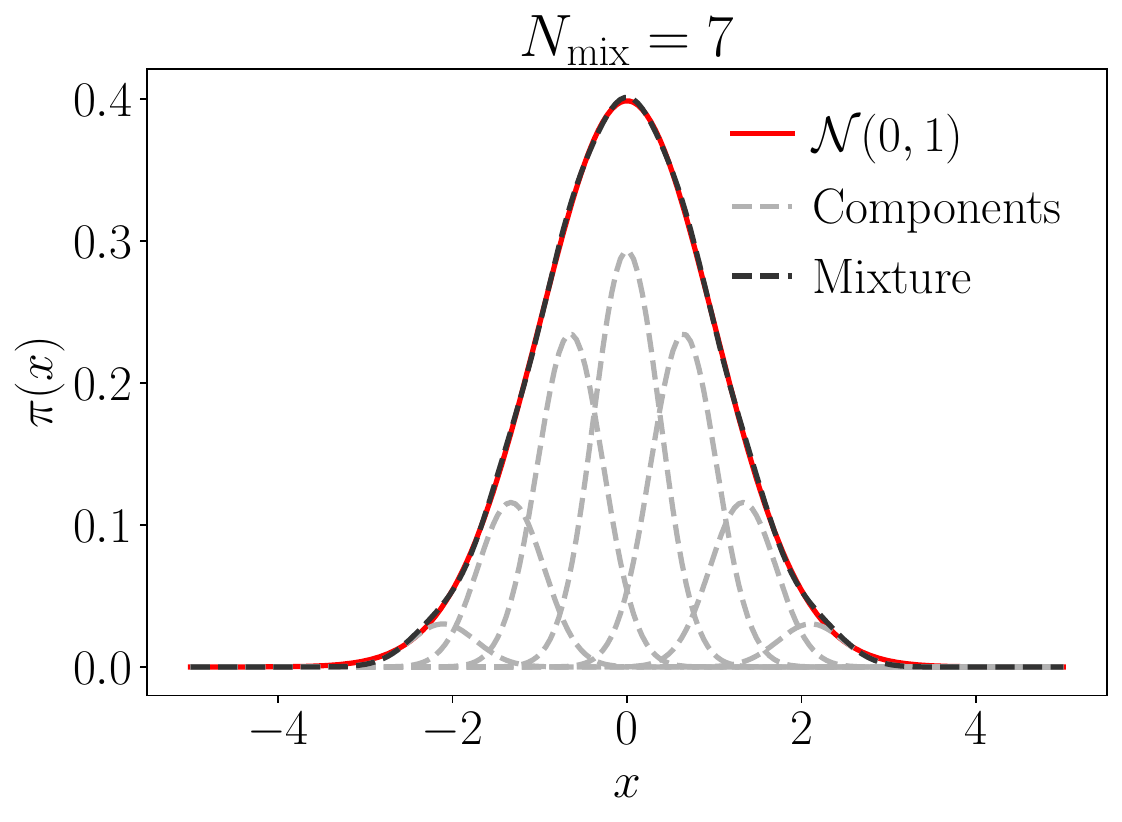}
	\includegraphics[width=0.3\textwidth]{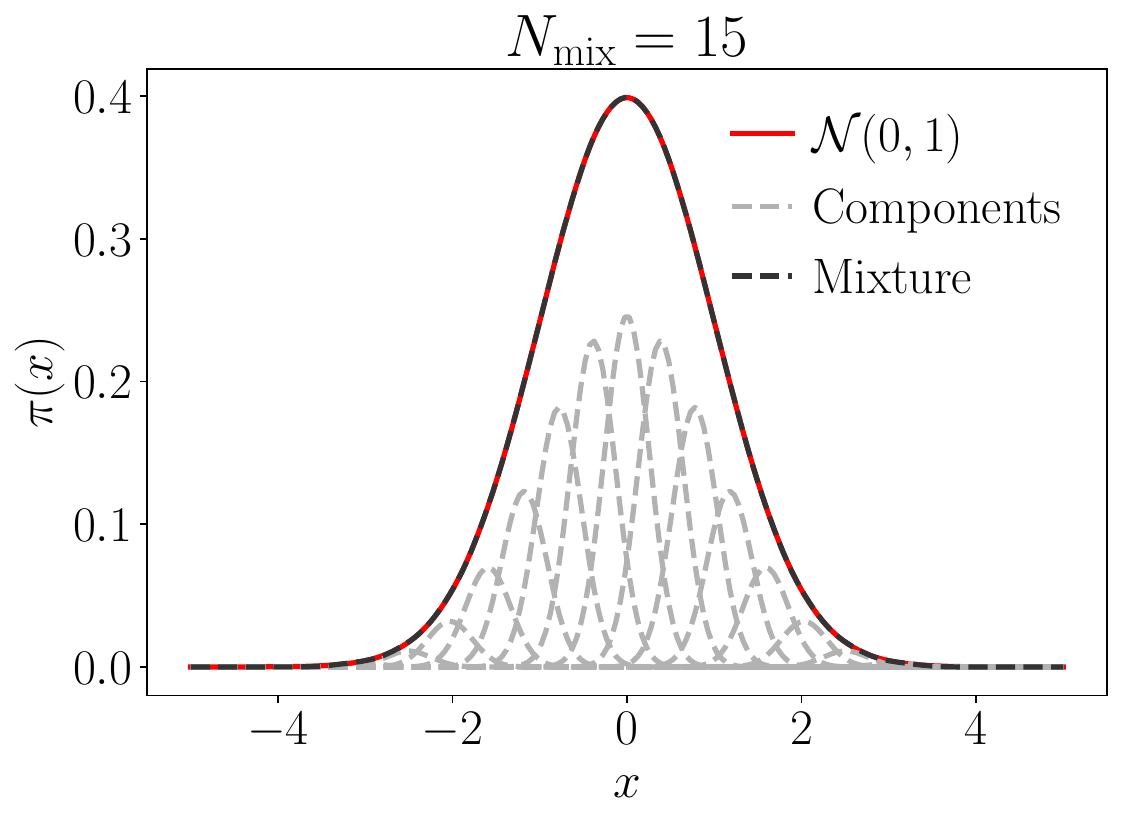}
	\caption{Probability density functions of Gaussian mixture approximations of the standard normal distribution using 
	$N_{\mix} = 3, 5, 7$ mixture components obtained from the univariate splitting library of \cite{VittaldevRussell16}. The mixture components have standard deviations $\sigma_i = N_{\mix}^{-p}$ with $p = 1/2$.}
	\label{fig:univariate_splitting}
\end{figure}

\subsection{Decomposition along a Karhunen--Lo\`eve direction}
As discussed in \cite{VittaldevRussell16}, the 1D mixture approximation 
$\nu_{0,\mix} = \sum_{i=1}^{N_{\mix}} w_i \cN(\mu_i, \sigma_i^2)$ of the standard Gaussian $\nu_{0} = \cN(0, 1)$ 
can be used to construct a Gaussian mixture for general Gaussian measures $\nu = \cN(\bar{m}, \cC)$. To this end, we consider the Karhunen--Lo\`eve (KL) expansion of $m \sim \nu$,
\[
m(\omega) = \bar{m} 
+ \sum_{j=1}^{\infty} \xi_j(\omega) \sqrt{\lambda_j} \varphi_j
\]
where $(\lambda_j, \varphi_j)_{j=1}^{\infty}$ are eigenvalues and eigenvectors of the covariance operator with $\lambda_1 \geq \lambda_2 \geq \dots$, and $\xi_j \sim \cN(0,1)$ are i.i.d. To construct the mixture approximation, we select a particular KLE mode $\varphi_k$, $k \in \bN$. We approximate its corresponding stochastic component, $\xi_k \sim \cN(0,1)$, by the 1D mixture approximation $\sum_{i=1}^{N_{\mix}} w_i \cN(\mu_i, \sigma_i^2)$. This leads to mixture components
\[
m_i(\omega) = \bar{m} + (\mu_i + \xi_k^i(\omega) \sigma_i) \sqrt{\lambda_k} \varphi_k
+ \sum_{j \neq k}^{\infty} \xi_j(\omega) \sqrt{\lambda_j} \varphi_j,
\]
where $\xi_k^i \sim \cN(0,1)$ independent to $\xi_j$, $j \neq k$. We observe that each component corresponds to the KLE of a Gaussian $\cN(\bar{m}_i, \cC_i)$ with mean
and covariance 
\begin{align}
	\label{eq:kl_mix_mean}
	\bar{m}_i &= \bar{m} + \mu_i \sqrt{\lambda_k} \varphi_k,\\
	\label{eq:kl_mix_cov}
	\cC_i &= \cC(\cI_{\cM} - \varphi_k \otimes \varphi_k) + \sigma_i^2 \lambda_k \varphi_k \otimes \varphi_k
	= \cC + (\sigma_i^2 - 1)\lambda_k \varphi_k \otimes \varphi_k,
\end{align}
where $\cI_{\cM}$ is the identity operator on $\cM$ and $\mu_i$ and $\sigma_i$ come from the 1D mixture approximation.
Thus, the Gaussian mixture approximation generated by a 1D mixture approximation along a KLE mode $\varphi_k$ is given by $\nu_{\mix} = \sum_{i=1}^{N_{\mix}} w_i \cN(\bar{m}_i, \cC_i)$, 
where the mean and covariance are given by \eqref{eq:kl_mix_mean} and \eqref{eq:kl_mix_cov}, and $w_i$ comes from the 1D approximation. Moreover, we can formally write the inverse of the component covariance operators as 
\begin{equation}
	\cC_i^{-1} 
	= \cC^{-1}(\cI_{\cM} - \varphi_k \otimes \varphi_k) + \frac{1}{\sigma_i^2 \lambda_k} \varphi_k \otimes \varphi_k
	= \cC^{-1} + \left(\frac{1}{\sigma_i^2} - 1 \right) \frac{1}{\lambda_k} \varphi_k \otimes \varphi_k.
\end{equation}

Such mixture components are spaced in alignment with the chosen direction $\varphi_k$. 
Moreover, the decomposition produces Gaussian mixture components that have their variance along $\varphi_k$ reduced by a factor of $\sigma_i^2$ compared to the original Gaussian, while the variances along all remaining KLE directions remain unchanged. Therefore, it makes sense to decompose along the dominant variance direction of $\cC$, i.e. direction with largest $\lambda_k$. This is very effective if the covariance exhibits rapid eigenvalue decay such that it has a notable high-variance direction. This is the case for Gaussian random fields with large correlation lengths.

\subsection{Decomposition along arbitrary directions}
The 1D decomposition can also be made along arbitrary directions in $\range(\cC)$. 
In \cite{VittaldevRussell16}, this is done by first performing the mixture decomposition in the white noise distribution, and then transforming the mixture components back to the original distribution
using the Cholesky or spectral decomposition of the covariance matrix.
Here, we will follow the same approach, but will avoid explicitly computing with $\cC^{1/2}$ 
via a decomposition due to the high-dimensionality of $\cM$. 
To this end, we start by recalling that  
$ m - \bar{m} = \cC^{1/2} \cW$,
where $\cW$ is again Gaussian white noise.
Formally, we can consider $\cW \sim \cN(0, \cI_{\cM})$ to be a centered Gaussian with the identity operator as the covariance. 
For a direction $\varphi \in \cM$ with $\|\varphi\|_{\cM} = 1$, we can formally perform the KLE-based decomposition for the distribution of $\cW$ along $\varphi$, which is a KLE mode of $\cI_{\cM}$ with unit variance. Based on \eqref{eq:kl_mix_mean} and \eqref{eq:kl_mix_cov}, we can write the mixture components of white noise as 
\[ 
\cW_i \sim \cN(\mu_i \varphi, \cI_{\cM} + (\sigma_i^2 - 1) \varphi \otimes \varphi).
\]
This leads to corresponding mixture components,
$m_i = \bar{m} + \cC^{1/2} \cW_i$, 
suggesting that the mixture components have distribution $m_i \sim \cN(\bar{m}_i, \cC_{i})$, with component mean and covariance
\begin{align}\label{eq:arb_mix_mean_phi}
	\bar{m}_i &= \bar{m} + \mu_i \cC^{1/2} \varphi, \\
\label{eq:arb_mix_cov_phi}
	\cC_i &= \cC^{1/2}(\cI_{\cM} + (\sigma_i^2 - 1)\varphi \otimes \varphi) \cC^{1/2}.
\end{align}

In this case, the mixture components are spaced along the direction $\cC^{1/2} \varphi$. 
If the desired alignment direction is specified as $\psi \in \range(\cC)$, then we can define
$ \varphi = {\cC^{-1/2} \psi}/{\|\cC^{-1/2} \psi\|_{\cM}}$
by a normalization.
Furthermore, we can introduce the variable
\begin{equation}\label{eq:pseudo_eigenvalue}
	\lambda_{\psi} = \|\cC^{-1/2} \psi\|_{\cM}^{-2}
	= \linner \psi, \cC^{-1} \psi \rinner_{\cM}^{-1},
\end{equation}
so that $\cC^{1/2} \varphi = \sqrt{\lambda_{\psi}} \psi$. 
The component mean and covariance can instead be expressed in terms of $\lambda_{\psi}$ and $\psi$ as
\begin{align}
	\label{eq:arb_mix_mean_psi}
	\bar{m}_i &= \bar{m} + \mu_i \sqrt{\lambda_{\psi}} \psi, \\
	\label{eq:arb_mix_cov_psi}
	\cC_i &= \cC + (\sigma_i^2 - 1) \lambda_{\psi} \psi \otimes \psi.
\end{align}
Moreover, the inverse of the component covariance is
\begin{equation}
	\cC_i^{-1} = \cC^{-1} 
	+ \left( \frac{1}{\sigma_i^2} - 1 \right) \lambda_{\psi} \cC^{-1} \psi \otimes \cC^{-1} \psi.
\end{equation}
The corresponding Gaussian mixture is then given by $\nu_{\mix} = \sum_{i=1}^{N_{\mix}} w_i \cN(\bar{m}_i, \cC_i)$.
Note that the final expressions for the component means and covariances do not involve $\cC^{1/2}$.

While the derivation for the mixture is largely formal in the infinite-dimensional setting, we can verify that the resulting mixture components $\nu_i = \cN(\bar{m}_i, \cC_i)$ are well-defined Gaussian measures. 
By inspection of \eqref{eq:arb_mix_cov_psi}, we see that $\cC_i$ is indeed a self-adjoint, positive operator, and its trace is 
\begin{equation} 
\tr(\cC_i) = \tr(\cC) + (\sigma_i^2 - 1) \frac{\linner \psi, \psi \rinner_{\cM}}{\linner \psi, \cC^{-1} \psi \rinner_{\cM}} \leq \tr(\cC).
\end{equation}
Thus, if $\cC$ is of trace-class, then so is $\cC_i$. Moreover, the reduction in variance is given by an inverse Rayleigh quotient of $\cC^{-1}$, 
where the variance reduction is maximized when $\psi$ is the dominant KLE direction $\varphi_1$. In fact, whenever $\psi = \varphi_k$ is a KLE direction, we have $\lambda_{\psi} = \lambda_k$, and the expressions coincide with the KLE case. 

Although the reduction in variance is not necessarily maximized when selecting $\psi$ that are not KLE directions, there are potential advantages for doing so. As we have noted, the error of the Taylor approximation also depends on the nonlinearity of the map $Q$. Thus, choosing $\psi$ based on directions with strong nonlinearity in addition to large variance can be beneficial. To this end, we consider $\psi$ obtained from solving the generalized Hessian eigenvalue problem \eqref{eq:hessian_ghep}. That is, we take the dominant eigenvector $\phi_1$ obtained from 
\begin{equation}\label{eq:dominant_ghep}
	D^2 Q(\bar{m}) \phi_1 = \lambda_1 \cC^{-1} \phi_1 
\end{equation}
as the decomposition direction $\psi$. Intuitively, this represents a balance between the nonlinearity, which informed by the Hessian, and variance, which is informed by the covariance operator. We refer to this as the Hessian eigenvalue problem (HEP)-based decomposition, in contrast to the KLE-based decomposition. 

We remark that problem \eqref{eq:dominant_ghep} can again be computed using iterative or randomized methods in a matrix-free manner. 
Moreover, in the 1D mixture approximations considered in this work, the mean $\bar{m}$ of the original Gaussian $\mathcal{N}(\bar{m}, \cC)$ will also be the mean of a subsequent mixture component $\mathcal{N}(\bar{m}_i, \cC_i)$ (i.e. $\bar{m}_i = \bar{m}$). 
Thus, state and adjoint solutions at $\bar{m}$ can be re-used for the risk measure computations at the mixture component 
such that no additional unique linear solves are needed to construct the HEP-based expansion when compared to the KLE-based expansion.

\subsection{Decomposition along multiple directions}
So far, the presented mixture approximations involve decomposing along a single dominant direction in parameter space. If multiple dominant directions are present, we also can simultaneously decompose each of the dominant directions. 

We begin with the case where all the decomposition 
directions are KLE modes, letting $\{\varphi_k\}_{k=1}^{r_K}$ denote the $r_K$ dominant KLE modes. 
This can be done as a tensor product of 1D mixture decompositions.
That is, for each of the directions $\varphi_k$, we use a 1D mixture approximation with $\{w_{i_k}, \mu_{i_k}, \sigma_{i_k} \}_{i_k = 1}^{N_k}$ as the means, variances, and weights, where $N_k$ is the number of 1D components in the $k^{th}$ direction.
We introduce a multi-index $\bs{i} = (i_1, \dots, i_{r_K})$, where each component $i_k$ of the multi-index $\bs{i}$ denotes the 1D mixture index along direction $k$. This leads to a total of 
$N_{\mix} = \prod_{k=1}^{r_K} N_k$ 
Gaussian mixture components, 
$\nu_{\bs{i}} = \cN(\bar{m}_{\bs{i}}, \cC_{\bs{i}})$, which have weights, means, and covariances
\begin{align}
\label{eq:kl_mix_multi_weight}
w_{\bs{i}} &= \prod_{k=1}^{r_K} w_{i_{k}}, \\
\label{eq:kl_mix_multi_mean}
\bar{m}_{\bs{i}} &= \bar{m} + \sum_{k=1}^{r_K} \mu_{i_k} \sqrt{\lambda_k} \varphi_k, \\
\label{eq:kl_mix_multi_cov}
\cC_{\bs{i}} &= \cC + \sum_{k=1}^{r_K} (\sigma_{i_k}^2 - 1) \lambda_k \varphi_k \otimes \varphi_k.
\end{align}

In this case, for each component, the variance along the direction $\varphi_k$ is reduced by a factor of $\sigma_{i_k}^2$ for each of the $k = 1, \dots, r_K$ decomposition directions. However, due to the tensor product nature of the approximation, the number of mixture components $N_{\mix}$ increases rapidly with the number of decomposition directions $r$. Thus, this approach is best restricted to only the first few KLE modes.

As shown in \cite{VittaldevRussell16}, mixture approximations can also be constructed recursively along arbitrary directions. That is, we can first decompose $\nu = \cN(\bar{m}, \cC)$ into a mixture 
\[ \nu_{\mix}^{(1)} = \sum_{i=1}^{N_{\mix}^{(1)}} w_i^{(1)} \nu_i^{(1)} = \sum_{i=1}^{N_{\mix}^{(1)}}w_i^1 \cN(\bar{m}_i^{(1)}, \cC_i^{(1)})
\]
along some direction $\psi^{(1)}$. 
Subsequently, components in the mixture, $\nu_i^{(1)} = \cN(\bar{m}_i^{(1)}, \cC_i^{(1)})$, can also be further approximated by
\[
\nu_i^{(1)} = \cN(\bar{m}_i^{(1)}, \cC_i^{(1)}) \approx \sum_{j=1}^{N_{\mix,i}^{(2)}} 
	w_{i,j}^{(2)} \nu_{i,j}^{(2)} = \sum_{j=1}^{N_{\mix},i} w_{i,j}^{(2)} \nu_{i,j}^{(2)},
\]
each along a second mixture direction $\psi^{(2)}_i$. 
Here, the superscript indicates the generation of the mixture, i.e., the number of decomposition directions used to produce it. This process can be repeated to a desired mixture size (number of components). 
Figure~\ref{fig:splitting} illustrates the different mixture construction strategies, showing the PDFs of a 2D Gaussian and those of its mixture approximations obtained via decomposing along a KLE direction, decomposing along an arbitrary direction, and a recursively constructed mixture where a mixture component is further decomposed along a second direction.
We note that for practical purposes, we consider only decompositions along a single dominant direction, 
since the number of mixture components, 
and hence the number of PDE solves involved in the mixture Taylor approximation, can scale prohibitively with the number of directions (exponentially in the full tensor product case). 
However, we do discuss the theoretical mixture approximation errors for the general case in Section~\ref{sec:analysis}.

\begin{figure}[htpb!]
	\centering
	\setlength{\tabcolsep}{1pt}
	\begin{tabular}{c c}
	\multirow{2}{*}[0.8in]{
	\includegraphics[width=0.22\textwidth]{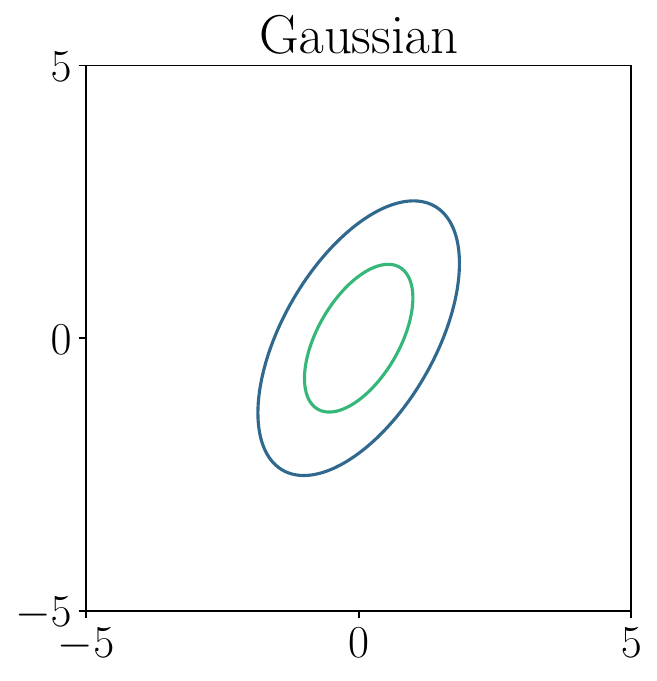}
	}
	& 
		\begin{tabular}{c c c}
		\includegraphics[width=0.22\textwidth]{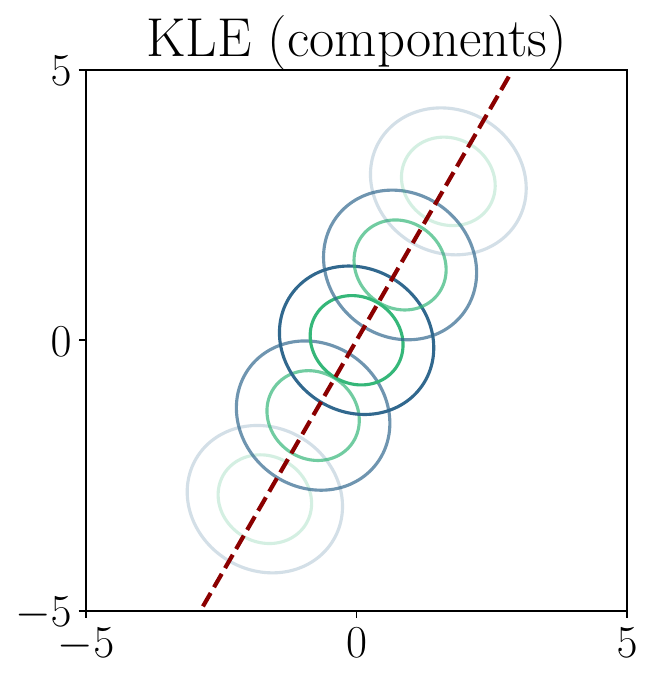}
		& 
		\includegraphics[width=0.22\textwidth]{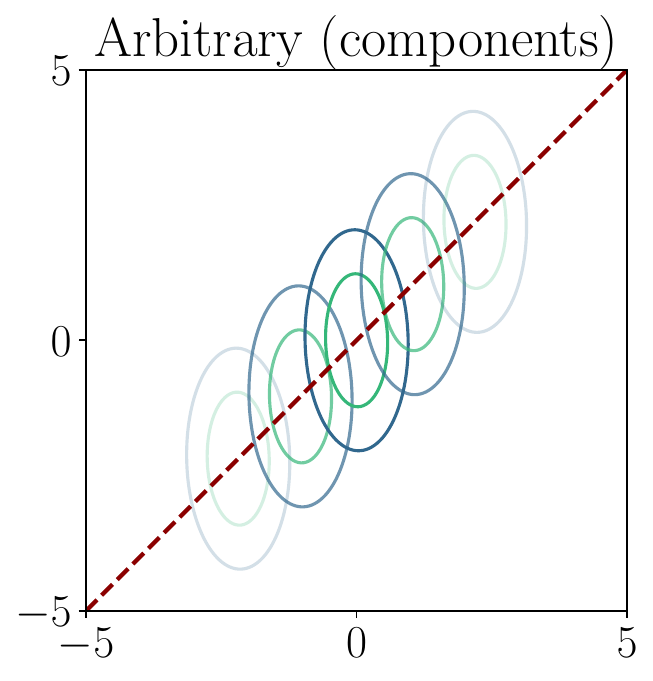}
		& 
		\includegraphics[width=0.22\textwidth]{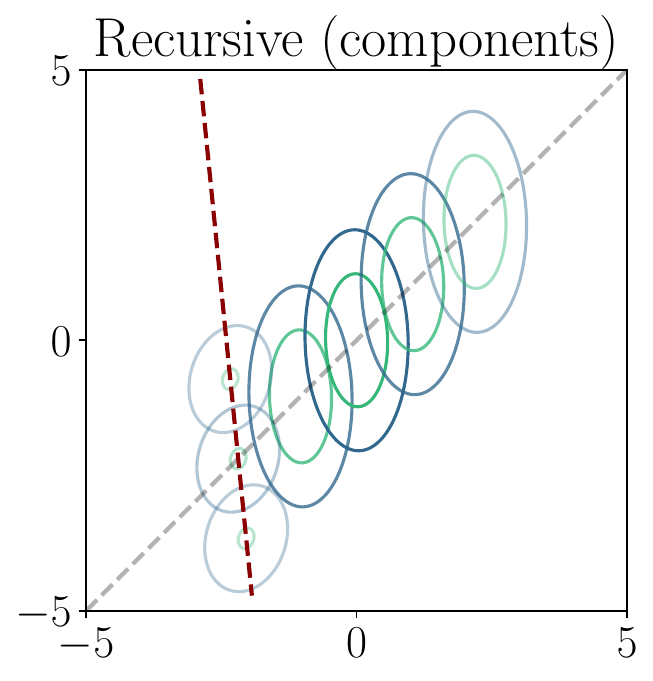}
		\\
		\includegraphics[width=0.22\textwidth]{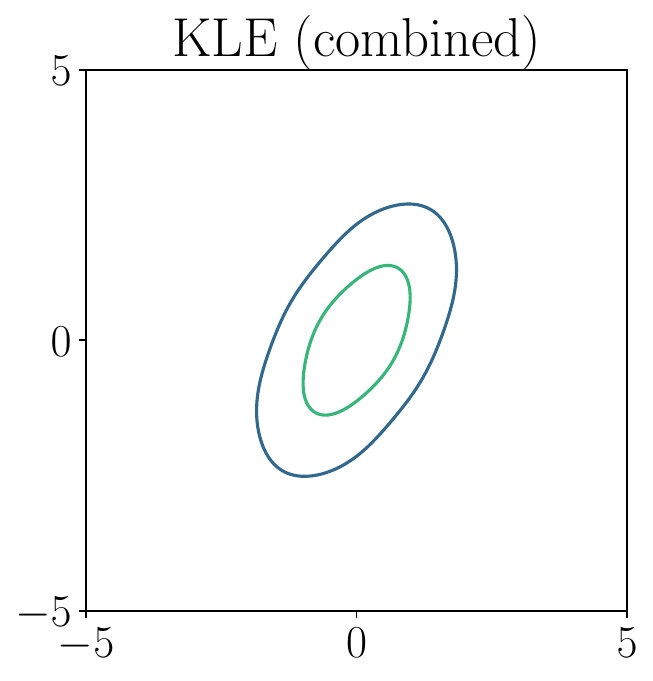}
		& 
		\includegraphics[width=0.22\textwidth]{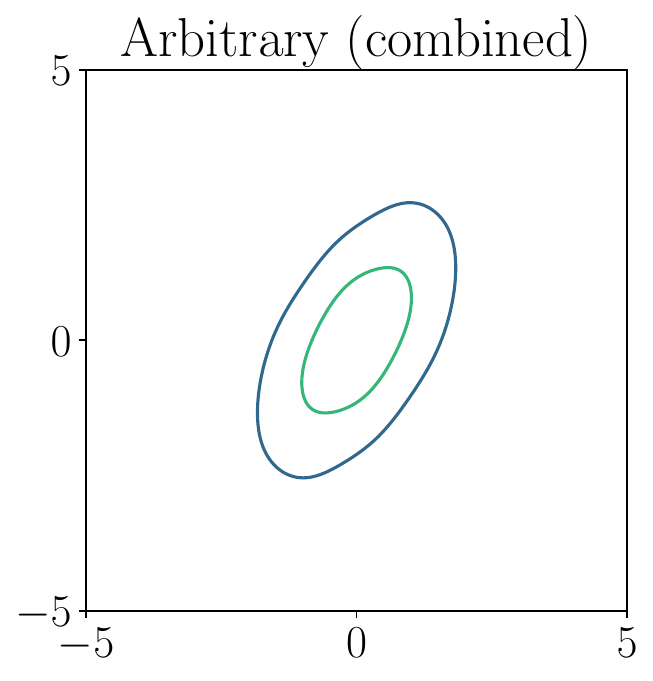}
		& 
		\includegraphics[width=0.22\textwidth]{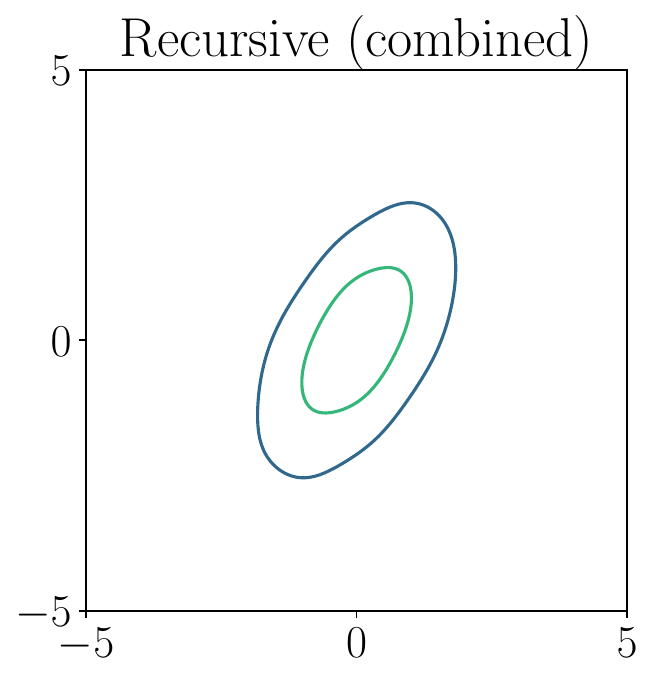}

		\end{tabular}
	\end{tabular}
	\caption{Gaussian mixture approximations of a 2D Gaussian distribution by 1D mixture decompositions along the dominant KLE direction, an arbitrary direction, and recursively by further decomposing a mixture component along a second direction. Top row shows the PDFs of the individual mixture components, where the opacity corresponds to the component weight and the dashed lines are the decomposition directions.. The bottom row shows the total mixture PDFs.}
	\label{fig:splitting}
\end{figure}

\section{Analysis of approximation errors}\label{sec:analysis}
\subsection{Gaussian mixture approximations of Gaussian measures}
The quality of the overall risk measure approximation will depend on the accuracy of the mixture approximation. 
Therefore, it is of interest to analyze the approximation errors committed by the 1D decomposition strategy presented in Section \ref{sec:gm}. 
We begin with an existence result for the 1D approximation for the standard Gaussian $\nu_0 = \cN(0,1)$. 
It is well known that Gaussian mixtures with increasing 
numbers of components and decreasing variances can approximate probability density functions (PDFs) to arbitrary accuracy in $L^1(\bR)$ or $L^{\infty}(\bR)$ (e.g. Theorem 1.1, \cite{Bacharoglou10}).
In particular, the $L^1(\bR)$ norm of the PDFs corresponds to the total variation (TV) distance, which, for two probability measures $\nu_{a}, \nu_{b}$ is given by
\begin{equation}
	\tv(\nu_{a}, \nu_{b}) = \sup_{A \in \Sigma} |\nu_{a}(A) - \nu_{b}(A)|.
\end{equation}
However, such results typically do not provide a relation between the number of mixture components and the variance of the components.
Thus, to show that the ansatz of $\sigma = N_{\mix}^{-p}$ used by \cite{VittaldevRussell16} to generate the mixture approximations is a well-motivated one, 
we provide the following result in the specific context of approximating Gaussian PDFs.
\begin{proposition}\label{thm:gm_1d}
Let $\nu_0 = \cN(0,1)$ be the standard normal distribution on $\bR$ with the PDFs 
$\pi_0(x) = \exp(-x^2/2)/\sqrt{2\pi}$. 
Given $p \in (0,1)$, for any $\epsilon > 0$, there exists some $N^* \in \bN$ such that for any $N \geq N^*$, there exists a corresponding Gaussian mixture approximation with $N$ components,
$\nu_{0,N} := \sum_{i=1}^{N} w_{N}^{i} \cN(\mu_{N}^{i}, \sigma_N^2)$ 
with uniform variance $\sigma_N = N^{-p}$ such that 
\begin{equation}
	\tv(\nu_0, \nu_{0,N}) \leq \epsilon.
\end{equation}
\end{proposition}
The proof of this proposition is constructive, making use of equally spaced Gaussian mixture components, and is presented in Appendix~\ref{appendix:proofs}. This result verifies the approximation capabilities of Gaussian mixtures with uniform variance $\sigma = N_{\mix}^{-p}$ in the context of approximating Gaussians, and justifies the strategy presented in \cite{VittaldevRussell16}. 

We then turn our attention to the approximation of multivariate Gaussians using the one-directional decomposition. 
\begin{proposition}[Dimension-independent mixture approximation error]\label{thm:gm}
Let $\nu = \cN(\bar{m}, \cC)$ be a non-degenerate Gaussian in $\bR^{n}$, and $\nu_{\mix} = \sum_{i=1}^{N} w_i \nu_i$ be the mixture approximation constructed along the direction $\psi = \cC^{1/2} \varphi$, from the 1D mixture with weights, means, and standard deviations $\{w_i, \mu_i, \sigma_i\}_{i=1}^{N}$ by \eqref{eq:arb_mix_mean_phi}--\eqref{eq:arb_mix_cov_phi} (equivalently, \eqref{eq:arb_mix_mean_psi}--\eqref{eq:arb_mix_cov_psi}),  where $\varphi \in \bR^{n}$ is a unit vector such that $\varphi^{\transpose} \varphi = 1$. That is, $\nu_i = \cN(\bar{m}_i, \cC_i)$ with 
$\bar{m}_i = \bar{m} + \mu_i \cC^{1/2} \varphi$ and 
$\cC_i = \cC^{1/2} (\cI + (\sigma_{i}^2 - 1)\varphi \otimes \varphi) \cC^{1/2}$.
Then, the TV distance between the Gaussian and its mixture approximation is given by 
\begin{equation}
	\tv(\nu, \nu_{\mix}) = \tv(\nu_0, \nu_{0, \mix})
\end{equation}
where $\nu_0 = \cN(0,1)$ and $\nu_{0, \mix} = \sum_{i=1}^{N} w_i \cN(\mu_i, \sigma_{i}^2)$ are the 1D standard Gaussian and its mixture approximations, respectively.
\end{proposition}
\begin{proof}
	Let $\pi$, $\pi_{mix}$, and $\pi_i$  denote the PDFs of the reference Gaussian $\nu$, mixture approximation $\nu_{\mix}$, and mixture components $\nu_i$, respectively. 
	We have 
	\begin{align*}
	\pi(m) &= \frac{1}{(2\pi)^{n/2} \det(\cC)^{1/2}} 
		\exp\left( -\frac{1}{2}(\bar{m} - \bar{m})^{\transpose} \cC^{-1}(m - \bar{m}) \right), \\
	\pi_i(m) &= \frac{1}{(2\pi)^{n/2} \det(\cC_i)^{1/2}} 
		\exp\left( -\frac{1}{2}(\bar{m} - \bar{m}_i)^{\transpose} \cC_i^{-1}(m - \bar{m}_i) \right),
	\end{align*}
	noting that in this case, $\cC$ and $\cC_{i}$ correspond to the covariance matrices of the original and mixture component Gaussian distributions.
	We also note that 
	\[ 
	\det(\cC_i) = \det(\cC^{1/2}(\cI_{\bR^{n}} + (\sigma_i^2 - 1) \varphi \otimes \varphi) \cC^{1/2}) =  \sigma_i^2 \det(\cC),
	\]
	where $\cI_{\bR^{n}}$ corresponds to the identity matrix on $\bR^{n}$, $\det(\cdot)$ is the matrix determinant, and $(\cdot)^{\transpose}$ is the vector transpose.
	We then use the following change of variables,
	$
		z = \Phi^{\transpose} \cC^{-1/2}(m - \bar{m}),
	$
	where $\Phi = [\varphi_1, \dots, \varphi_n] \in \bR^{n \times n}$ is an orthonormal basis of $\bR^{n}$ with the first column $\varphi_1 = \varphi$. Thus, we have 
	\[ 
	(m - \bar{m})^{\transpose} \cC^{-1} (m - \bar{m}) 
	= \left(\Phi^{\transpose} \cC^{-1/2}(m - \bar{m}) \right)^{\transpose} \! \! \left(\Phi^{\transpose} \cC^{-1/2}(m - \bar{m}) \right)
	= z^{\transpose} z
	\]
	since $\Phi^{\transpose} \Phi = \Phi \Phi^{\transpose} = \cI_{\bR^{n}}$. Similarly, 
	\begin{align*}
	(m - \bar{m}_i)^{\transpose} \cC_i^{-1} (m - \bar{m}_i) &= 
		\left(\cC^{-1/2}(m - \bar{m} - \mu_i \cC^{1/2}\varphi) \right)^{\transpose} \!\!
		\left((\cI_{\bR^{n}} - \varphi \otimes \varphi) + \frac{1}{\sigma_i^2} \varphi \otimes \varphi \right) \\
		& \qquad \left(\cC^{-1/2}(m - \bar{m} - \mu_i \cC^{1/2}\varphi) \right) \\
	&= (\Phi z - \mu_i \varphi)^{\transpose} 
		\left((\cI_{\bR^{n}} - \varphi \otimes \varphi) + \frac{1}{\sigma_i^2} \varphi \otimes \varphi \right)
		(\Phi z - \mu_i \varphi) \\
	&= z^{\transpose} \Phi^{\transpose}(\cI_{\bR^{n}} - \varphi \otimes \varphi) \Phi z 
		+ \frac{1}{\sigma_i^2} (\Phi z - \mu_i \varphi)^{\transpose} \varphi \otimes \varphi (\Phi z - \mu_i \varphi) \\
	&= z^{\transpose}({\cI}_{\bR^{n}} - e_1 \otimes e_1) z 
		+ \frac{1}{\sigma_i^2}\left((z^{\transpose} e_1)^2 - 2\mu_i z^{\transpose} e_1 + \mu_i^2 \right),
	\end{align*}
	where we have made use of the orthonormality of $\Phi$ with $\varphi=\varphi_1$. Here, $e_i = (1, 0 \dots, 0)$ refers to the first standard basis vector in $\bR^n$. Moreover, for $z = (z_1, \dots, z_n)$, we introduce $y = (z_2, \dots, z_n) \in \bR^{n - 1}$ to denote all but the first component of $z$. This allows us to write
	\[
	(m - \bar{m}_i)^{\transpose} \cC_i^{-1} (m - \bar{m}_i) = y^{\mathsf{T}} y + \frac{(z_1 - \mu_i)^2}{\sigma_i^2}
	\]
	Substituting this change of variables into the integral for TV in terms of the PDFs gives 
	\[
	\tv(\nu, \nu_{\mix}) = \frac{1}{2} \int \left| 
		1 - \frac{
		\sum_{i=1}^{N} \frac{1}{(2\pi)^{n/2} \sigma_i} \exp\left(-\frac{y^{\transpose} y}{2} - \frac{(z_1 - \mu_i)^2}{2\sigma_i^2}\right)
		}{
		\frac{1}{(2\pi)^{n/2}} \exp\left(-\frac{y^{\transpose} y}{2} - \frac{z_1^2}{2}\right)
		}
	\right|
		\frac{1}{(2\pi)^{n/2}} \exp\left(-\frac{y^{\transpose} y}{2} - \frac{z_1^2}{2}\right) dz.
	\]
	Cancelling and marginalizing over $y = (z_2, \dots, z_n)$ gives 
	\[
	\tv(\nu, \nu_{\mix}) = \frac{1}{2} \int \left| 
		1 - \frac{
		\sum_{i=1}^{N} \frac{1}{(2\pi)^{1/2} \sigma_i} \exp\left(- \frac{(z_1 - \mu_i)^2}{2\sigma_i^2}\right)
		}{
		\frac{1}{(2\pi)^{1/2}} \exp\left(- \frac{z_1^2}{2}\right)
		}
	\right|
		\frac{1}{(2\pi)^{1/2}} \exp\left(- \frac{z_1^2}{2}\right) dz_1.
	\]
	We recognize that this is simply 
	\[
	\tv(\nu, \nu_{\mix}) = \frac{1}{2} \int \left|1 - \frac{\pi_{0, \mix}(z_1)}{\pi_0(z_1)} \right| \pi_0(z_1) dz_1 
	= \tv(\nu_0, \nu_{0, \mix}),
	\]
	where $\pi_0$, $\pi_{0, \mix}$ are the 1D PDFs of $\nu_0$ and $\nu_{0, \mix}$, respectively. 
\end{proof}

Thus, the error committed by the approximation of $\nu$ along a particular direction reduces to the approximation error for the 1D Gaussian. 
Importantly, this does not depend on the direction $\varphi$, the dimension $n$, or the covariance structure $\cC$.
In particular, when $\cM$ is infinite dimensional and $\cN(\bar{m}, \cC)$ is a Gaussian random field, its finite dimensional discretizations are Gaussian random vectors in $\bR^{n}$. Hence, Proposition \ref{thm:gm} applies to such cases, where the TV error depends only on the choice of the 1D Gaussian mixture approximation $\nu_{0, \mix}$ of $\nu_0$, and not the discretization dimension.

Errors for multi-directional decompositions can then be obtained by applying the single directional results. We begin with a lemma.
\begin{lemma}\label{lemma:recursive_mix}
Let $\nu_{\mix} = \sum_{i=1}^{N} w_i \cN(\bar{m}_i, \cC_i)$ be a Gaussian mixture. Without loss of generality, consider approximating the $N^{th}$ component, $\nu_N = \cN(\bar{m}_N, \cC_N)$, by a Gaussian mixture 
$\nu_{N,\mix} = \sum_{i=1}^{M}w_{N,i}\;\cN(\bar{m}_{N,i}, \cC_{N,i})$, 
to obtain a new Gaussian mixture approximation of $\nu_{\mix}$,
\begin{equation}
\nu_{\mix}' 
= \sum_{i=1}^{N-1} w_i \cN(\bar{m}_i, \cC_i) + w_N \nu_{N,\mix}
= \sum_{i=1}^{N-1} w_i \cN(\bar{m}_i, \cC_i) + w_N \sum_{j=1}^{M} w_{N,j} \cN(\bar{m}_{N,j}, \cC_{N,j}).
\end{equation}
Then the TV distance between the two mixtures is given by
\begin{equation}
\tv(\nu_{\mix}, \nu_{\mix}') = w_N \; \tv(\nu_{N}, \nu_{N,\mix}).
\end{equation}
\end{lemma}
\begin{proof}
This follows directly from the definition of TV, since for any $A \in \cF$,
$$
\tv(\nu_{\mix}, \nu_{\mix}') = \sup_{A \in \cF} |\nu_{\mix}(A)- \nu_{mix}'(A)| 
	=  \sup_{A \in \cF} |w_N \nu_{N}(A) - w_N \nu_{\mix}^N(A)|
	= w_N \; \tv(\nu_{N}, \nu_{N,\mix}).
$$
We note that when the mixture approximation $\nu_{N,\mix} \approx \nu_N$ is constructed using the 1D decomposition of $\nu_N$ along some direction $\cC^{1/2}\varphi$, this error $\tv(\nu_{N}, \nu_{N,\mix})$ further reduces to the 1D TV distance as in Proposition \ref{thm:gm}.
\end{proof}

This can be directly applied to the recursive reconstruction algorithms of \cite{VittaldevRussell16} using a triangle inequality. That is, for a Gaussian mixture constructed by recursively expanding its components along single directions, the final TV distance can be computed as a cumulative sum over the TV errors of the 1D mixture approximations used every time when a mixture component is further decomposed, multiplied by the component's weight prior to its decomposition (as in Lemma \ref{lemma:recursive_mix}).

We provide a concrete use of this idea to derive the approximation error for the tensor product of mixture approximations along KLE modes. First note that we can extend the result of Lemma \ref{lemma:recursive_mix} as follows.
\begin{lemma}\label{lemma:recursive_entire}
Let $\nu_{\mix} = \sum_{i=1}^{N} w_i \cN(\bar{m}_i, \cC_i)$ be a Gaussian mixture. Suppose we approximate each component $\nu_i = \cN(\bar{m}_i, \cC_i)$, by a Gaussian mixture $\nu_{\mix,j} = \sum_{i=1}^{N_j}w_{j,i} \; \cN(\bar{m}_{j,i}, \cC_{j,i})$, to obtain a new Gaussian mixture,
$\nu_{mix}' 
= \sum_{j=1}^{N} w_j \nu_{\mix, j}.
$
Then the TV distance between the two mixtures is bounded by 
\[
\tv(\nu_{\mix}, \nu_{\mix}') \leq \sum_{j=1}^{N} w_j \tv(\nu_j, \nu_{\mix, j})
\]
\end{lemma}
\begin{proof}
This follows from the triangle inequality and Lemma \ref{lemma:recursive_mix}, applied successively for the approximation of each component of $\nu_{\mix}$, 
\begin{align*}
\tv(\nu_{\mix}, \nu_{\mix}') &= \tv \left(\sum_{j=1}^{N} w_i \nu_i, \sum_{j=1}^{N} w_i \nu_{\mix,i} \right) \\
&\leq \tv \left(\sum_{j=1}^{N} w_j \nu_j, \sum_{j=1}^{N-1}\nu_j + w_N \nu_{\mix, N} \right)
+ \tv \left(\sum_{j=1}^{N-1}\nu_j + w_N \nu_{\mix, N},  \sum_{j=1}^{N} w_i \nu_{\mix,i}\right) \\
&\leq w_N \tv(\nu_N, \nu_{\mix,N})
+ \tv \left(\sum_{j=1}^{N-1}\nu_j,  \sum_{j=1}^{N-1} w_i \nu_{\mix,i}\right).
\end{align*}
Repeating this process for the remaining $N-1$ components yields the desired result.
\end{proof}

Again, when the individual mixtures $\nu_{\mix,j} \approx \nu_j$ are constructed using a 1D mixture $\nu_{\mix,0}$ along some direction $\cC^{1/2} \varphi$, as in Proposition \ref{thm:gm}, the distance reduces to the weighted sum over the individual 1D Gaussian mixture approximations.
For example, when the same 1D mixture approximation $\nu_{0, \mix}$ is used for all mixture components $i \leq N$, we simply have 
\[
\tv(\nu_{\mix}, \nu_{\mix}') \leq \tv(\nu_{0}, \nu_{0, \mix}).
\]
since $\sum_{j=1}^{N}w_j = 1$.
This allows us to derive an error estimate for the multi-directional decomposition case. 
\begin{proposition}
 	Let $\nu = \cN(\bar{m}, \cC)$ be a non-degenerate Gaussian on $\bR^{n}$. Suppose $\nu_{\mix} = \sum_{\bs{i}} w_{\bs{i} \cN(\bar{m}_{\bs{i}}, \cC_{\bs{i}})}$ is a mixture approximation constructed as a tensor product along the first $r_K$ KLE modes using \eqref{eq:kl_mix_multi_weight}--\eqref{eq:kl_mix_multi_cov}. Then, the overall approximation error is given by 
	\begin{equation}
 	\tv(\nu, \nu_{\mix}) = \sum_{k=1}^{r_K} \tv(\nu_{0}, \nu_{0, \mix}^{(k)}),
	\end{equation}
	where $\nu_{0, \mix}^{(k)}$ is the mixture approximation of the standard 1D Gaussian used for the $k$th KLE mode with weights, means, and standard deviations given by $\{w_{i_k}, \mu_{i_k}, \sigma_{i_k}\}_{i_k=1}^{N_k}$.
\end{proposition}
\begin{proof}
We start by noting that this mixture can be recursively constructed by decomposing the initial Gaussian first in the $k=1$ mode. This is followed by decomposing the components of the new mixture in the $k=2$ mode, and then continuing for all $r_K$ directions.
Thus, the error can be computed by successively applying Lemma \ref{lemma:recursive_entire} recursively for each decomposition procedure for $k = 1, \dots, r_K$. Since 
the the same 1D approximations are used for a particular $k$, each successive decomposition direction contributes an additional error of $\tv(\nu_{0}, \nu_{0,\mix}^{(k)})$. 
\end{proof}

\subsection{Mixture Taylor approximation error}
We now consider the combined error of the Gaussian mixture and Taylor approximations, which depends on both the mixture approximation error and the Taylor truncation error. 
This is summarized in the following estimate for the expectation. 
\begin{proposition}
Let $Q_{s}$ be the $s^{th}$ order Taylor approximation of the quantity of interest $Q$, $\nu_{\mix} = \sum_{i=1}^{N_{\mix}}w_i \nu_i$, $\nu_i = \cN(\bar{m}_i, \cC_i)$ be a Gaussian mixture approximation of the $\nu = \cN(\bar{m}, \cC)$, and $Q$ has finite second moments for each $\nu_i$ and $\nu$. Let
\[
\widehat{Q}_{s, \gm} = \sum_{i=1}^{N_{\mix}} w_i \bE_{\nu_i}[Q_{s,i}] 
\]
be the mixture Taylor approximation estimator of order $s$,
where $Q_{s,i}$ are the $s^{th}$ Taylor approximations centered at the means $\bar{m}_i$. Then, the approximation error has the bound
\begin{equation}\label{eq:mean_mix_bound}
| \bE_{\nu}[Q] - \widehat{Q}_{s,\gm} | \leq
	2 (\bE_{\nu}[Q^2] + \bE_{\nu_{\mix}}[Q^2])^{1/2} \tv(\nu, \nu_{\mix})^{1/2}
	+ \sum_{i=1}^{N_{\mix}} w_i | \bE_{\nu_i}[Q] - \bE_{\nu_i}[Q_{s,i}] |.
\end{equation}
\end{proposition}
\begin{proof}
We can decompose the error as 
\[
| \bE_{\nu}[Q] - \widehat{Q}_{s,\gm} | \leq
	| \bE_{\nu}[Q] - \bE_{\nu_{\mix}}[Q] | 
	+ \sum_{i=1}^{N_{\mix}} w_i | \bE_{\nu_i}[Q] - \bE_{\nu_i}[Q_{s,i}] |
\]
where $s$ is the order of the Taylor approximation (linear/quadratic), and $Q_{s,i}$ denotes the $s^{th}$ order Taylor approximation about the mean $\bar{m}_i$ of component $\nu_i$.
We can bound the mixture approximation error by a standard result from probability theory (e.g. Lemma 1.30, \cite{LawStuartZygalakis15}),
\[
| \bE_{\nu}[Q] - \bE_{\nu_{\mix}}[Q] | \leq 2(\bE_{\nu}[Q^2] + \bE_{\nu_{\mix}}[Q^2])^{1/2} \tv(\nu, \nu_{\mix})^{1/2}.
\]
The overall approximation error bound follows.
\end{proof}

We use this result to gain some intuition for the behavior of the approximation error by considering the quadratic Taylor approximation. Recall that for QoIs with bounded third order derivatives, the quadratic Taylor approximation error satisfies 
\[ 
| \bE_{\nu_i}[Q] - \bE_{\nu_i}[Q_{\qua,i}] | \leq \sqrt{3} K \tr(\cC_i)^{3/2}.
\]
Therefore, if the same covariance is used for each mixture component, i.e. $\cC_i = \cC_{\mix}$, we have 
\[
| \bE_{\nu}[Q] - \widehat{Q}_{\qua,\gm} | \leq
	2(\bE_{\nu}[Q^2] + \bE_{\nu_{\mix}}[Q^2])^{1/2} \tv(\nu, \nu_{\mix})^{1/2}
	+ \sqrt{3} K \tr(\cC_{\mix})^{3/2},
\]
since $\sum_{i=1}^{N_{\mix}} w_i = 1$. From this, we see that the mixture Taylor approximation works in a pre-asymptotic manner by reducing $\tr(\cC_{\mix})$ while maintaining a small $\tv(\nu, \nu_{\mix})$. For example, if $\nu_i$ is obtained by a 1D decomposition along the dominant KLE direction $\lambda_1, \phi_1$ using the 1D mixture given by $\{w_i, \mu_i, \sigma \}_{i=1}^{N_{\mix}}$ as in \eqref{eq:kl_mix_mean}--\eqref{eq:kl_mix_cov}, we have 
$ \tr(\cC_{\mix}) = \sigma^2 \lambda_1 + \sum_{k=2}^{\infty} \lambda_k. $
The effect of the variance reduction due to the mixture approximation can be seen by comparing this to the original error bound, $\sqrt{3} K \tr(\cC)^{3/2}$. We note
\[
\tr(\cC)^{3/2} - \tr(\cC_{\mix})^{3/2} = 
\left(\lambda_1 + \sum_{k=2}^{\infty} \lambda_k \right)^{3/2}
- \left(\sigma^2 \lambda_1 + \sum_{k=2}^{\infty} \lambda_k \right)^{3/2},
\]
for which we have the lower bound
\[
\tr(\cC)^{3/2} - \tr(\cC_{\mix})^{3/2} \geq
(1 - \sigma^2) \lambda_1^{3/2},
\]
where we have made use of the inequality $(1+x)^{3/2} - (\sigma^2 + x)^{3/2} \geq (1-\sigma^2)^{3/2}$ for $x \geq 0$ and $\sigma \in [0,1].$ Thus, compared to the single quadratic approximation about the mean $\bar{m}$, the reduction in the upper bound of the errors can be estimated as 
\[
	(1-\sigma^2) \sqrt{3} K \lambda_1^{3/2}
	- 2(\bE_{\nu}[Q^2] + \bE_{\nu_{\mix}}[Q^2])^{1/2} \tv(\nu, \nu_{\mix})^{1/2}.
\]
From this, we expect that the mixture Taylor approximation along a dominant KLE direction is effective when $\sigma$ and $\lambda_1$ are large, and $\tv(\nu, \nu_{\mix})$ can be made small.

For the CVaR, we can derive an analogous error estimate.
\begin{proposition}
Let $Q_{s}$ be the $s^{th}$ order Taylor approximation of the quantity of interest $Q$, $\nu_{\mix} = \sum_{i=1}^{N_{\mix}} \nu_i$ be a Gaussian mixture approximation of the $\nu$, for which $Q$ has a finite second moment, and 
\[
(\widehat{C}_{\alpha})_{s,\gm} = \min_{t \in \bR} t + \frac{1}{1-\alpha} \sum_{i=1}^{N_{\mix}} w_i \bE_{\nu_i}[(Q_{s,i} - t)^+]
\]
be the mixture Taylor approximation estimator of order $s$,
where $Q_{s,i}$ are the $s^{th}$ Taylor approximations centered at the means $\bar{m}_i$. Then, the approximation error has the bound
\begin{equation}\label{eq:cvar_mix_bound}
| \CVaR_{\alpha}[Q] - (\widehat{C}_{\alpha})_{s,\gm} | 
\leq \frac{1}{1-\alpha}\left( 
2(\bE_{\nu}[Q^2] + \bE_{\nu_{\mix}}[Q^2])^{1/2} \tv(\nu, \nu_{\mix})^{1/2} 
+ \sum_{i=1}^{N_{\mix}} w_i \bE_{\nu_i}[|Q - Q_{s,i}|]
\right) 
\end{equation}
for $0 < \alpha < 1$. 
\end{proposition}
\begin{proof}
Let
\[
f(t) := t + \frac{1}{1-\alpha} \bE_{\nu}[(Q - t)^{+}] \text{ and }
g(t) := 
t + \frac{1}{1-\alpha} \sum_{i=1}^{N_{\mix}} w_i \bE_{\nu_i}[(Q_{s,i} - t)^+]
\]
be the cost functionals for the original CVaR and mixture Taylor CVaR approximation, respectively. Then
\begin{align*}
|f(t) - g(t)| &= \frac{1}{1-\alpha} \left| \bE_{\nu}[(Q - t)^{+}] - \sum_{i=1}^{N_{\mix}} w_i \bE_{\nu_i}[(Q_{s,i} - t)^{+}] \right|\\
&\leq \frac{1}{1-\alpha} \left( 
	|\bE_{\nu}[(Q - t)^{+}] - \bE_{\nu_{\mix}}[(Q - t)^{+}]| + \sum_{i=1}^{N_{\mix}} w_i \bE_{\nu_i}[|(Q - t)^{+} - (Q_{s,i} - t)^{+}|],
\right)
\end{align*}
where the error is decomposed into a mixture approximation error and individual Taylor approximation errors for the mixture components. Since the maximum function, $(\cdot)^{+}$, is Lipschitz continuous with Lipschitz constant one, we have 
\begin{equation}\label{eq:lipschitz}
	|(Q - t)^{+} - (Q_{s,i} - t)^{+}| \leq |(Q - t) - (Q_{s,i} - t)| = |Q - Q_{s,i}|.
\end{equation}
This leads to a conservative bound for the Taylor approximation errors
\[
\bE_{\nu_i}[|(Q - t)^{+} - (Q_{s,i} - t)^{+}|] \leq \bE_{\nu_i}[|Q - Q_{s,i}|].
\]
For the mixture approximation error, we consider the cases of $t \geq 0$ and $t < 0$ separately. 
When $t \geq 0$, under the assumption of finite second moments, we have
\[
|\bE_{\nu}[(Q - t)^{+}] - \bE_{\nu_{\mix}}[(Q - t)^{+}]| \leq 
2 ( \bE_{\nu}[|(Q - t)^{+}|^2] +  \bE_{\nu_{\mix}}[|(Q - t)^{+}|^2] )^{1/2} \tv(\nu, \nu_{\mix})^{1/2}.
\]
Since $|(Q - t)^{+}| \leq |Q|$ when $t \geq 0$, we can replace the upper bound by 
\[
2 ( \bE_{\nu}[Q^2] +  \bE_{\nu_{\mix}}[Q^2] )^{1/2} \tv(\nu, \nu_{\mix})^{1/2}. 
\]
When $t < 0$, we can obtain the same bound as in the $t \geq 0$ case by
\begin{align*}|\bE_{\nu}[(Q - t)^{+}] - \bE_{\nu_{\mix}}[(Q - t)^{+}]|  
	&= |\bE_{\nu}[(Q - t)^{+} + t] - \bE_{\nu_{\mix}}[(Q - t)^{+} + t]| \\
	&\leq 2 (\bE_{\nu}[|(Q - t)^{+} + t|^2] +  \bE_{\nu_{\mix}}[|(Q - t)^{+} + t|^2] )^{1/2} tv(\nu, \nu_{\mix})^{1/2} \\
	&\leq 2 (\bE_{\nu}[|Q|^2] +  \bE_{\nu_{\mix}}[|Q|^2] )^{1/2} \tv(\nu, \nu_{\mix})^{1/2},
\end{align*}
where we have used $|(Q-t)^{+} + t| \leq |Q|$ for $t < 0$.

Overall, for any $t \in \bR$, we have 
\begin{equation}
|f(t) - g(t) |
\leq \frac{1}{1-\alpha}\left( 
2(\bE_{\nu}[Q^2] + \bE_{\nu_{\mix}}[Q^2])^{1/2} \tv(\nu, \nu_{\mix})^{1/2}
+ \sum_{i=1}^{N_{\mix}} w_i \bE_{\nu_i}[|Q - Q_{s,i}|]
\right)
\end{equation}
as an upper bound.
As both minima are attained, we have
\[
|\CVaR_{\alpha}[Q] - (\widehat{C}_{\alpha})_{s,\gm} | = | \min_{t \in \bR} f(t) - \min_{t \in \bR} g(t) | \leq \sup_{t \in \bR}|f(t) - g(t)|, 
\]
which leads to the bound in \eqref{eq:cvar_mix_bound}.
\end{proof}

In this analysis, the CVaR error bound is made to look like the case for the expectation, with a constant that degrades as $\alpha \rightarrow 1$. However, for large values of $\alpha$, one expects $\VaR_{\alpha}$ to be large and hence several inequalities, such as \eqref{eq:lipschitz} become excessively conservative.

\section{Numerical results}\label{sec:numerical_results}

\subsection{Advection-diffusion-reaction equation with a log-normal conductivity field}
We first consider a semilinear advection-diffusion-reaction (ADR) equation with an uncertain conductivity field over the square domain $\Omega = (0,1)^2$,
\begin{align}
	-\nabla \cdot (e^m \nabla u) + \bs{v} \cdot \nabla u + a u^3 &= f \quad \text{in } \Omega, \\
	u &= 0 \quad \text{on } \Gamma_L,  \\ 
	e^m \nabla u \cdot \bs{n} &= 0 \quad \text{on } \Gamma \setminus \Gamma_L.
\end{align}
Here, $\bs{n}$ is the unit outward normal, $\Gamma$ is the domain boundary, and $\Gamma_L$ is the left boundary $x_1 = 0$. 
The velocity $\bs{v} = (0.1, 0.1)^{\transpose}$ and reaction coefficient $a = 0.01$ are assumed to be fixed, and $f$ is a Gaussian source centered at (0.25, 0.5). 
The random parameter, $m \sim \cN(\bar{m}, \cC)$ is a Gaussian random field with covariance of the form $\cC = (\delta -\gamma \Delta)^{-2}$ with Robin boundary conditions following \cite{DaonStadler18}. 
In this case, the resulting covariance function has a characteristic correlation length of $l_x = \sqrt{8 \gamma / \delta}$ and pointwise variance of $\sigma^2_{x} = \sqrt{\gamma \delta}$. 
A diagram of the problem along with a realization of $m$ and its corresponding state $u(m)$ are shown in Figure~\ref{fig:adr_sample} for the case with pointwise variance and correlation both equal to 1.

\begin{figure}[htpb!]
	\centering
	\begin{subfigure}{0.32\textwidth}
		\centering
		\begin{tikzpicture}
		\draw[thick] (0,0) -- (3,0) -- (3,3) -- (0,3) -- (0,0);
		\draw[very thick, blue] (0,3) -- (0,0);
		\draw[thick,-stealth] (1.5,0.75) -- (2.25,1.5);
		\node at (1.5, 0.5) {$\bs{v} = (0.1, 0.1)$};
		\node at (-0.3, 1.5) {$\Gamma_L$};

		\filldraw[black] (0.75,1.5) circle (2pt) node[anchor=north]{Source};
		\end{tikzpicture}
	\end{subfigure}
	\begin{subfigure}{0.25\textwidth}
		\centering
		\includegraphics[width=\textwidth]{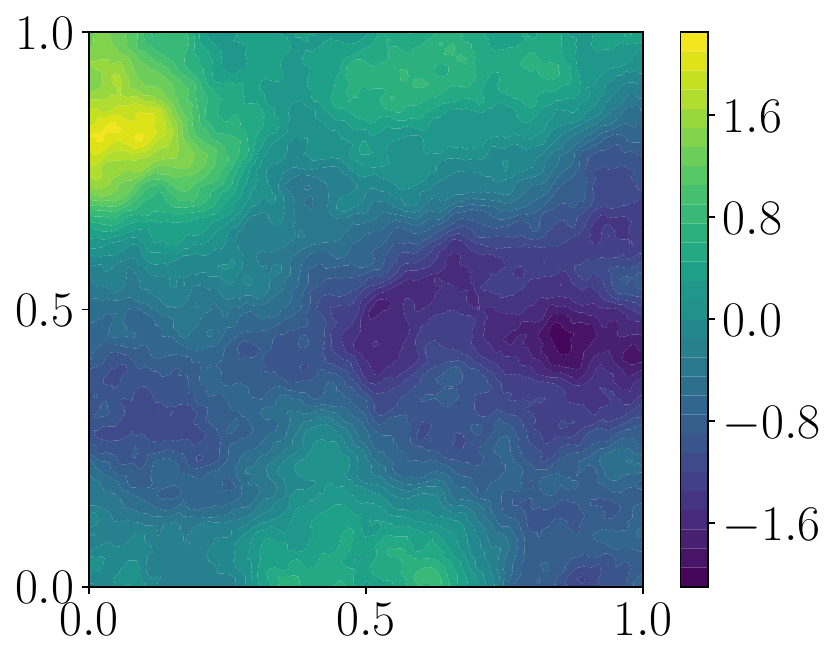}
	\end{subfigure}
	\begin{subfigure}{0.25\textwidth}
		\centering
		\includegraphics[width=\textwidth]{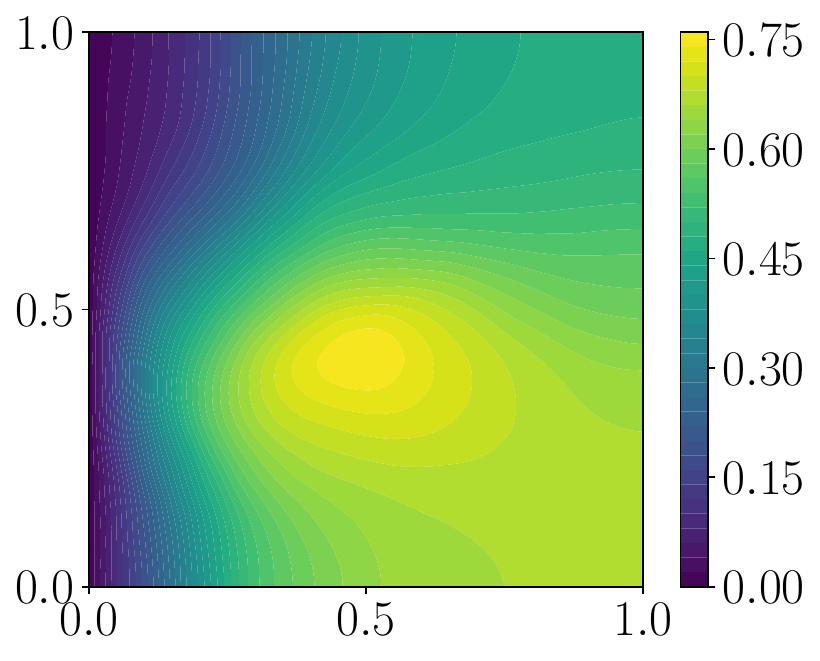}
	\end{subfigure}
	\caption{Left: diagram of the PDE domain and set. Center: a sample of the parameter $m \sim \cN(\bar{m}, \cC)$. Right: corresponding sample PDE solution $u = u(m)$. Samples are obtained for correlation length and pointwise variances equal to 1.}
	\label{fig:adr_sample}
\end{figure}
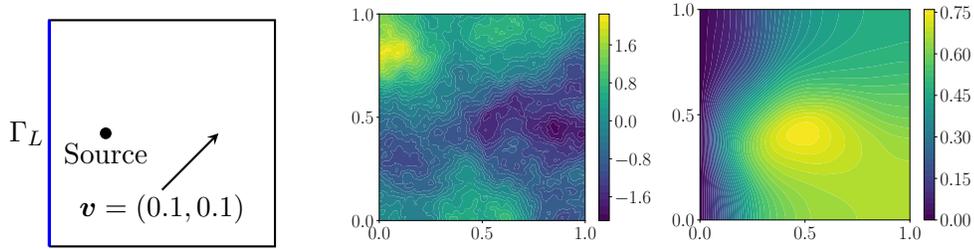

We consider three quantities of interest,
\begin{align}
	Q_{L^2} &= \int_{\Omega} u^2 \dx, \\
	Q_{L^3} &= \int_{\Omega} u^3 \dx, \\
	Q_{\text{energy}} &= \int_{\Omega} e^{m} |\nabla u|^2 \dx,
\end{align}
all of which depend non-linearly on the state, and the third also depends explicitly on the parameter. 

Figure \ref{fig:adr_spectra} shows the spectra of the covariance operator $\cC$ and the covariance-preconditioned Hessian operators $\cH = \cC^{1/2} D^2 Q(\bar{m}) \cC^{1/2}$ corresponding to each quantity of interest, with correlation lengths $l_x = 0.1, 0.25, 0.5, 1, 2$, and pointwise variance $\sigma_x^2 = 1$. 
Note that though the eigenvalues decay for all the operators, the decay tends to be slower with smaller correlation lengths as they have larger contributions from high frequency modes. 
Despite this, the covariance-preconditioned Hessian operator tends to reveal a dominant eigenvector whose eigenvalue is several times larger than the subsequent eigenvector.
For illustration, the dominant eigenvectors of $\cC$ and $\cH$ are compared in Figure \ref{fig:adr_eigenvectors} for the $L^2$ QoI and the case of $l_x = 1$ and $\sigma_x^2 = 1$.
We observe that the eigenvector of $\cH$ captures information about the underlying mapping, 
and is influenced by the source and boundary conditions.

\begin{figure}[h!]
	\centering
	\begin{subfigure}{0.24\textwidth}
		\centering
		\includegraphics[width=\textwidth]{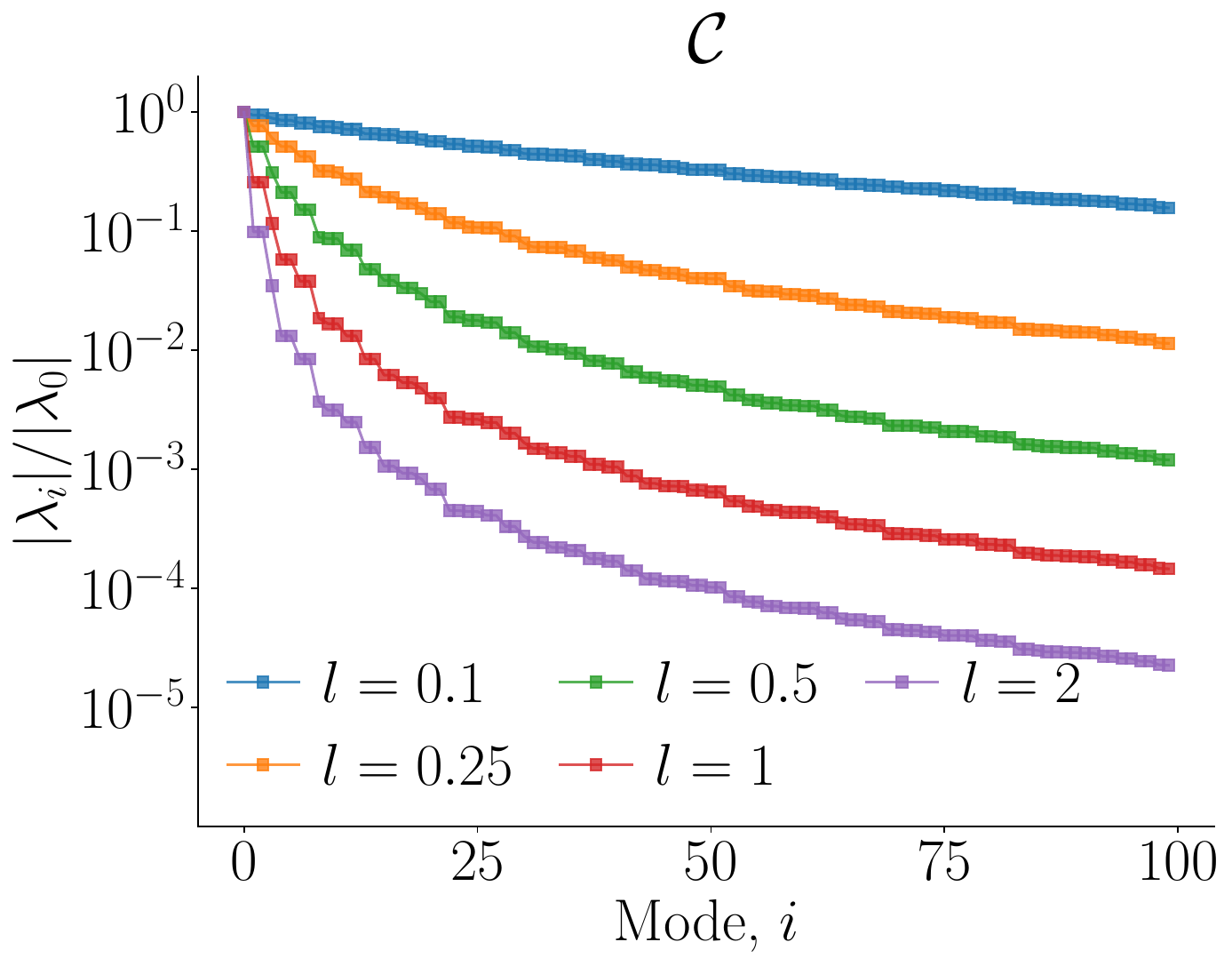}
	\end{subfigure}
	\begin{subfigure}{0.24\textwidth}
		\centering
		\includegraphics[width=\textwidth]{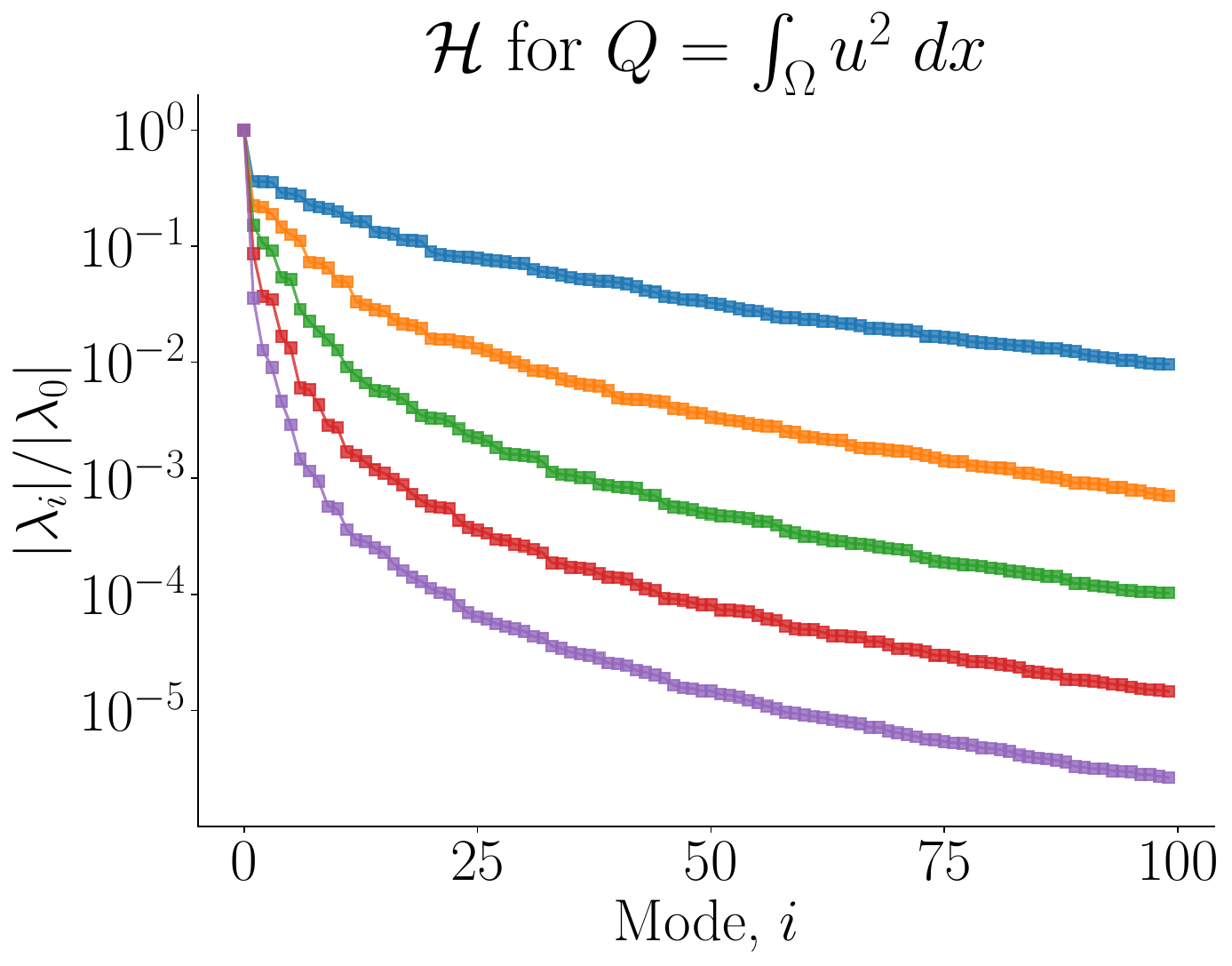}
	\end{subfigure}
	\begin{subfigure}{0.24\textwidth}
		\centering
		\includegraphics[width=\textwidth]{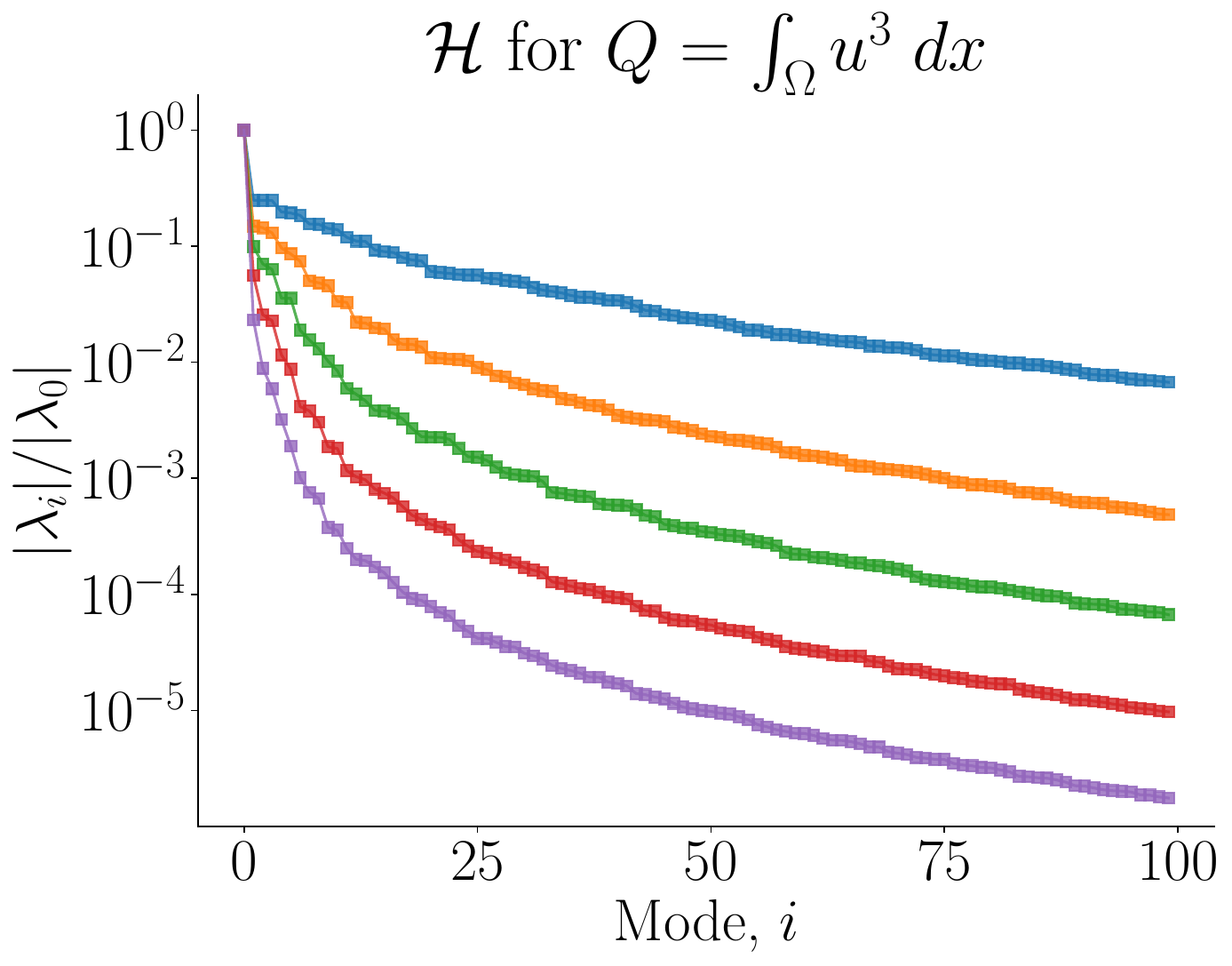}
	\end{subfigure}
	\begin{subfigure}{0.24\textwidth}
		\centering
		\includegraphics[width=\textwidth]{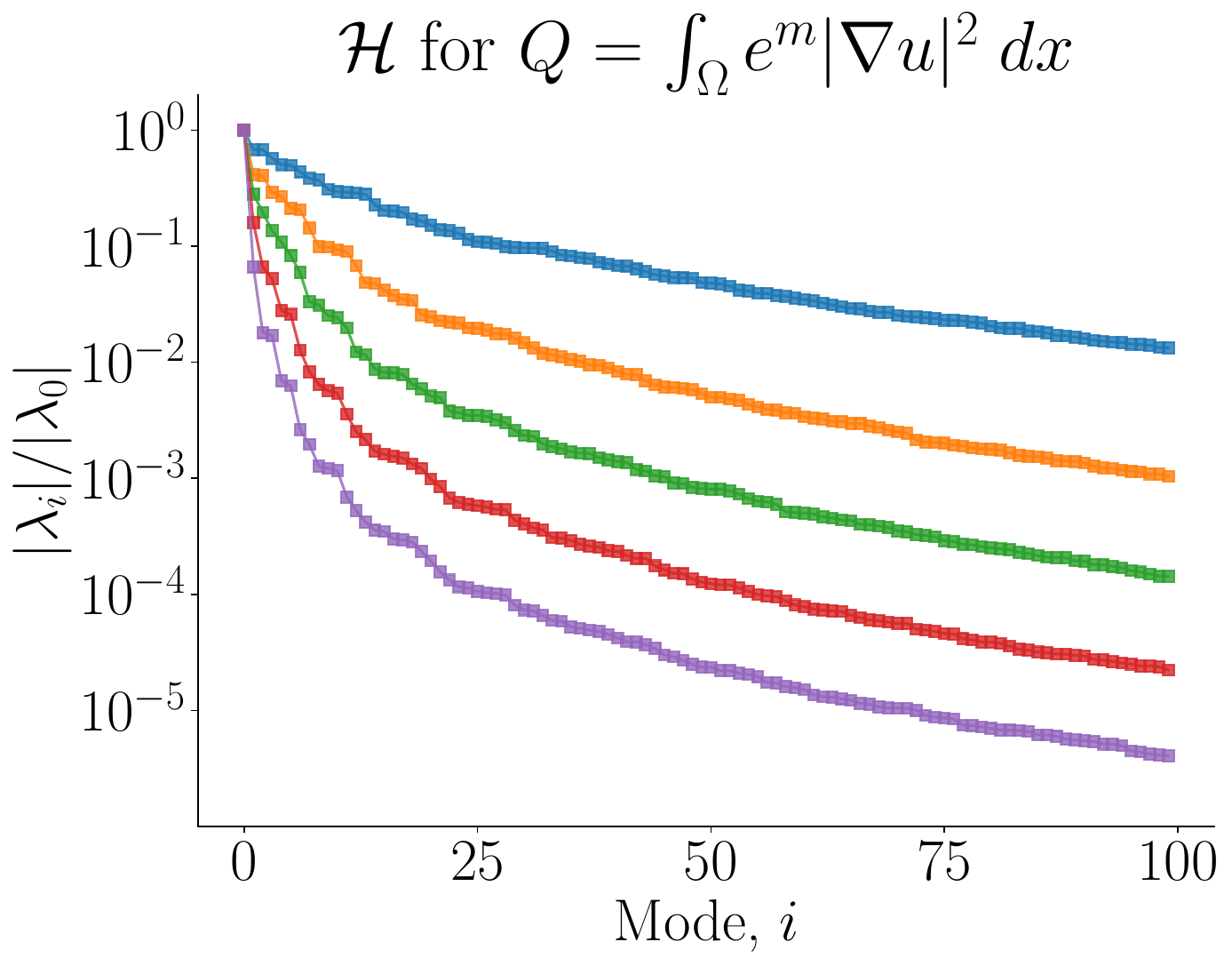}
	\end{subfigure}

	\caption{Magnitudes of the first 100 eigenvalues (as a ratio to the largest eigenvalue) for the covariance operator (left) and the covariance-preconditioned Hessian of the QoIs at $\bar{m}$ (left to right: $L^2$, $L^3$, and energy). Spectra are shown for correlation lengths of $0.1, 0.25, 0.5, 1, 2$ and a pointwise variance of $1$.}
	\label{fig:adr_spectra}
\end{figure}

\begin{figure}[h!]
	\centering
	\begin{subfigure}{0.2\textwidth}
		\centering
		\includegraphics[width=\textwidth]{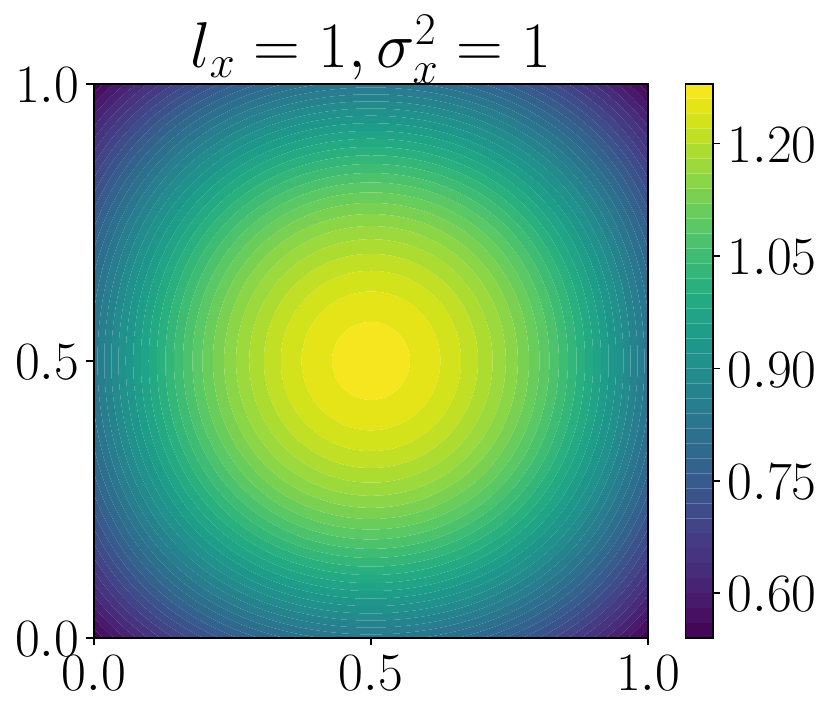}
	\end{subfigure}
	\begin{subfigure}{0.2\textwidth}
		\centering
		\includegraphics[width=\textwidth]{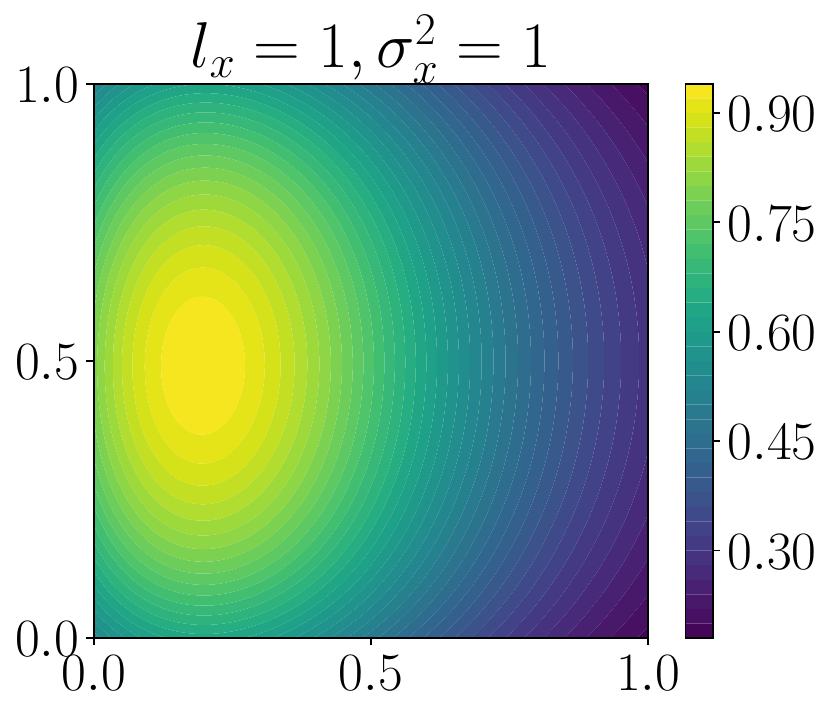}
	\end{subfigure}
	\caption{Dominant eigenvector of the covariance (left) and covariance-preconditioned Hessian at $\bar{m}$ for the $L^2$ QoI (right) with correlation length and pointwise variance both equal to one.}
	\label{fig:adr_eigenvectors}
\end{figure}

We then consider evaluating the mean, variance, and CVaR values using the Gaussian mixture Taylor approximation. 
The mixture approximations are constructed using up to 39 mixture components, generated using the rule $\sigma = N^{-1/2}$, along the dominant eigenvector of either the covariance or preconditioned Hessian operator. 
We then consider both linear and low-rank quadratic Taylor approximations at each of the mixture component means.
In particular, for the quadratic Taylor approximations, we use a uniform rank of $r_{\cH} = 200$ and draw a large number samples ($10^5$) for each component to make sampling errors negligible. 

The estimated values of the risk measures are compared against ``ground truth'' values computed by Monte Carlo (MC) sampling using $10^5$ samples. 
The mixture Taylor approximations and their relative errors are shown in Figure \ref{fig:adr_mean_sd_cvar} for the $L^2$ QoI using a covariance with $l_x = 1$ and $\sigma_x^2 = 1$, plotted as a function of the number of mixture components. 
Note that the points with $N_{\mix} = 1$ corresponds to the standard Taylor approximation at the mean, without the Gaussian mixture approximation.
We also compare against MC estimates using up to 1,000 samples, which are also shown in Figure \ref{fig:adr_mean_sd_cvar}, 
where we plot the band of the mean $\pm$ one standard deviation of the MC estimates over 200 trials, as well as their relative root-mean-square-errors (RMSEs).
The relative error and relative RMSEs are defined as 
\[
\text{Relative error} := \frac{|\widehat{\rho} - \rho_{\nu}[Q]|}{|\rho_{\nu}[Q]|}, \qquad 
\text{Relative RMSE} := \frac{\sqrt{\sum_{i=1}^{N_{\text{trial}}} (\widehat{\rho}_{M}^{(i)} - \rho_{\nu}[Q])^2/N_{\text{trial}}}}{|\rho_{\nu}[Q]|},
\]
where $\rho_{\nu}[Q]$ refers to the true value of the risk measure, 
$\widehat{\rho}$ is an approximation of the risk measure, 
and $\widehat{\rho}_{{M}}^{(i)}$, $i = 1, \dots N_{\text{trial}}$ are $N_{\text{trial}}$ trials of sampling estimates with $M$ samples each.

For the QoI and covariance operator shown, 
we recognize that the variance of the QoI is sufficiently large
such that MC estimates with $10^3$ samples still show RMSEs of approximately 10\%.
In this case, neither the linear or quadratic Taylor approximations are accurate, yielding errors between 10\%--100\%.
However, adding mixture components improves accuracy for all cases, producing 1--2 orders of magnitude of improvement with $N_{\mix} = 39$.  
In particular, with ${N_{\mix} = 39}$, the quadratic approximations errors tend to be $<1\%$, 
which are 1--2 orders of magnitude smaller than MC estimates with $10^3$ samples,
while the linear approximations using the HEP-based mixture decomposition are comparable in accuracy MC estimates with $10^3$ samples. 
However, we also note that the approximations tend to stagnate in accuracy. That is, increasing $N_{\mix}$ further yields diminishing returns due the truncation errors in directions other than the expansion direction. 
In this problem, the state equation is typically takes $n_L = 3$ Newton iterations to solve, and so the computational cost associated with each mixture component is approximately $n_L + 2 = 5$ unique linear PDE solves when using the quadratic Taylor approximation,
i.e.\ approximately $1.7\times$ that of each MC sample.
Taking this into account, for the PDE and QoIs considered, the Gaussian mixture with quadratic Taylor approximations can, nevertheless, be over 1--2 orders more efficient compared to standard MC estimation.

\begin{figure}[htbp!]
	\centering
	\begin{subfigure}{0.32\textwidth}
	\centering
	\includegraphics[width=\textwidth]{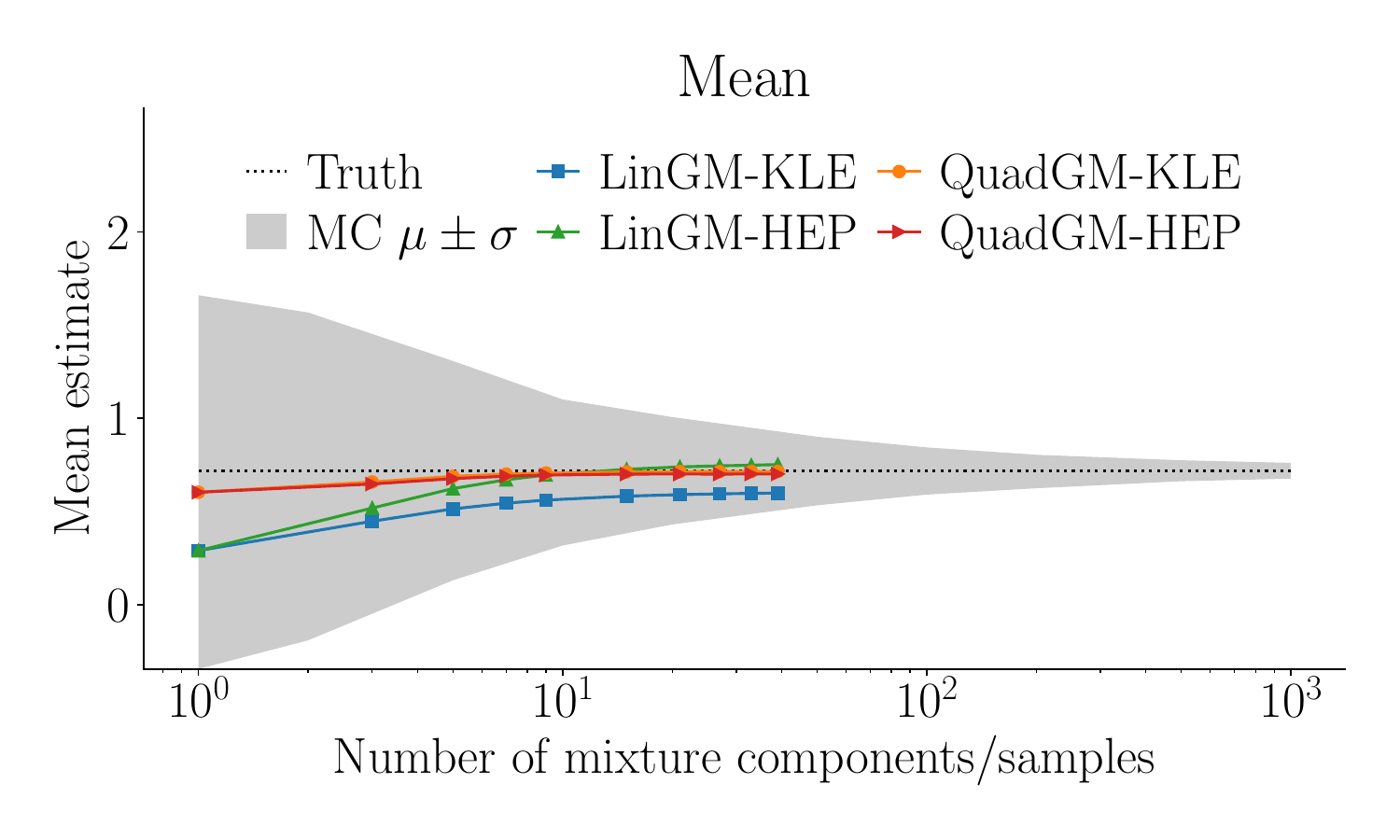}
	\end{subfigure}
	\begin{subfigure}{0.32\textwidth}
	\centering
	\includegraphics[width=\textwidth]{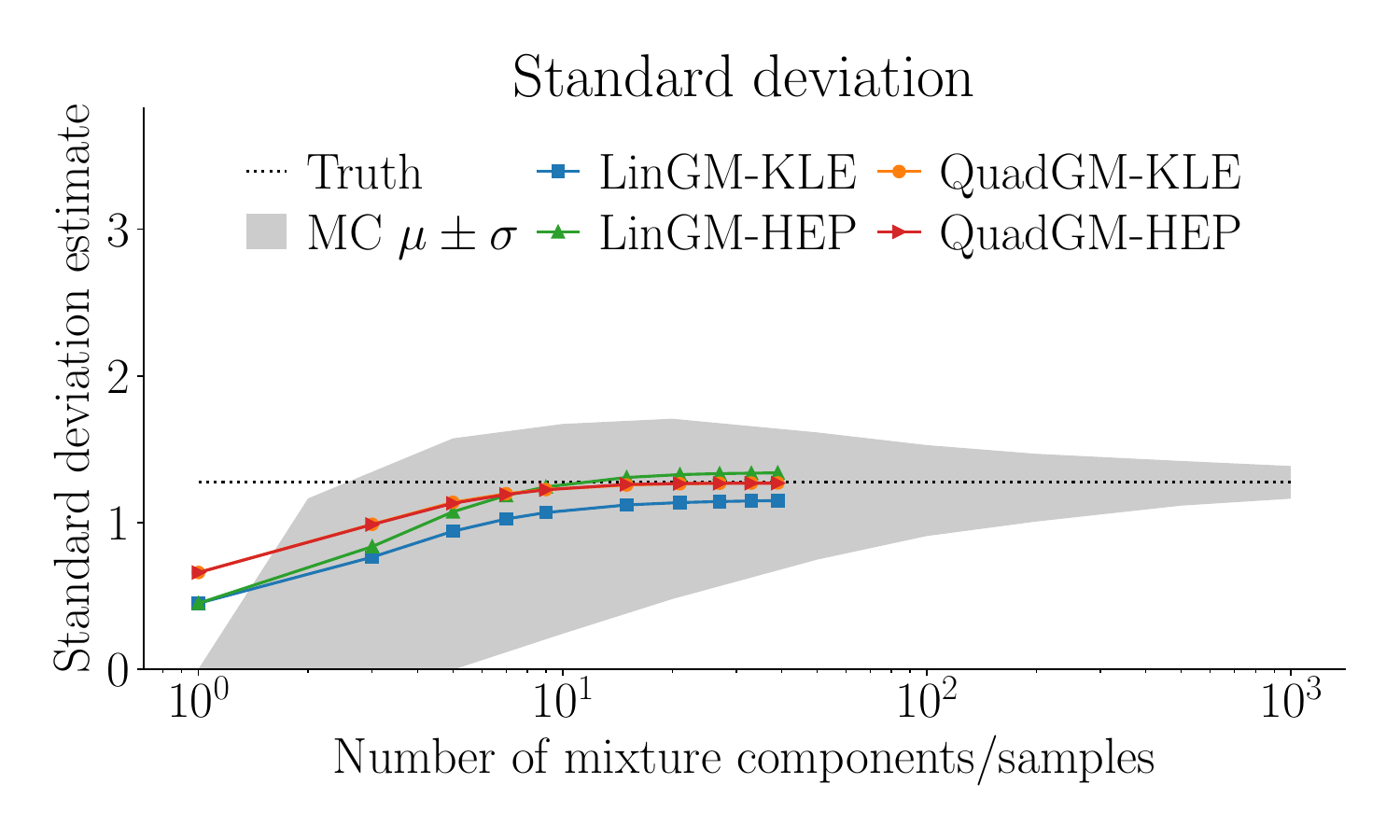}
	\end{subfigure}
	\begin{subfigure}{0.32\textwidth}
	\centering
	\includegraphics[width=\textwidth]{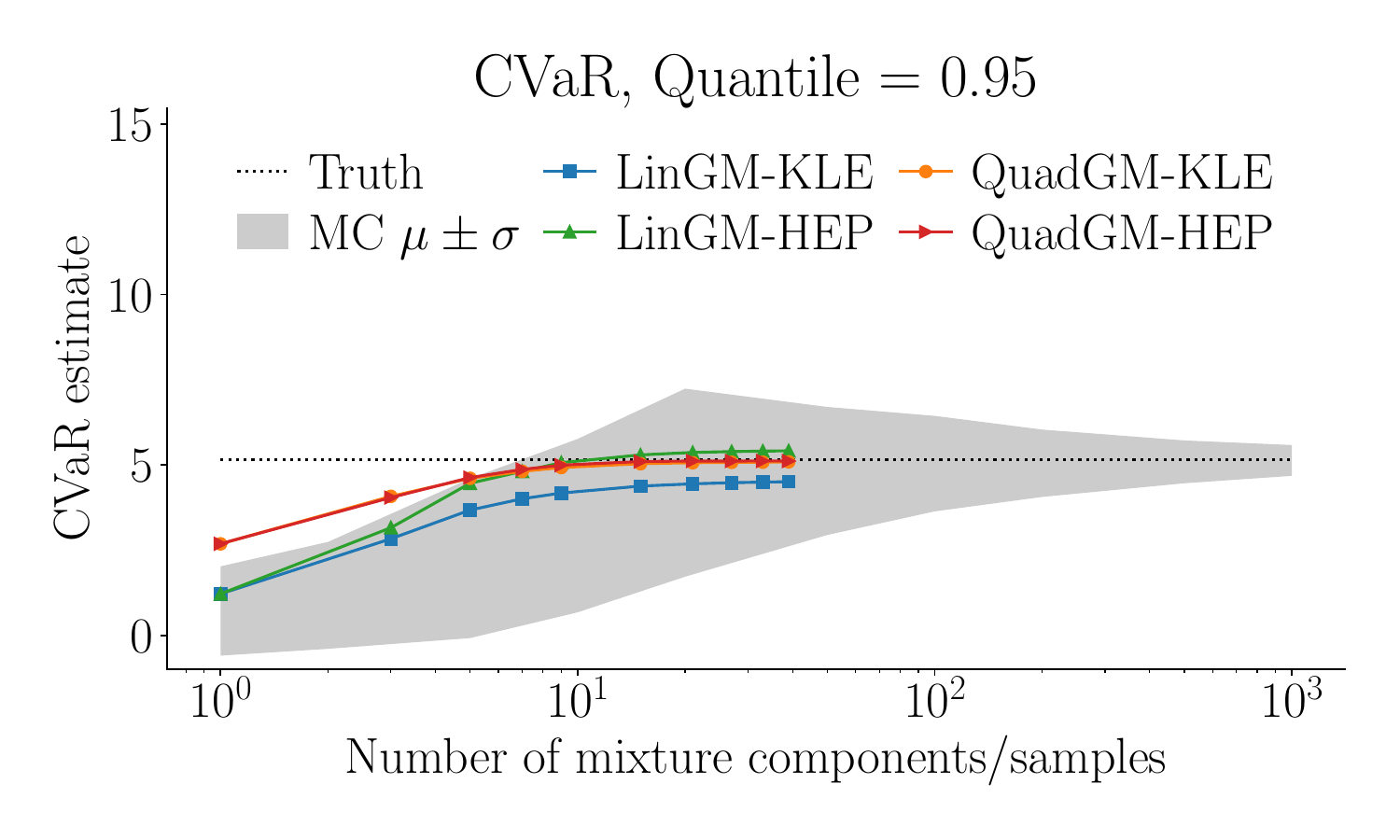}
	\end{subfigure}

	\begin{subfigure}{0.32\textwidth}
	\centering
	\includegraphics[width=\textwidth]{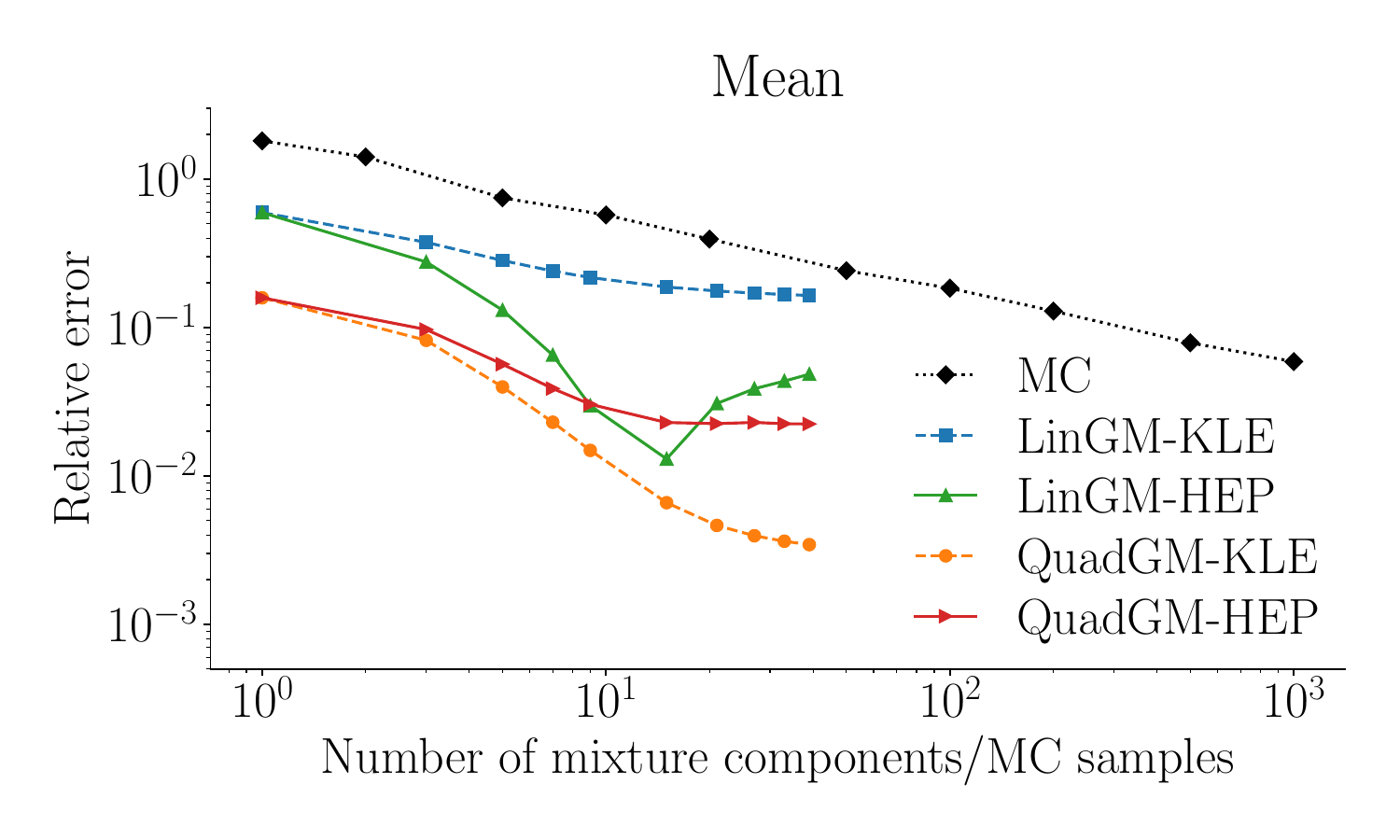}
	\end{subfigure}
	\begin{subfigure}{0.32\textwidth}
	\centering
	\includegraphics[width=\textwidth]{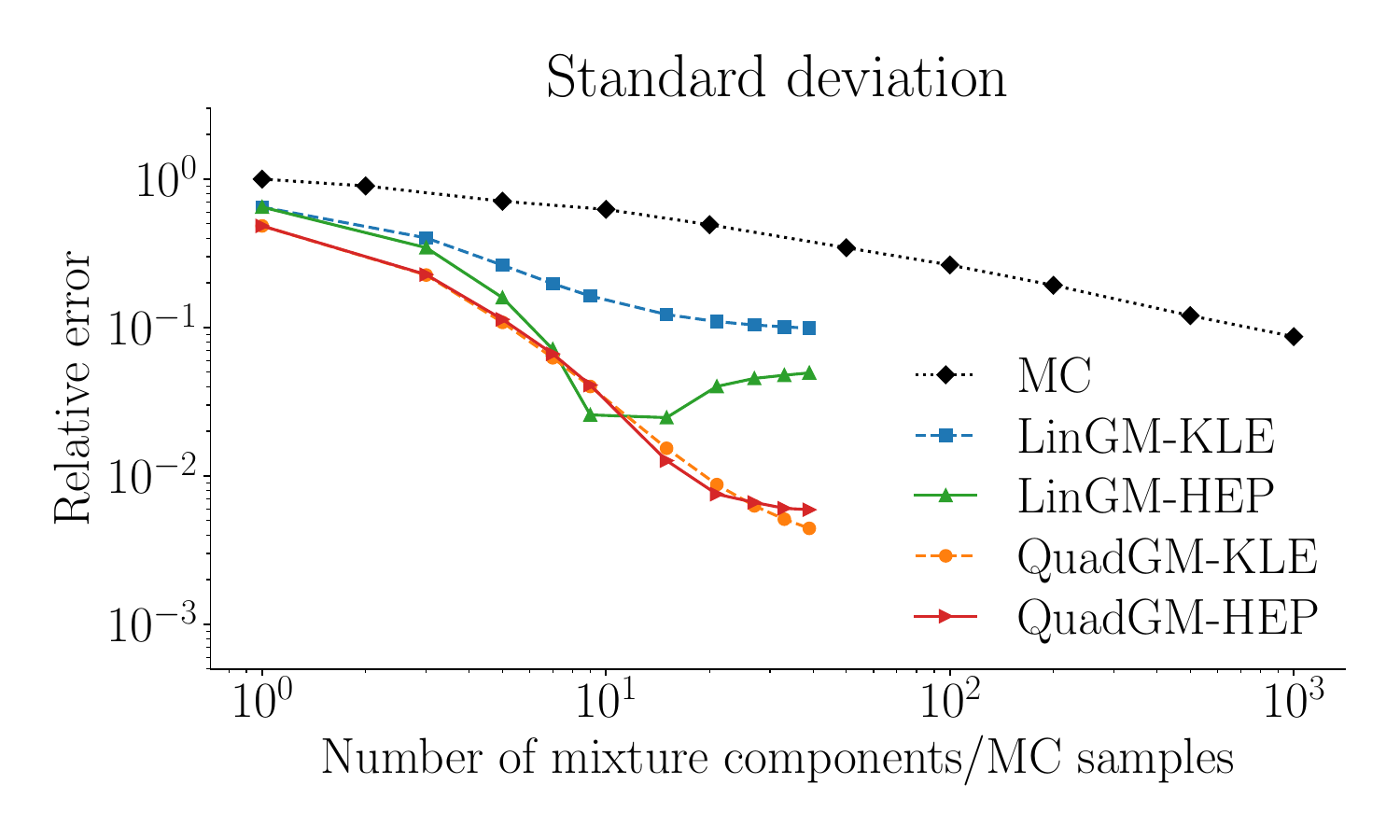}
	\end{subfigure}
	\begin{subfigure}{0.32\textwidth}
	\centering
	\includegraphics[width=\textwidth]{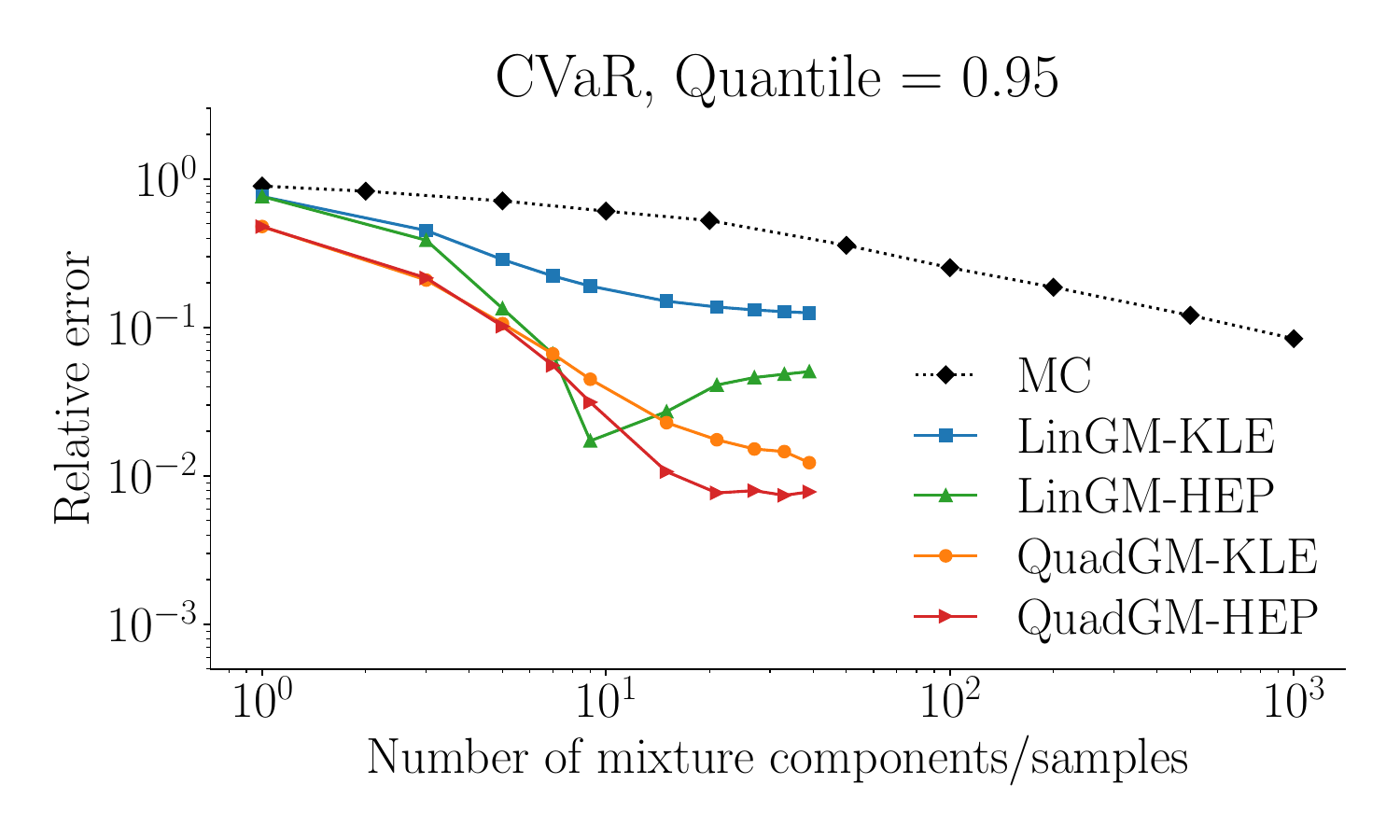}
	\end{subfigure}

	\caption{Mean, standard deviation, and CVaR ($\alpha=0.95$) estimates for the $L^2$ QoI using linear and quadratic 
		Gaussian mixture Taylor approximations with up to 
		$N_{\mix} = 39$ components along the dominant eigendirection. Top row shows the estimates and bottom row shows the relative errors.
		Results are for correlation length and pointwise variance both equal to one.}
	\label{fig:adr_mean_sd_cvar}
\end{figure}

Focusing on the CVaR, we compute estimates for quantile values ranging from $\alpha = 0.9$ to $\alpha = 0.999$, 
representing increasing levels of importance of the upper tails of the distribution. 
Again, considering a covariance with $l_x = 1$ and $\sigma_x = 1$, and all three QoIs,
we show the relative errors of the estimates using $N_{\mix} = 39$ in Figure \ref{fig:adr_quantile_sweep}, 
plotted as a function of the tail size $1-\alpha$.
For comparison, we also plot the relative RMSE of the MC estimates using $10^3$ and $10^4$ samples.

Here we observe that the mixture approximations using the HEP, 
which is informed by both the parameter distribution and the QoI mapping, 
tend to perform better than those obtained from the covariance alone. 
With the HEP mixture, the linear approximation errors are smaller than the RMSE from MC estimates with $10^3$ samples, 
while the quadratic approximation errors are again $<1\%$ and smaller than those from $10^4$ MC samples. 
Additionally, the relative errors appear to be fairly consistent over the range of $\alpha$ considered.

\begin{figure}[htbp!]
	\centering
	\begin{subfigure}{0.32\textwidth}
	\centering
	\includegraphics[width=\textwidth]{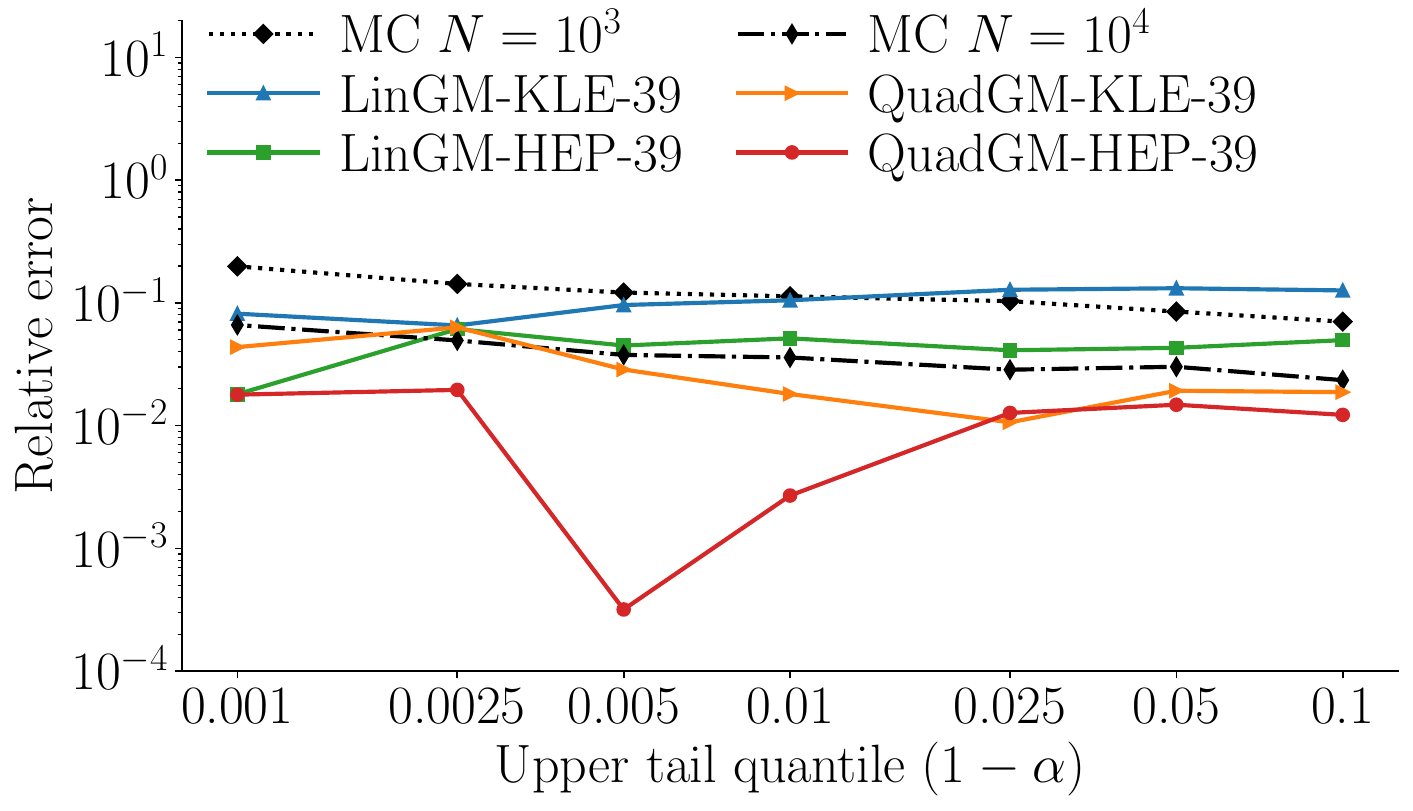}
	\caption*{$L^2$}
	\end{subfigure}
	\begin{subfigure}{0.32\textwidth}
	\centering
	\includegraphics[width=\textwidth]{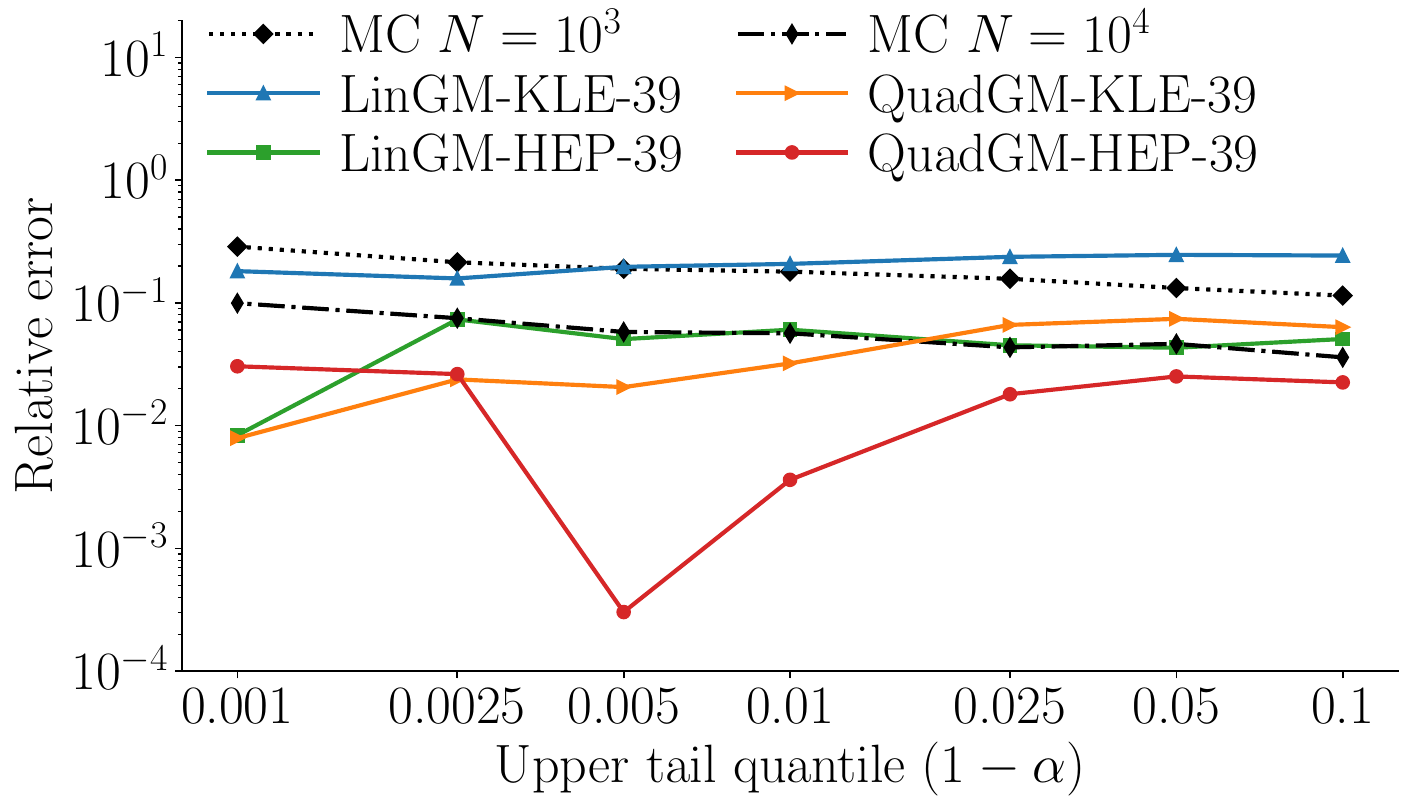}
	\caption*{$L^3$}
	\end{subfigure}
	\begin{subfigure}{0.32\textwidth}
	\centering
	\includegraphics[width=\textwidth]{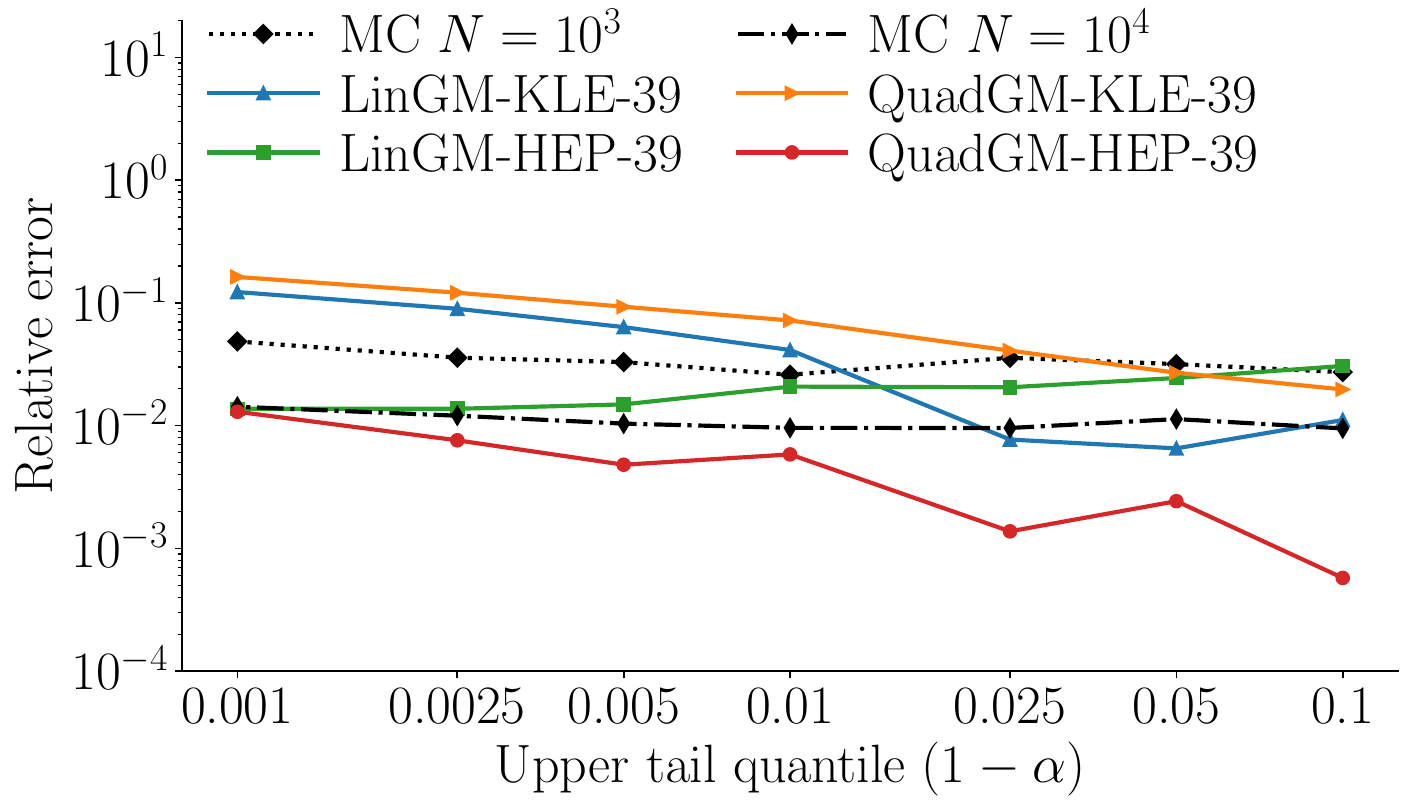}
	\caption*{Energy}
	\end{subfigure}
	\vspace{0pt}
	\caption{Relative errors for estimates of CVaR using 39 mixture components along the dominant KLE or HEP directions as a function of $1-\alpha$, for quantiles $\alpha$ from $0.95$ to $0.999$. These are compared against MC estimates using $10^3$ and $10^4$ samples. Results are shown for the $L^2$ (left), $L^3$ (center), and energy (right) QoIs, with correlation length and pointwise variance both equal to 1.}
	\label{fig:adr_quantile_sweep}
\end{figure}

We also demonstrate the effects of changing the correlation length $l_x$ and pointwise variance $\sigma_x^2$ of the distribution. 
First, for a fixed pointwise variance of $\sigma_x = 1$, we consider correlation lengths from $l_x = 0.1$ to $l_x = 2$, and compute risk measure estimates for the three QoI. 
In Figure~\ref{fig:adr_corr_sweep}, we present the results for CVaR with $\alpha = 0.95$, plotting the relative errors of the mixture Taylor approximations with $N_{\mix} = 39$, along with the relative RMSEs of MC estimates using $10^3$ and $10^4$ samples for comparison. 
It is again evident that the HEP based mixture Taylor approximations consistently perform better than their KLE counterparts, 
with the quadratic approximations yielding $<1\%$ errors, lower than those obtained using $10^4$ MC samples. 

We also consider changing the pointwise variances from $\sigma_x^2 = 0.1$ to $\sigma_x^2 = 2$ while fixing the correlation length to be $l_x = 1$. 
The relative errors are presented in Figure \ref{fig:adr_var_sweep}. For simplicity, we shows only the HEP-based mixture Taylor approximations with $N_{\mix} = 39$. Instead, we plot the errors of standard linear and quadratic Taylor approximations without the Gaussian mixture. 
The figure illustrates the significant improvement in accuracy by using the mixture approximation compared to using only a single Taylor approximation.
We remark that there is a small increase in error with increasing pointwise variance, particularly for the energy QoI, alluding to a possible limitation for applications to QoIs with much larger variance.

\begin{figure}[htbp!]
	\centering
	\begin{subfigure}{0.32\textwidth}
	\centering
	\includegraphics[width=\textwidth]{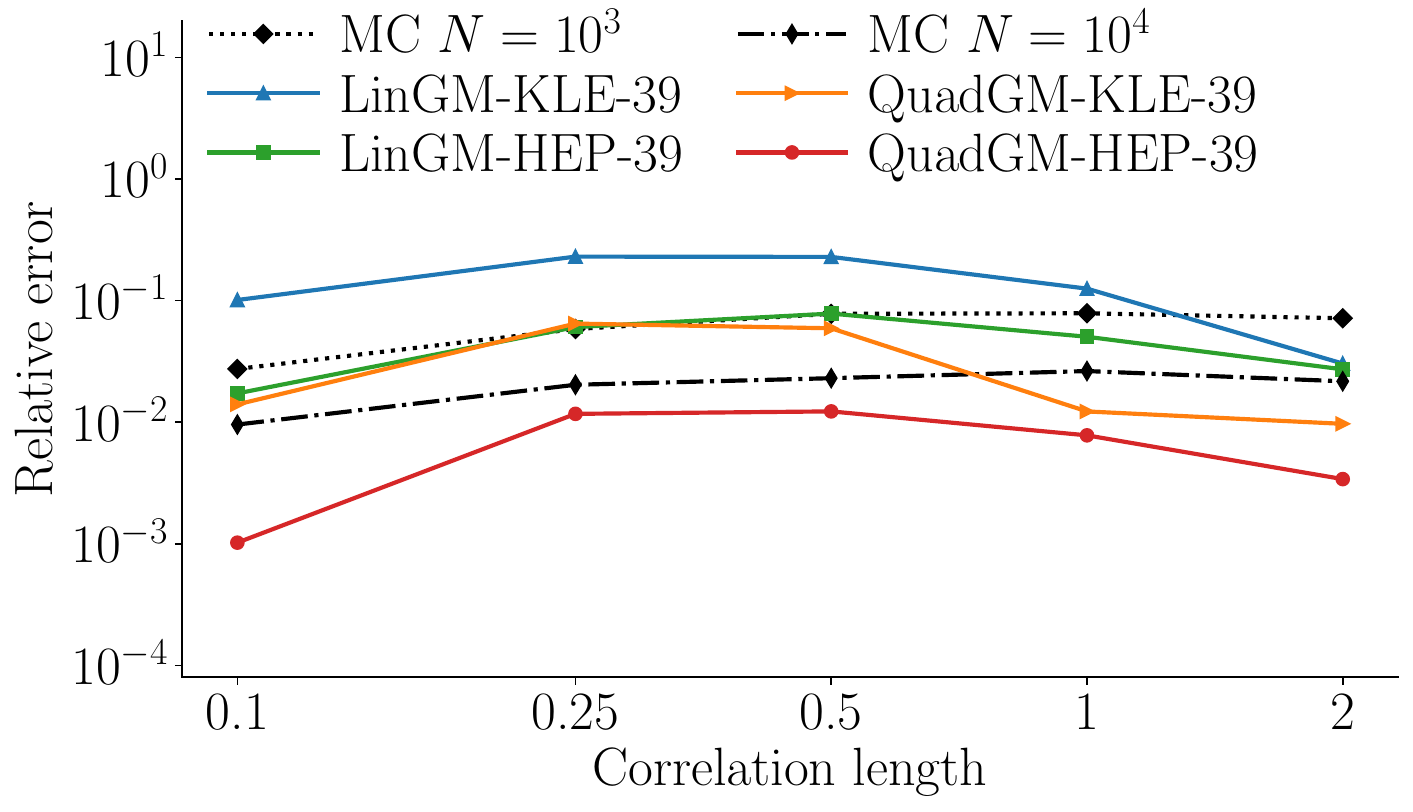}
	\caption*{$L^2$}
	\end{subfigure}
	\begin{subfigure}{0.32\textwidth}
	\centering
	\includegraphics[width=\textwidth]{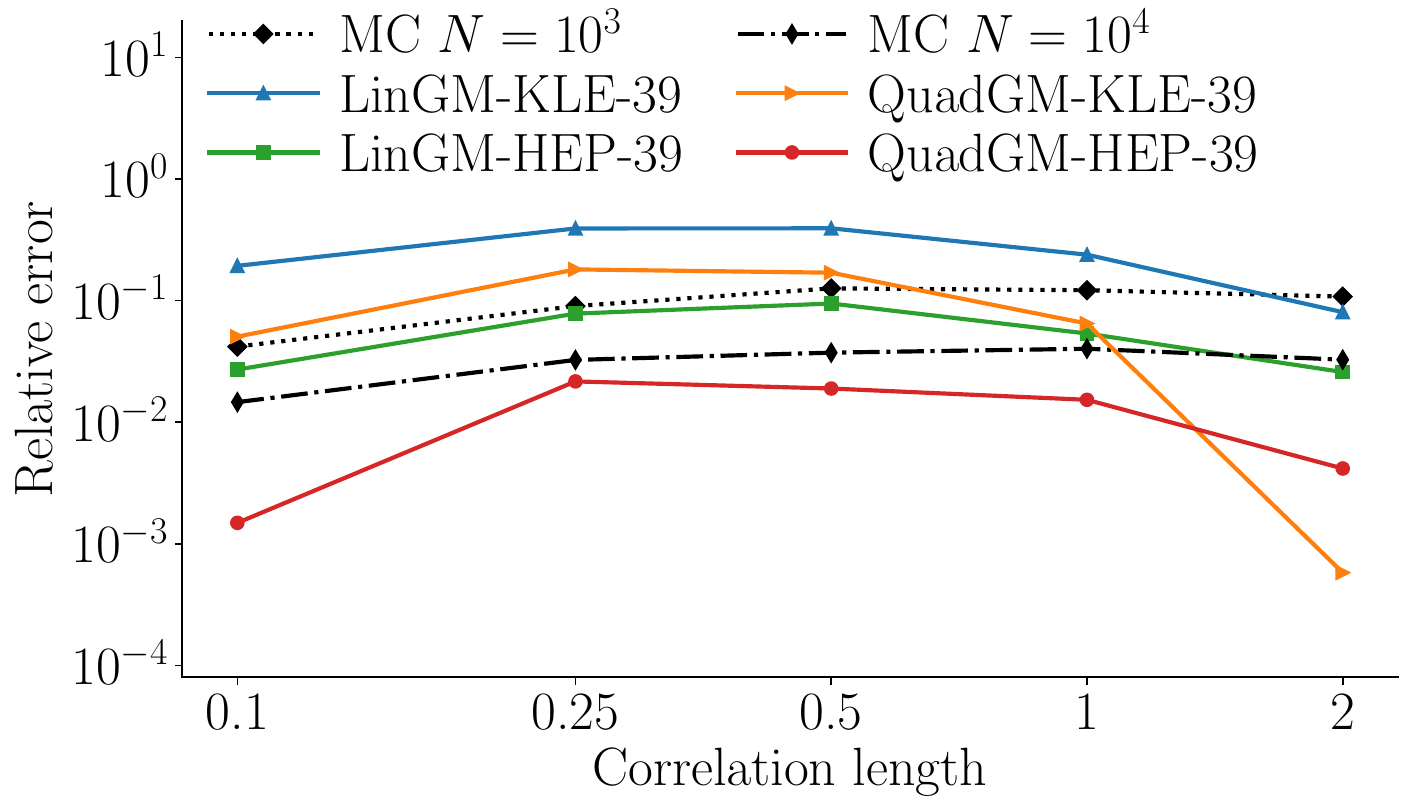}
	\caption*{$L^3$}
	\end{subfigure}
	\begin{subfigure}{0.32\textwidth}
	\centering
	\includegraphics[width=\textwidth]{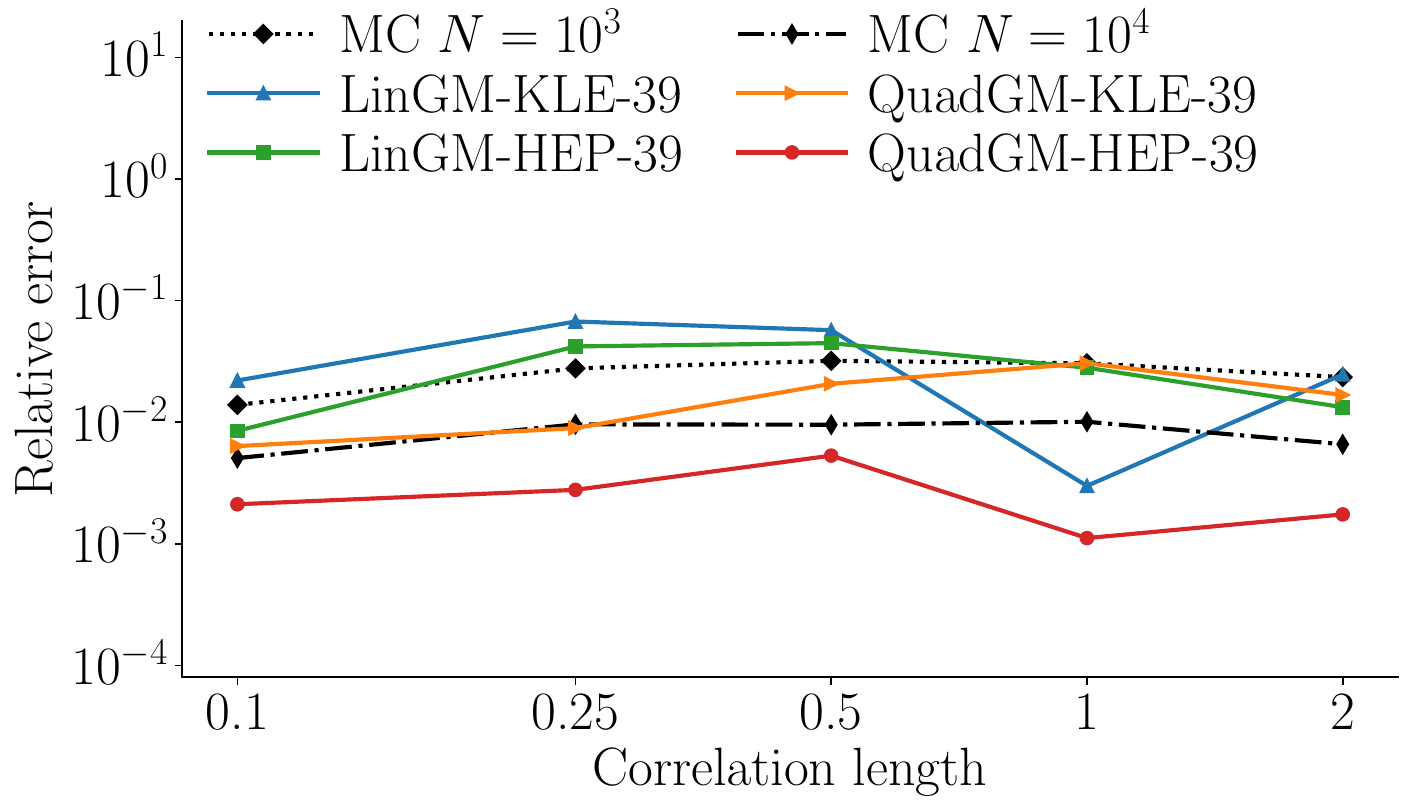}
	\caption*{Energy}
	\end{subfigure}
	\vspace{0pt}
	\caption{Relative errors for estimates of CVaR ($\alpha = 0.95$) using 39 mixture components along the dominant KLE or HEP directions for correlation lengths from $0.1$ to $2$. These are compared against MC estimates using $10^3$ and $10^4$ samples. Results are shown for the $L^2$ (left), $L^3$ (center), and energy (right) QoIs.}
	\label{fig:adr_corr_sweep}
\end{figure}

\begin{figure}[htbp!]
	\centering
	\begin{subfigure}{0.32\textwidth}
	\centering
	\includegraphics[width=\textwidth]{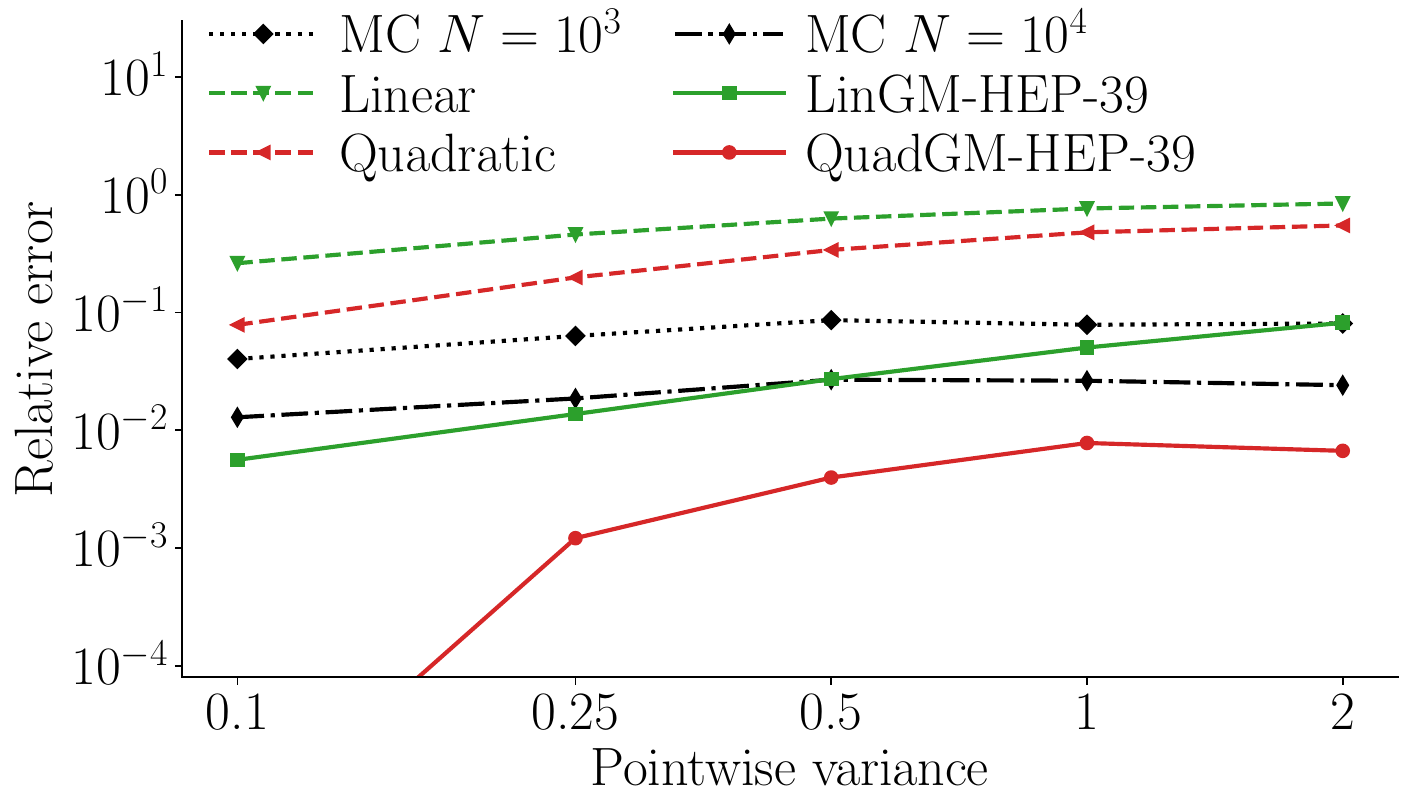}
	\caption*{$L^2$}
	\end{subfigure}
	\begin{subfigure}{0.32\textwidth}
	\centering
	\includegraphics[width=\textwidth]{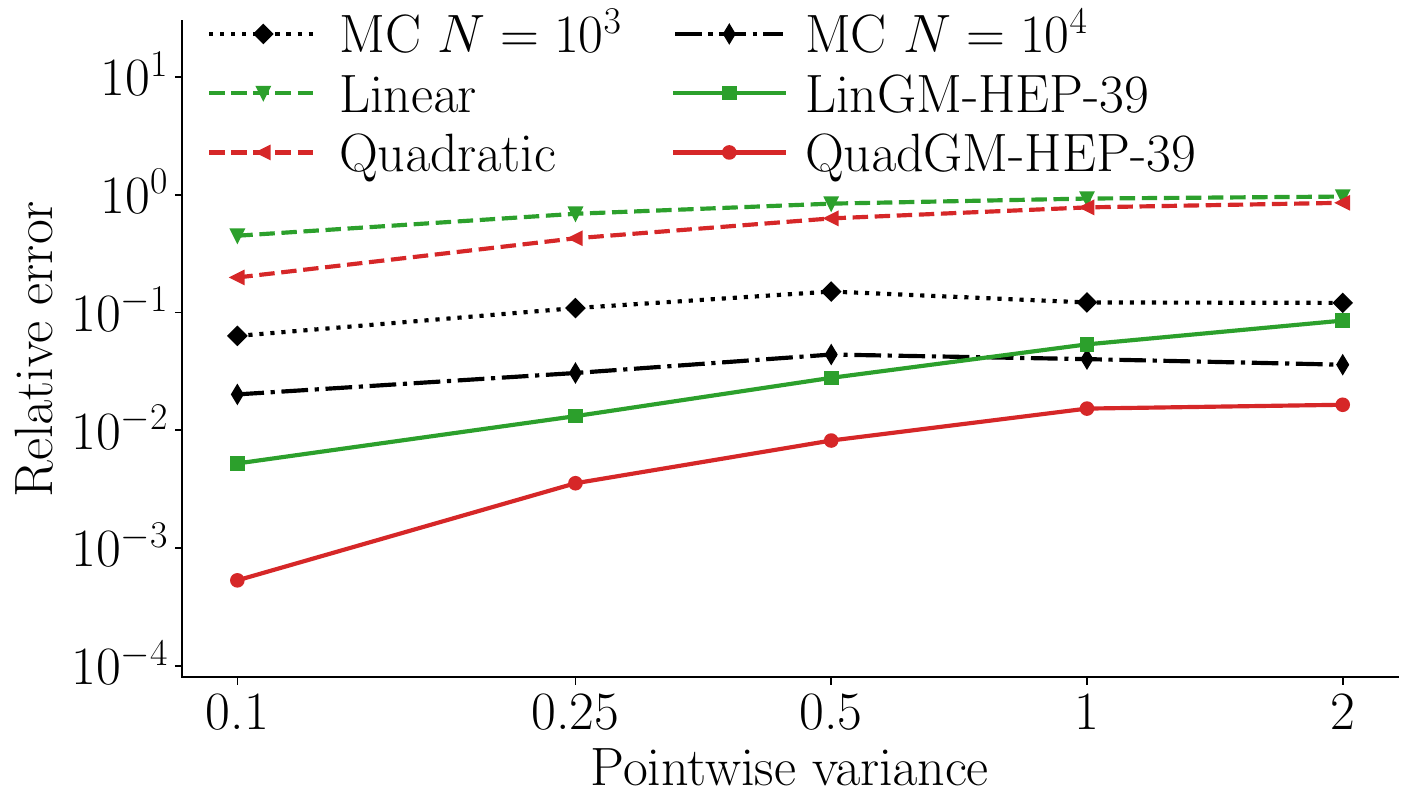}
	\caption*{$L^3$}
	\end{subfigure}
	\begin{subfigure}{0.32\textwidth}
	\centering
	\includegraphics[width=\textwidth]{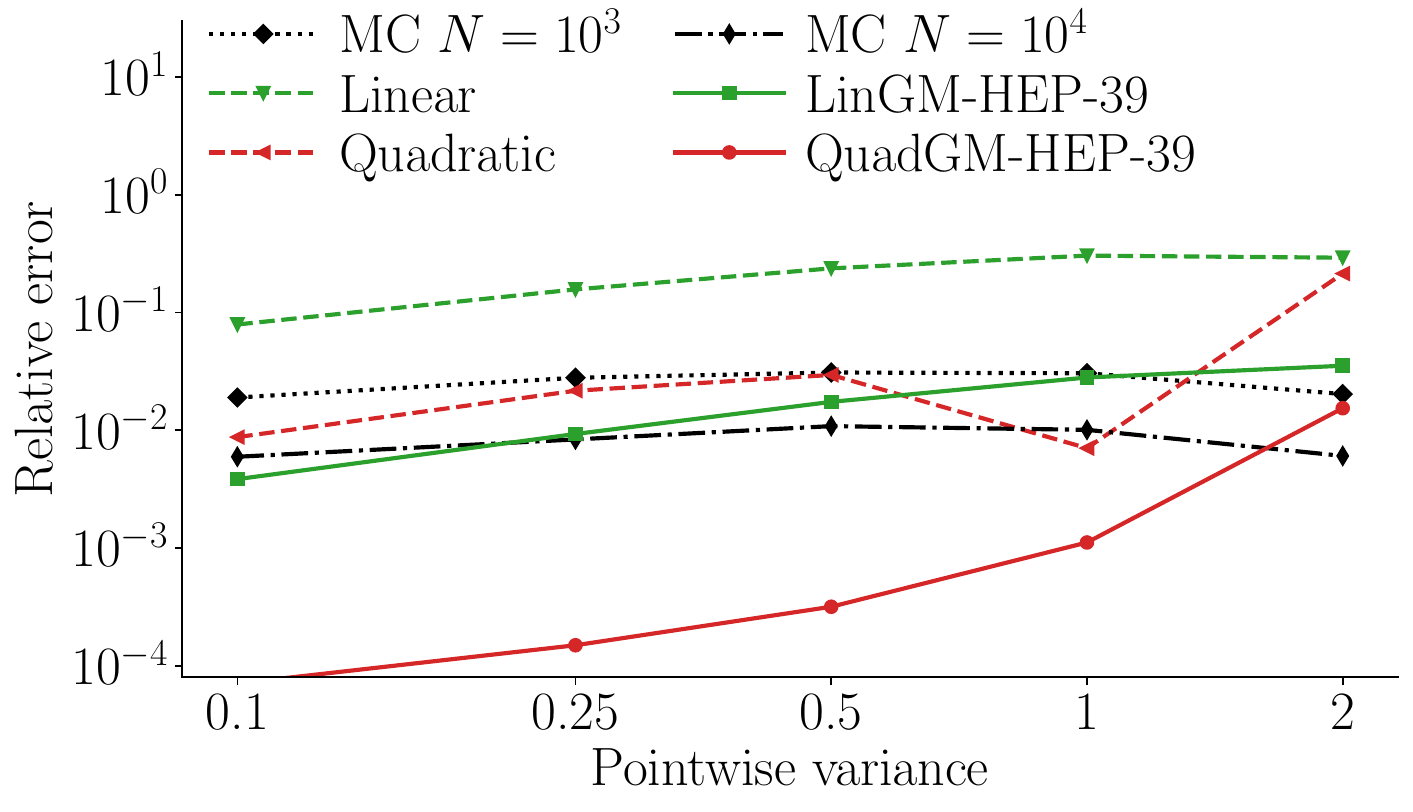}
	\caption*{Energy}
	\end{subfigure}
	\vspace{0pt}
	\caption{Relative errors for estimates of CVaR ($\alpha = 0.95$) using 39 mixture components along the dominant HEP directions for pointwise variances from $0.1$ to $2$. These are compared against MC estimates using $10^3$ and $10^4$ samples. Results are shown for the $L^2$ (left), $L^3$ (center), and energy (right) QoIs.}
	\label{fig:adr_var_sweep}
\end{figure}

\subsection{Helmholtz equation with uncertain wavespeeds}
We next consider numerical experiments for the scattering of acoustic waves around a circular obstacle with boundary $\partial \cD_{\text{obstacle}}$ surrounded by a heterogeneous medium with uncertain material properties.
The problem is governed by the Helmholtz equation,
\begin{align}\label{eq:helmholtz}
	- \Delta u + k^2 u - (k^2 - k_0^2) u_{\mathrm{inc}} &= 0 \qquad \text{in } \cD, \\
	\nabla u \cdot \bs{n} - \nabla u_{\mathrm{inc}} \cdot \bs{n}  &= 0 \qquad \text{on } \partial \cD_{\text{obstacle}}, \\
	\lim_{r \rightarrow \infty} r^{(d-1)/2}\left(\frac{\partial u}{\partial r} - iku\right) &= 0. \label{eq:helmholtz_radiation_condition}
\end{align}
Here, $i$ denotes the imaginary unit, $r = \sqrt{x_1^2 + x_2^2}$, and $u^{\text{inc}} = \exp(i k_0 x \cdot \bs{e}_1)$ is the incident wave with wavenumber $k_0 = 2\pi$ that is propagating in the $\bs{e}_1 = (1,0)$ direction. The wavenumber of the medium is given by 
\begin{equation}
	k(x) = \bar{k}(x)\exp(\mathbbm{1}_{A_R}(x) m(x)),
\end{equation}
where $A_R$ is the annulus with outer radius $R_O = 2$ and inner radius $R_I = 1$, $\mathbbm{1}_{A_R}$ is the indicator function for $A_R$, and $\bar{k}$ is the nominal wavenumber, given by
\begin{equation}
	\bar{k}(x) = \begin{cases}
		k_0 (1 - (r(x) - R_o)^2/8) & \text{if } x \in A_R, \\
		k_0 & \text{if } x \in \cD \setminus A_R.
	\end{cases}
\end{equation}
This represents a medium that changes within the annulus, which is further perturbed by a multiplicative random field within the annulus. 
The condition \eqref{eq:helmholtz_radiation_condition} is the Sommerfield radiation condition, which ensures an outgoing scattered wave. 
We discretize the PDE over a square computational domain with a circular obstacle of radius $r = 1$ centered at the origin, $\cD = [-6, 6]^2 \setminus B_1(0)$, and make use of a perfectly matched layer (PML) of unit length to enforce the radiation condition. 
We consider the QoI to be the total intensity of the scattered wave,
\begin{equation}
	Q = \int_{\cD}|u|^2 \, dx = \int_{\cD} u_r^2 + u_i^2 \dx,
\end{equation}
noting that $u$ is complex valued with real and imaginary components $u_r$ and $u_i$. 

In our examples, we consider $(\gamma, \delta) = (4,8)$, $(8,8)$, and $(8,4)$ for the covariance $\cC = (\delta -\gamma \Delta)^{-2}$.
Figure \ref{fig:helmholtz_example} shows the mean wavenumber $\bar{k}$, a sample of the random parameter field $m$ for $(\gamma, \delta) = (4, 8)$, and the corresponding real and complex parts of the scattered wave. 
We also plot the spectra from the KLE and HEP in Figure \ref{fig:helmholtz_eigenvectors}.
Again we note the visual differences between the dominant eigenmodes of the covariance and Hessian operators, where the covariance eigenvector is radially uniform, while the Hessian eigenvector additionally counts for the dependence of the QoI on the parameter.

\begin{figure}[h!]
	\centering
	\begin{subfigure}{0.24\textwidth}
		\centering
		\includegraphics[width=\textwidth]{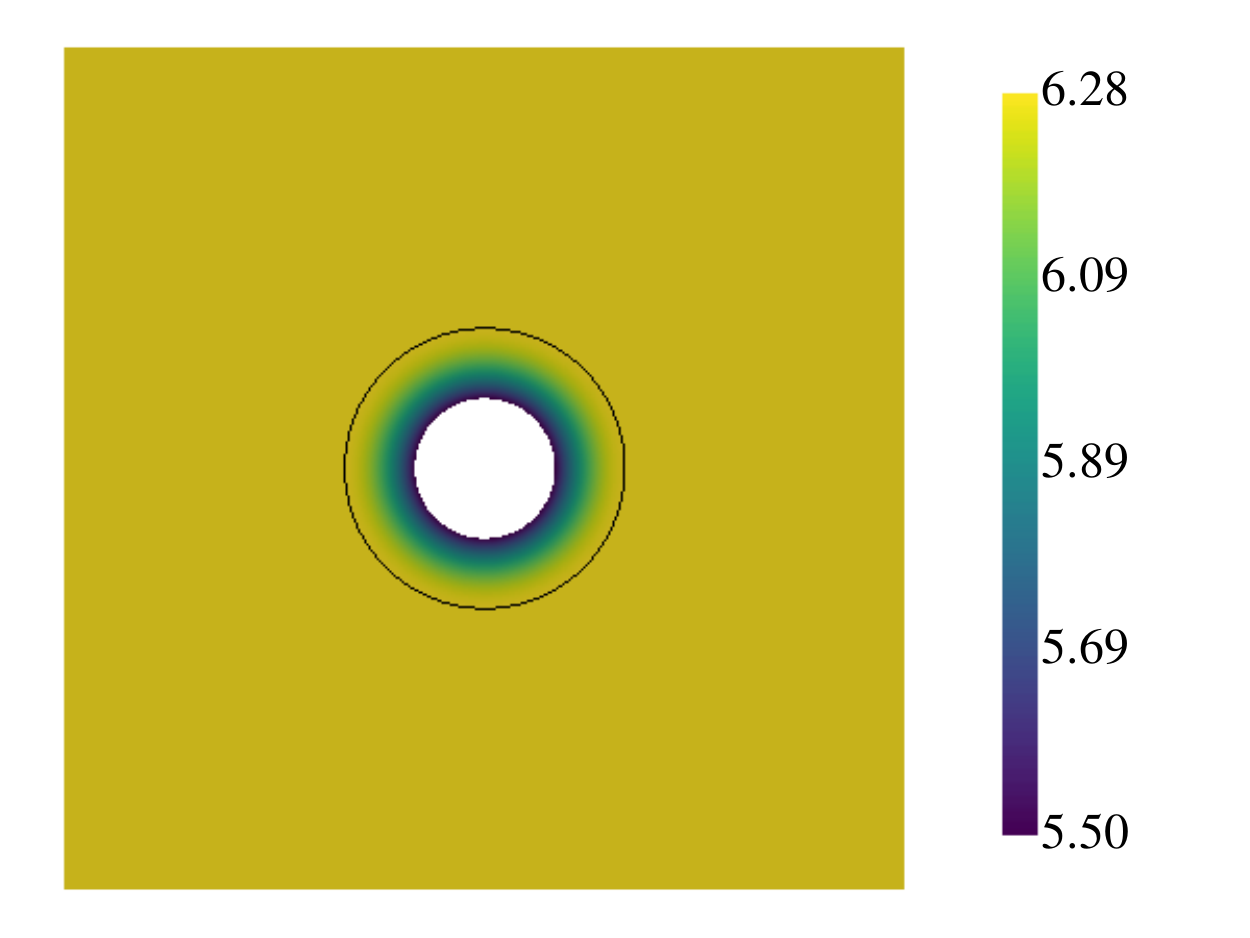}
		\caption{$\bar{k}$}
	\end{subfigure}
	\begin{subfigure}{0.24\textwidth}
		\centering
		\includegraphics[width=\textwidth]{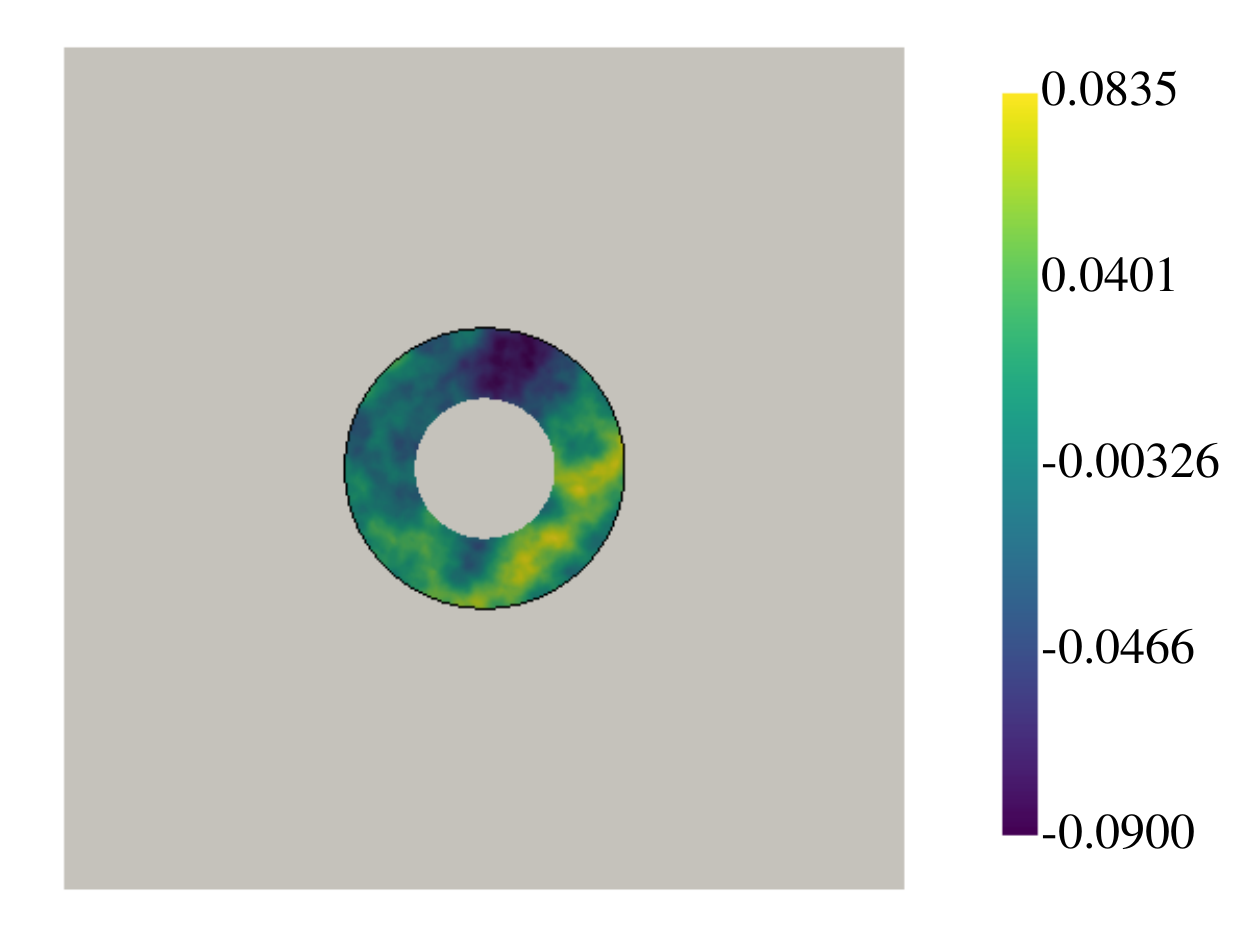}
		\caption{$m$ sample}
	\end{subfigure}
	\begin{subfigure}{0.24\textwidth}
		\centering
		\includegraphics[width=\textwidth]{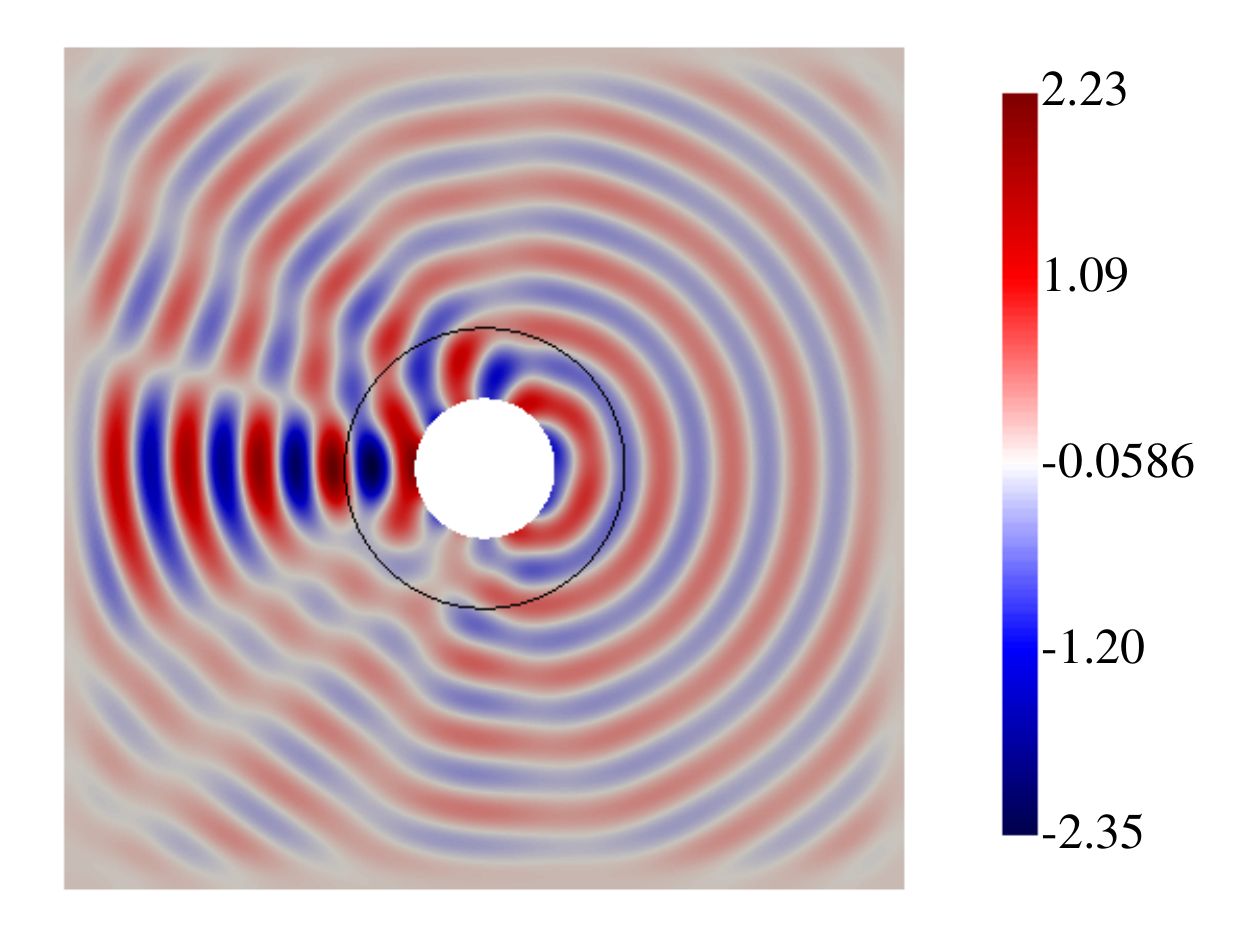}
		\caption{$u_r$ sample}
	\end{subfigure}
	\begin{subfigure}{0.24\textwidth}
		\centering
		\includegraphics[width=\textwidth]{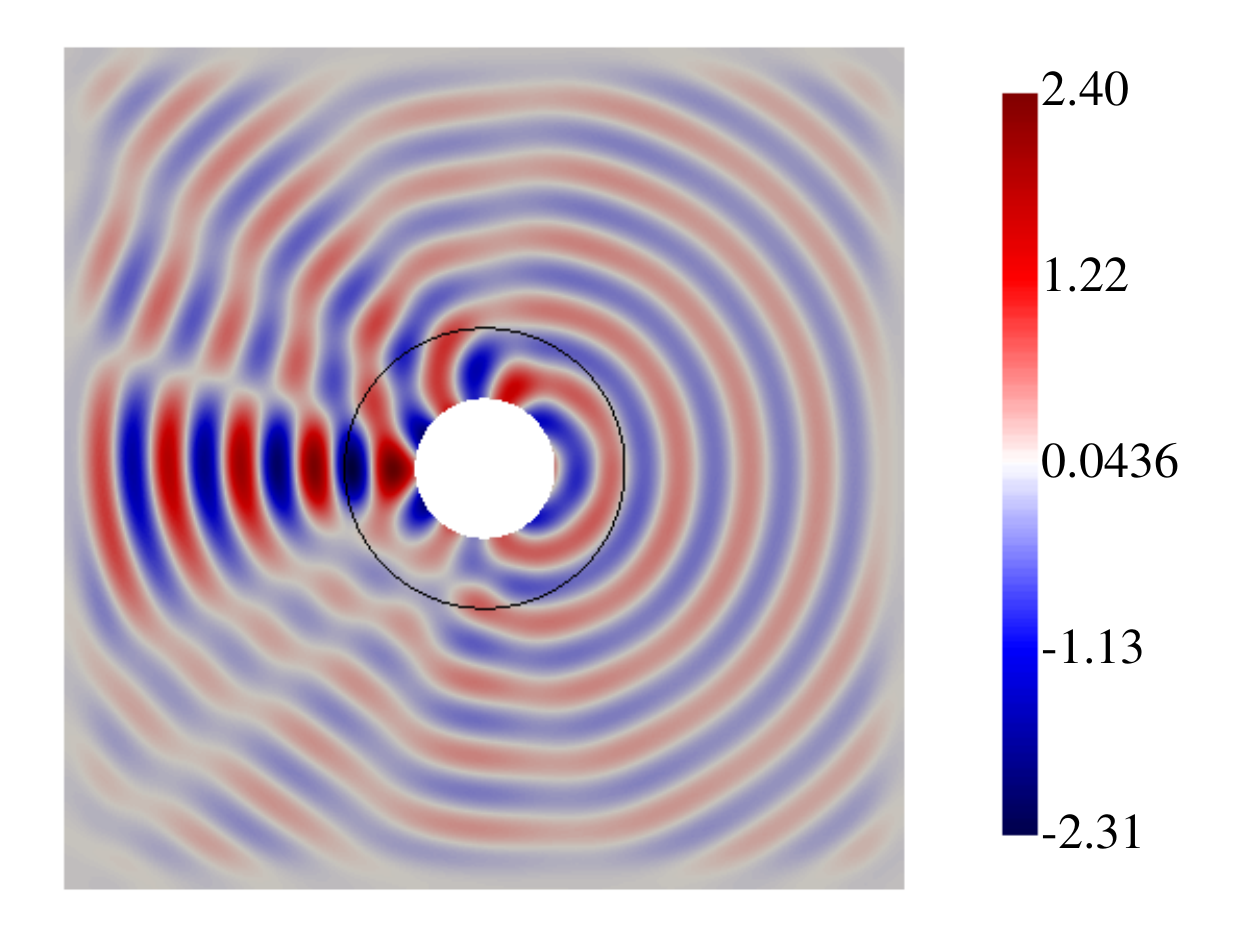}
		\caption{$u_i$ sample}
	\end{subfigure}
	\vspace{0pt}
	\caption{The mean wavenumber $\bar{k}$, a sample of the random parameter field $m$, and the corresponding real and complex parts of the solution at the given sample.}
	\label{fig:helmholtz_example}
\end{figure}

\begin{figure}[h!]
	\centering
	\begin{subfigure}{0.24\textwidth}
		\centering
		\includegraphics[width=\textwidth]{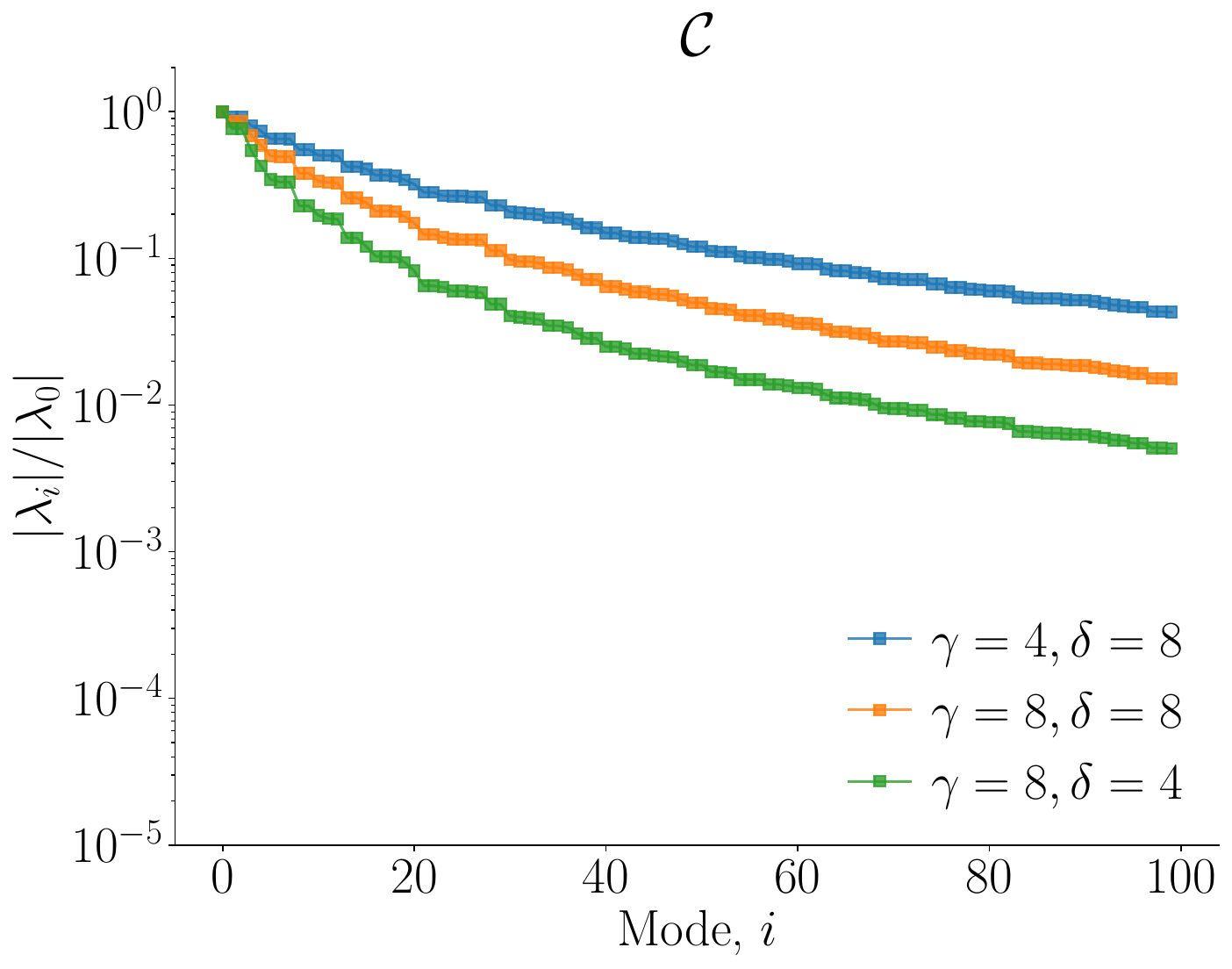}
		\caption{KLE spectrum}
	\end{subfigure}
	\begin{subfigure}{0.24\textwidth}
		\centering
		\includegraphics[width=\textwidth]{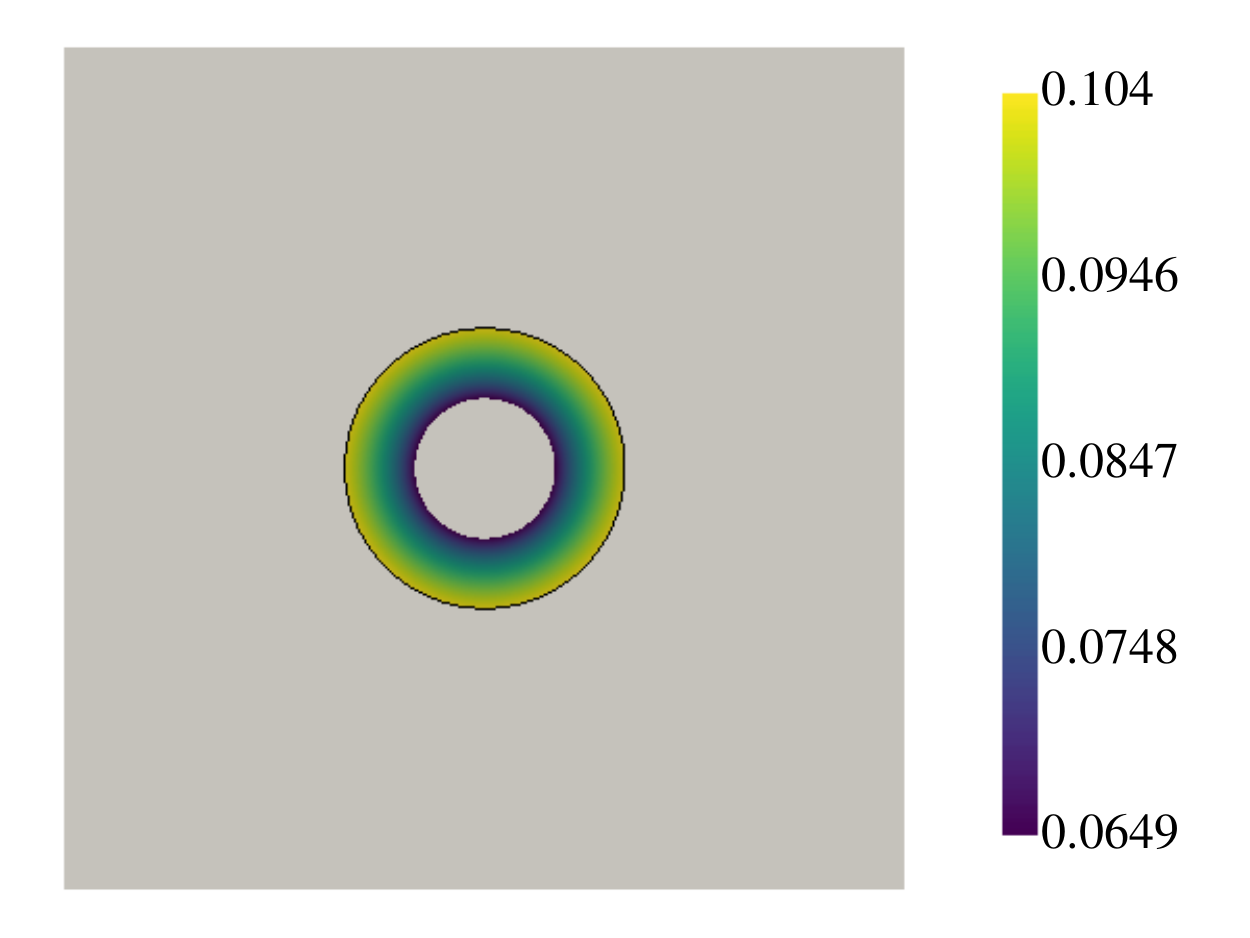}
		\caption{KLE eigenvector}
	\end{subfigure}
	\begin{subfigure}{0.24\textwidth}
		\centering
		\includegraphics[width=\textwidth]{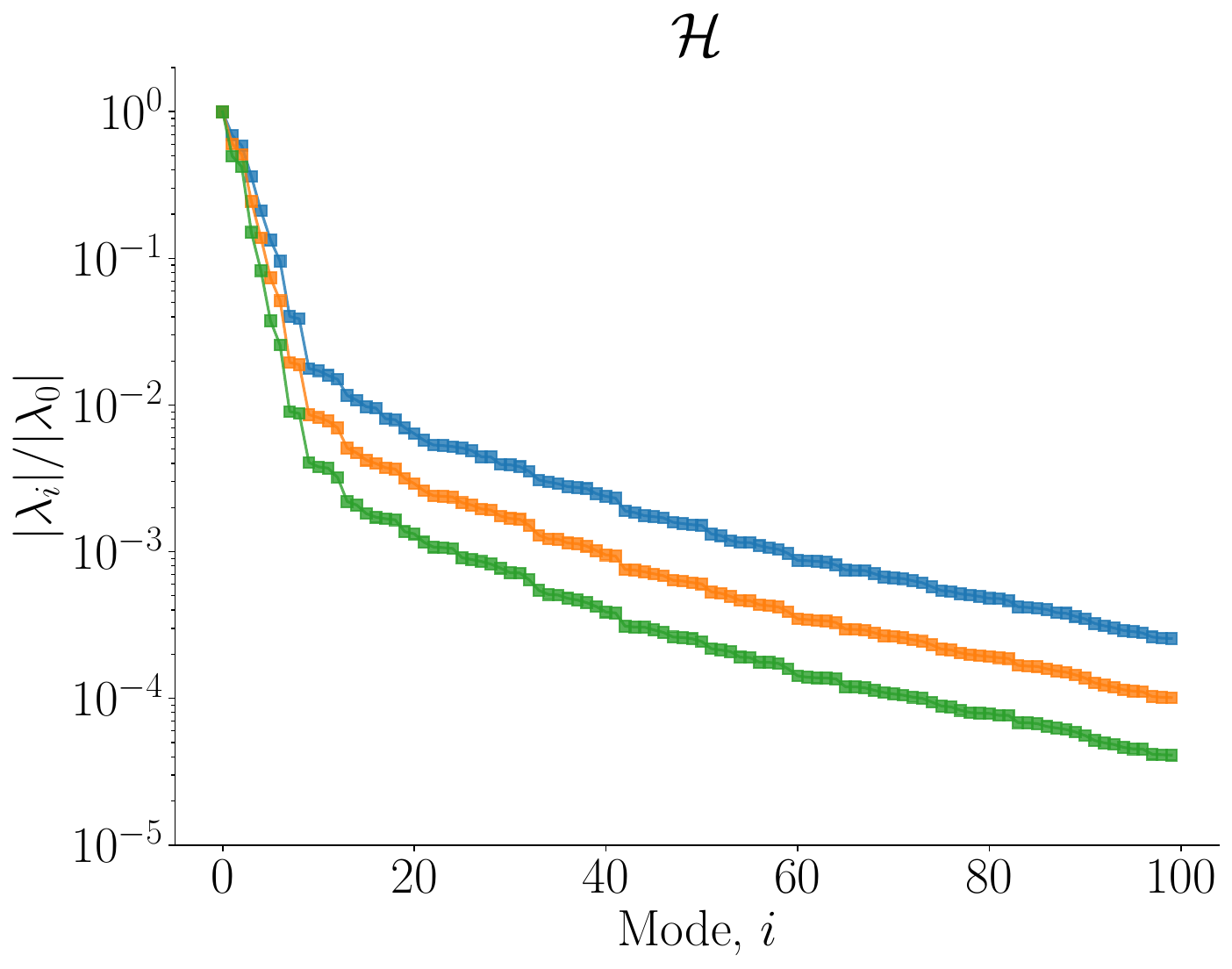}
		\caption{HEP spectrum}
	\end{subfigure}
	\begin{subfigure}{0.24\textwidth}
		\centering
		\includegraphics[width=\textwidth]{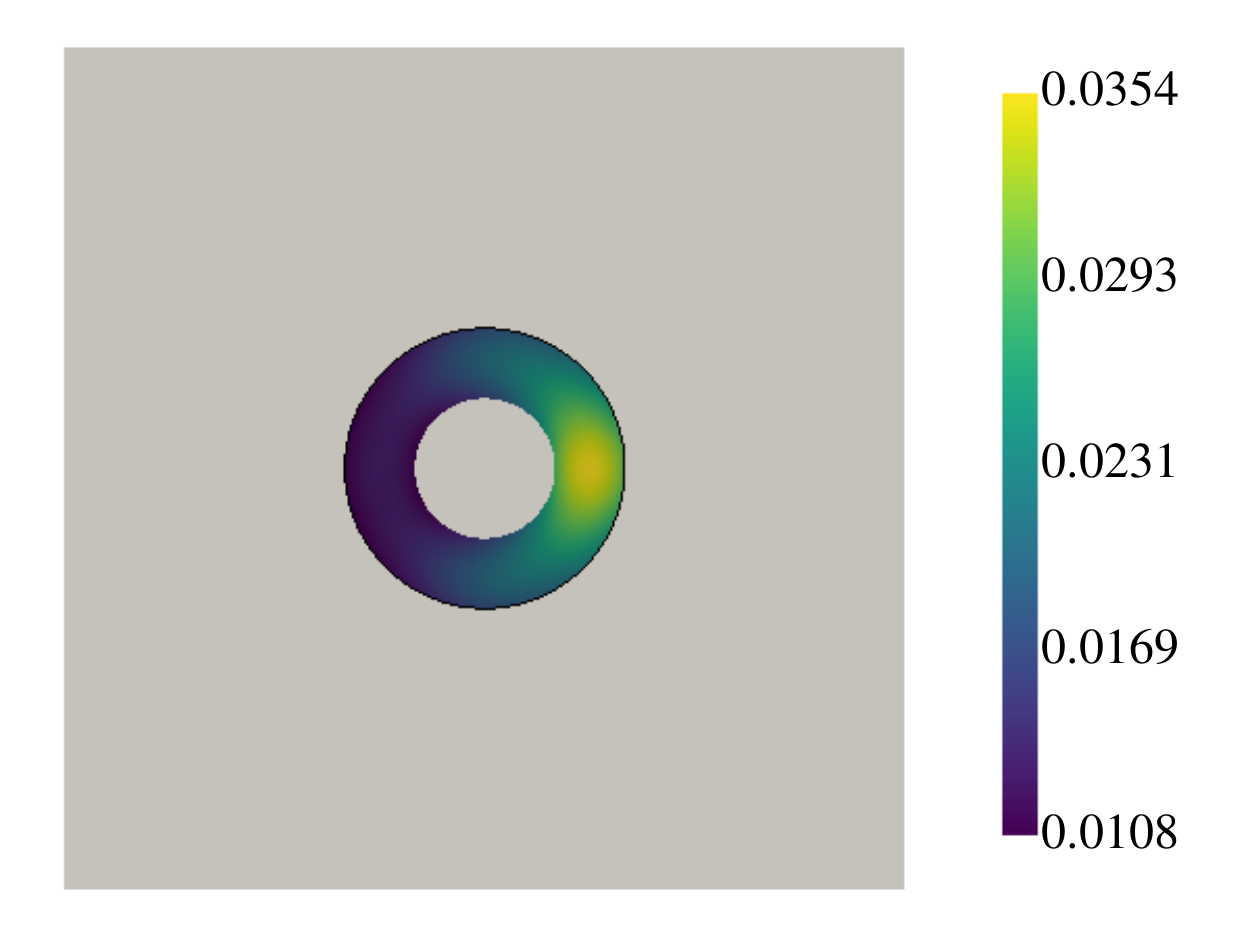}
		\caption{HEP eigenvector}
	\end{subfigure}
	\vspace{0pt}
	\caption{Spectra and dominant eigenvector of the covariance (left) and covariance-preconditioned Hessian at $\bar{m}$ for the Helmholtz QoI (right). Eigenvectors shown are for the case of $(\gamma, \delta) = (4,8)$.}
	\label{fig:helmholtz_eigenvectors}
\end{figure}

We focus on the CVaR, computing estimates for quantiles from $\alpha = 0.95$ to $\alpha = 0.999$ using Gaussian mixture Taylor approximations using up to 39 mixture components along the dominant eigenvectors of $\cC$ and $\cH$. 
We again compare this against ``ground truth'' estimates by MC sampling using $10^5$ samples. 
The relative errors of Gaussian mixture Taylor approximations for CVaR with $\alpha = 0.95$ are shown in Figure \ref{fig:helmholtz_cvar_results} for the three cases of $(\gamma, \delta) = (4, 8)$, $(8, 8)$, and $(8, 4)$. For reference, we also plot the relative RMSE of MC estimates using up to $10^3$ samples.
In this example, we observe that the KLE-based mixture approximations do not yield significant improvements over the standard Taylor approximation. 
On the other hand, the HEP-based mixture approximations yield much larger improvements, and are 1--2 orders of magnitude more accurate than MC estimates using the same number of state PDE solves. 

We also present the relative errors for the range of quantile values in Figure \ref{fig:helmholtz_quantile_sweep}. 
These are again shown for the mixture Taylor approximations with $N_{\mix} = 39$, as well as the relative RMSEs of MC estimates using $10^3$ and $10^4$ samples.
Here, the HEP is also shown to be more effective than the KLE for constructing the mixture approximation. 
In particular, the quadratic approximations using the HEP eigenvectors yield $<1\%$ errors across the tests considered, and consistently yield smaller errors than the MC estimates using $10^4$ samples.

\begin{figure}[h!]
	\centering
	\begin{subfigure}{0.32\textwidth}
		\centering
		\includegraphics[width=\textwidth]{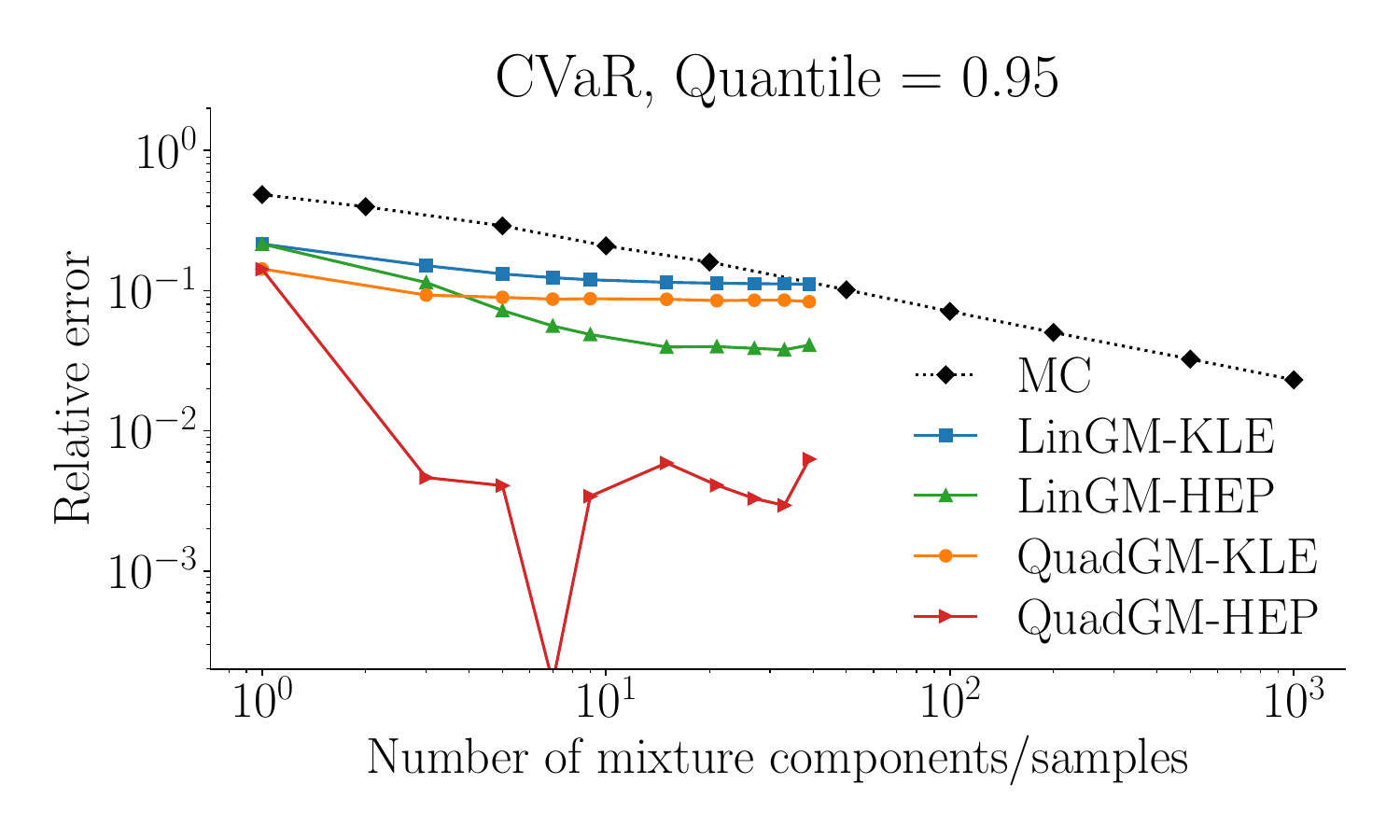}
		\caption{$\gamma=4, \delta=8$}
	\end{subfigure}
	\begin{subfigure}{0.32\textwidth}
		\centering
		\includegraphics[width=\textwidth]{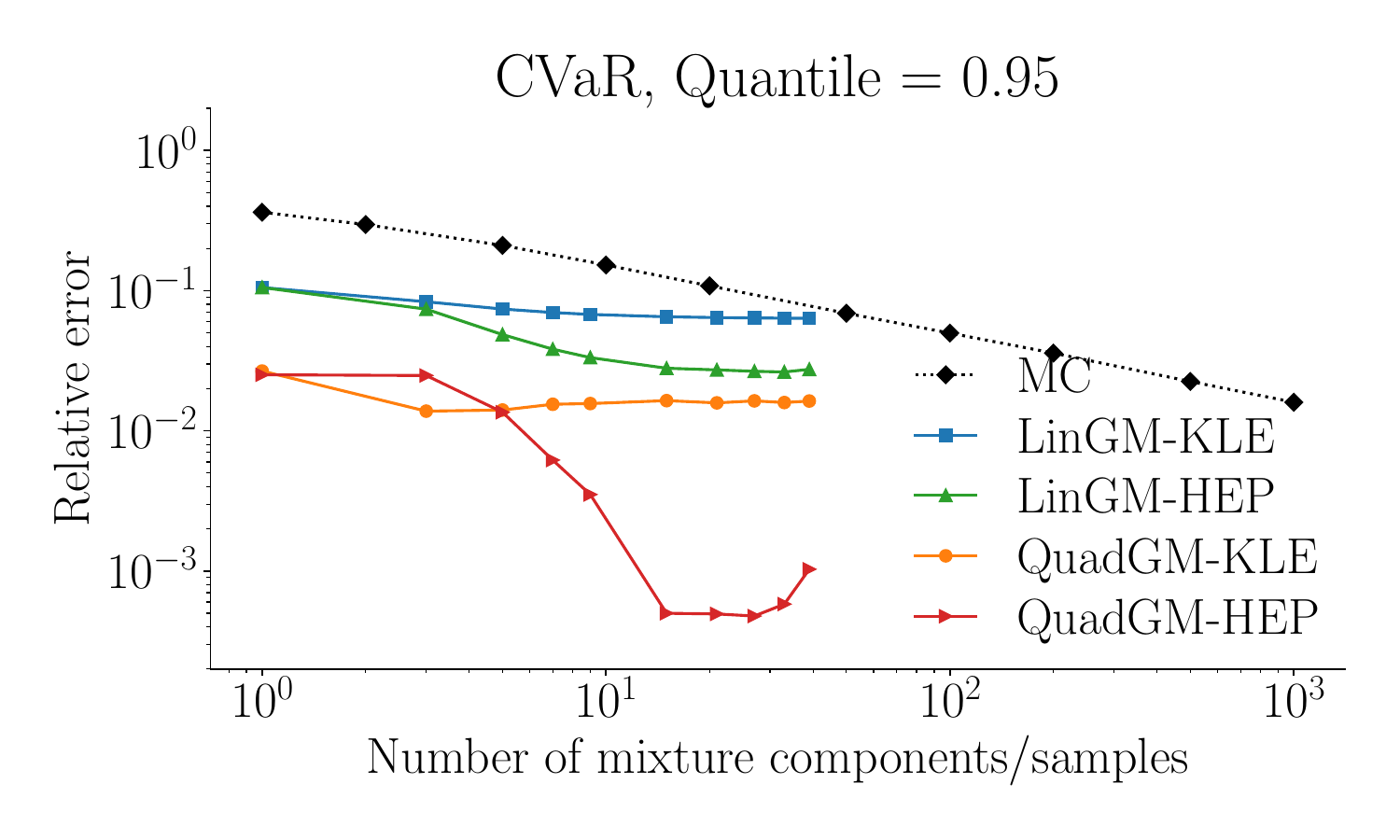}
		\caption{$\gamma=8, \delta=8$}
	\end{subfigure}
	\begin{subfigure}{0.32\textwidth}
		\centering
		\includegraphics[width=\textwidth]{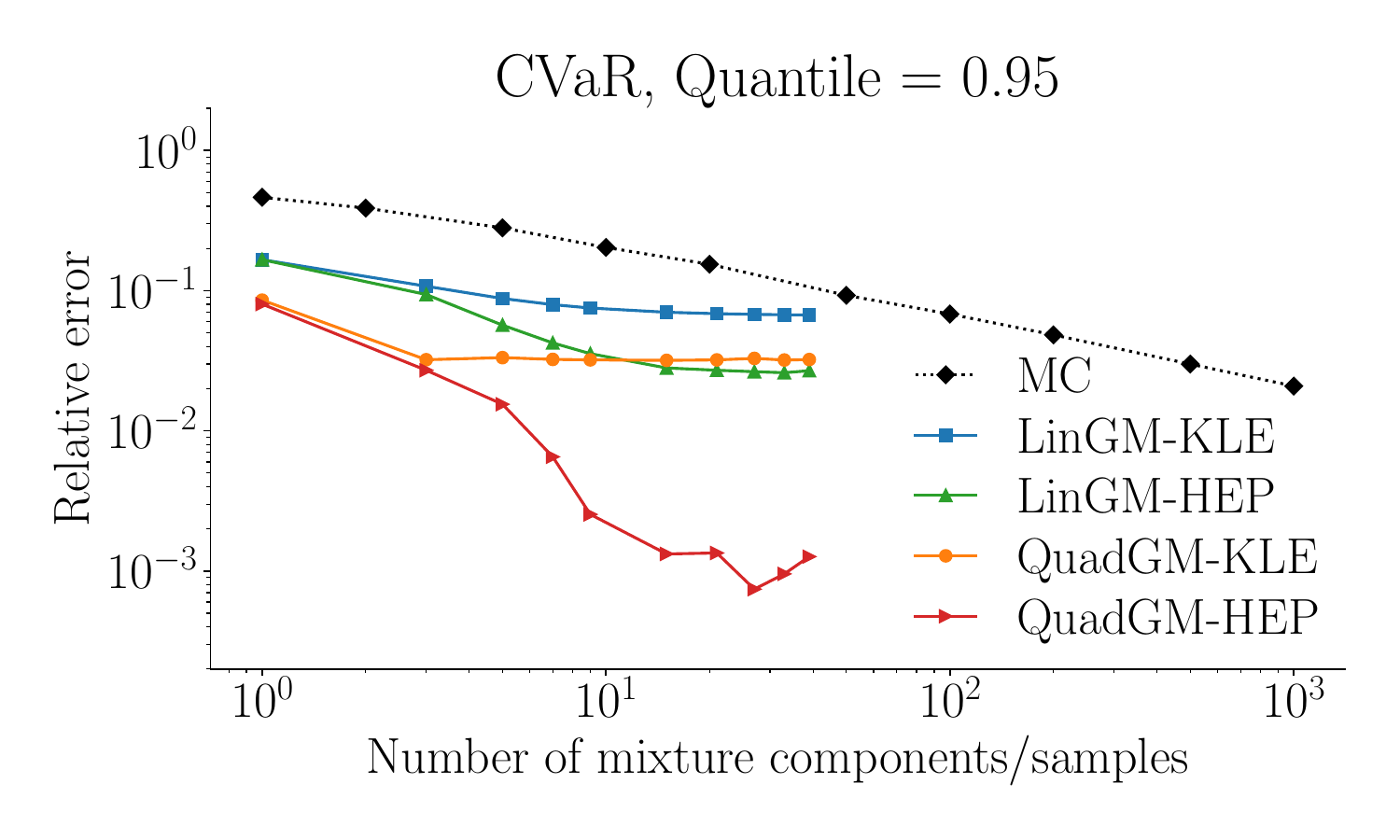}
		\caption{$\gamma=8, \delta=4$}
	\end{subfigure}
	\caption{Relative errors for estimates of CVaR ($\alpha = 0.95$) using linear and quadratic Gaussian mixture Taylor approximations with up to 
		$N_{\mix} = 39$ components along the dominant eigendirection. 
		Results are for $(\gamma, \delta)$ = (4, 8), (8, 8), and (8, 4), from left to right.}
	\label{fig:helmholtz_cvar_results}
\end{figure}

\begin{figure}[htbp!]
	\centering
	\begin{subfigure}{0.32\textwidth}
	\centering
	\includegraphics[width=\textwidth]{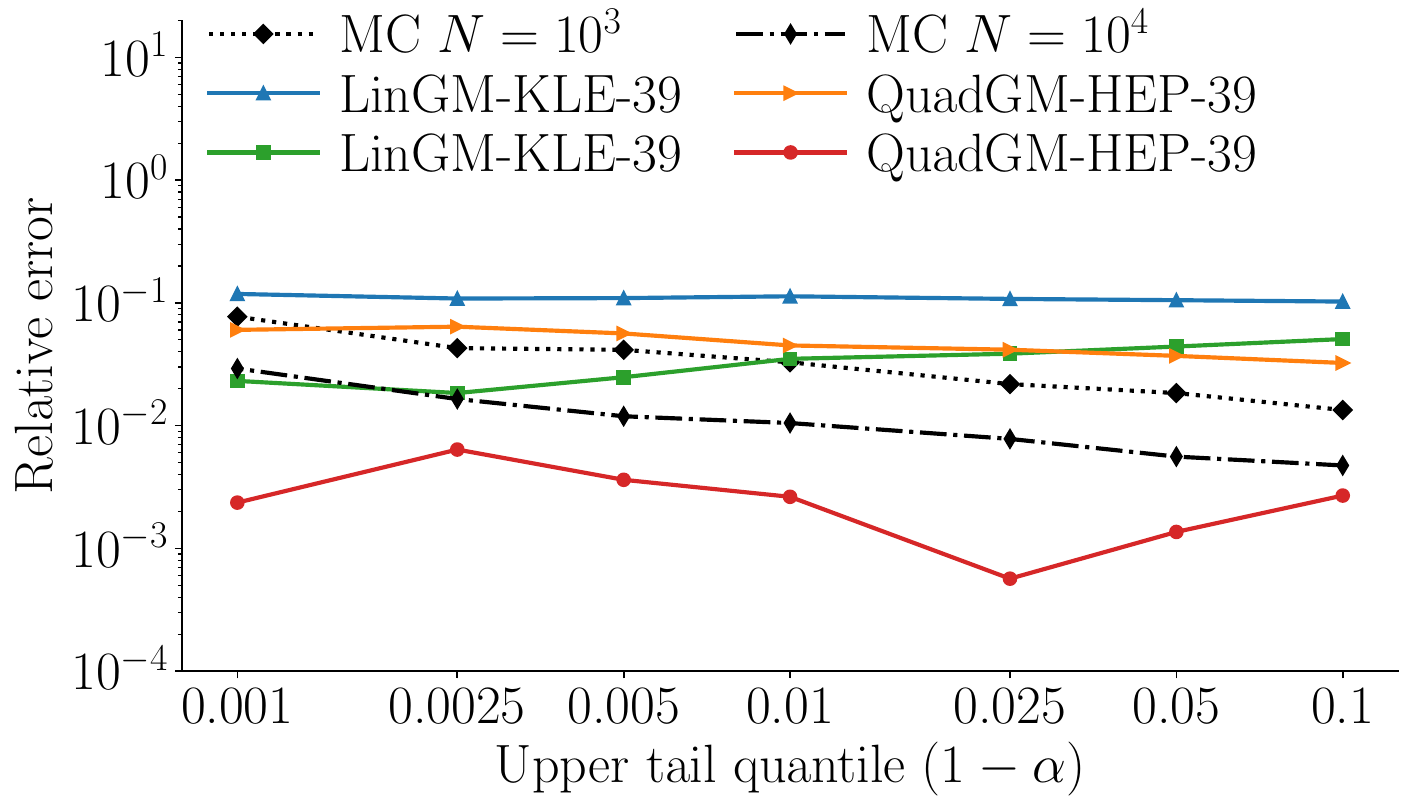}
	\caption{$\gamma=4, \delta=8$}
	\end{subfigure}
	\begin{subfigure}{0.32\textwidth}
	\centering
	\includegraphics[width=\textwidth]{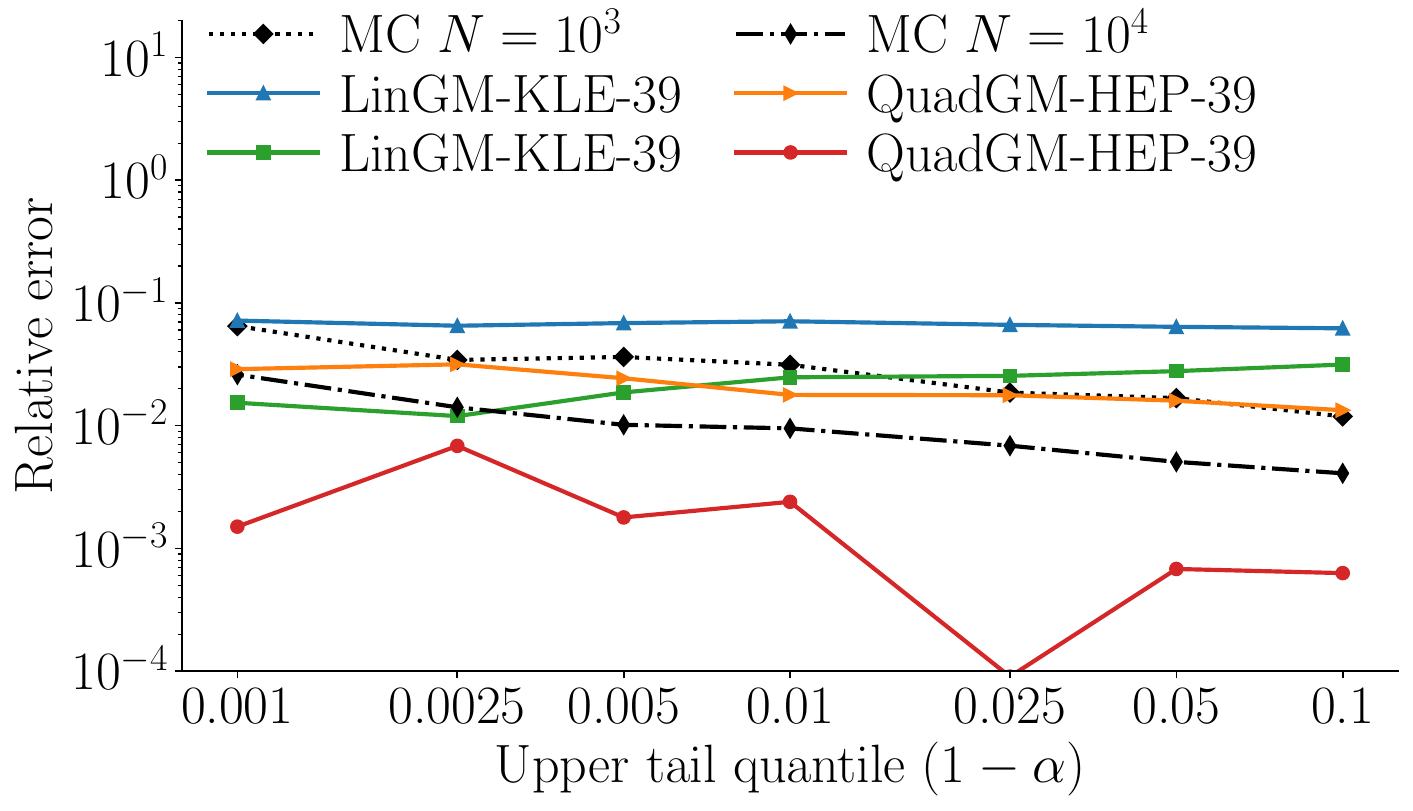}
	\caption{$\gamma=8, \delta=8$}
	\end{subfigure}
	\begin{subfigure}{0.32\textwidth}
	\centering
	\includegraphics[width=\textwidth]{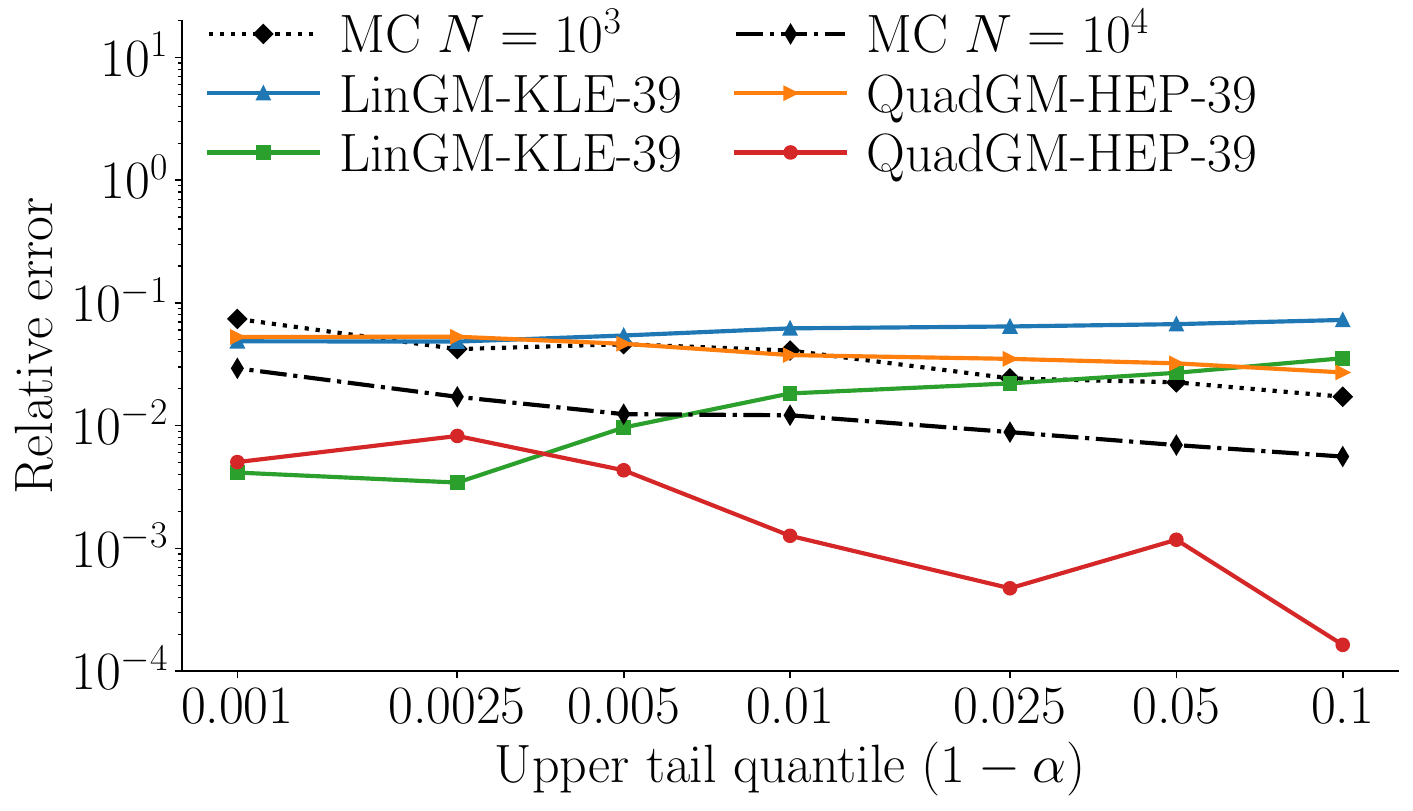}
	\caption{$\gamma=8, \delta=4$}
	\end{subfigure}
	\vspace{0pt}
	\caption{Relative errors for estimates of CVaR using 39 mixture components along the dominant KLE or HEP directions as a function of $1-\alpha$, for quantiles $\alpha$ from $0.9$ to $0.999$. These are compared against MC estimates using $10^3$ and $10^4$ samples. Results are shown for the cases of $(\gamma, \delta)$ = (4, 8), (8, 8), and (8, 4), from left to right.}
	\label{fig:helmholtz_quantile_sweep}
\end{figure}

\section{Conclusion}\label{sec:conclusion}
In this work, we presented a framework for improving the accuracy of Taylor approximations for risk measures of QoIs with uncertain parameters. 
A Gaussian mixture approximation of the underlying Gaussian measure is able to significantly improve the accuracy of the Taylor approximation, mitigating the deficiency of Taylor approximations for distributions with significant variances.
In our numerical results, using the mixture Taylor approximation with $\cO(10^1)$--$\cO(10^2)$ state PDE solves, risk measures can be estimated to within $1\%$ error, a level of accuracy comparable to Monte Carlo estimators with $\cO(10^{4})$ samples. 
This is also experimentally shown to be effective for tail based risk measures, such as the CVaR, which are often difficult to estimate using standard Monte Carlo sampling.
Thus, the proposed method can serve as a way to rapidly estimate risk measures when computational budget is limited to pre-asymptotic regimes.

In its current form, the method considered in this work also has several limitations. For example, we do observe that for QoIs with more extreme variations, Gaussian mixture approximations, though an improvement over a standard Taylor approximation, still may not be sufficient.
Moreover, we focused on mixture construction using a decomposition along a single direction, while further improvements to accuracy may require decompositions along additional directions. However, subsequent expansion directions become increasingly expensive, as a naive tensor product based construction of the Gaussian mixture approximation suffers from the usual curse of dimensionality. 
One may need to consider an adaptive mixture construction approach with sparsification schemes to further extend this method. 

\section{Acknowledgements}
The authors would like to acknowledge Professor Ryan P. Russell for the helpful discussions.

\appendix 

\section{Sampling the low rank Quadratic Taylor approximation}\label{sec:appendix_sampling}
We present the derivation for the sampling expression \eqref{eq:quadratic_sample}, which assumes $m \sim \cN(\bar{m}, \cC)$ is a Gaussian distribution on $\cM = \bR^n$ arising from the discretization of the space $\cM$ and $\cC$ is the covariance matrix corresponding to the discretization of the covariance operator.
Recall that the low rank quadratic approximation is given by
\[
Q_{\qua, \lr} = Q(\bar{m}) + \linner g, m - \bar{m} \rinner_{\cM}
+ \frac{1}{2} \sum_{i=1}^{r_{\cH}} \lambda_j \linner \cC^{-1} \phi_j, m - \bar{m} \rinner_{\cM}^2,
\]
with $g := D Q(\bar{m})$, 
and $\lambda_j, \phi_j$ is obtained from the generalized eigenvalue problem
$ D^2Q(\bar{m}) \phi_j = \lambda_j \cC^{-1} \phi_j. $
We also define $\psi_j = \cC^{-1/2} \phi_j$, so that 
$\cH \psi_j = \cC^{1/2} D^2 Q (\bar{m}) \cC^{1/2} \psi_j = \lambda_j \psi_j.$
Moreover, since $m \in \bR^{n}$, samples of $m - \bar{m}$ can be drawn from $m - \bar{m} = \cC^{1/2} \xi$, where $\cC^{1/2}$ is the square root of the covariance matrix 
and $\xi \sim \cN(0, \cI_{\bR^{n}})$ in $\bR^{n}$. 
This allows us to write the low rank quadratic approximation as 
\begin{equation}
Q_{\qua, \lr} \equaldist Q(\bar{m}) 
+ \linner g, \cC^{1/2} \xi \rinner_{\cM} + \frac{1}{2} \sum_{i=1}^{r_{\cH}} \lambda_j \linner \psi_j, \xi \rinner_{\cM}^2.
\end{equation}

In order to avoid sampling vectors $\xi \in \bR^{n}$, we consider an alternative representation. To this end, we let 
$\Psi = [\psi_1, \dots, \psi_{r_{\cH}}, \psi_{r_{\cH}+1}, \dots, \psi_n]$, 
where $\psi_{r_{\cH}+1}, \dots, \psi_{n}$ are the remaining eigenvectors of 
$\cH$
(completed to span $\bR^{n})$. Since this is an orthonormal basis for $\bR^{n}$, $\Psi \Psi^{\transpose} = \Psi^{\transpose} \Psi = \cI_{\bR^{n}}$. 
This allows to write 
\[
Q_{\qua, \lr} \equaldist Q(\bar{m}) + \linner g, \cC^{1/2} \Psi \Psi^{\transpose} \xi \rinner_{\cM}
+ \frac{1}{2} \linner \Psi^{\transpose} \xi, \Lambda \Psi^{\transpose} \xi \rinner_{\cM}
,
\]
where $\Lambda$ is a diagonal matrix with entries $\{\lambda_j\}_{j=1}^{n}$ on the diagonal. 
By orthogonality, we also have $y := \Psi^{\transpose} \xi  \sim \cN(0, \cI_{\bR^{n}})$. 
We now decompose $y$ into $x_{r_{\cH}} = [y_1, \dots, y_{r_{\cH}}]$ and $x_{r_{\cH}}' = [y_{r_{\cH}+1}, \dots y_{n}]$, and correspondingly, $\Psi$ into $\Psi_{r_{\cH}} = [\psi_1, \dots, \psi_{r_{\cH}}]$, $\Psi_{r_{\cH}}' = [\psi_{r_{\cH}+1}, \dots, \psi_{n}]$, and write $\Lambda_{r_{\cH}} = \diag(\lambda_1, \dots, \lambda_{r_{\cH}}).$ The low-rank quadratic approximation is therefore
\[
Q_{\qua, \lr}(m) \equaldist Q(\bar{m}) 
+ \linner g, \cC^{1/2} (\Psi_{r_{\cH}} x_{r_{\cH}} + \Psi_{r_{\cH}}' x_{r_{\cH}}') \rinner_{\cM}
+ \frac{1}{2} \linner x_{r_{\cH}}, \Lambda_{r_{\cH}} x_{r_{\cH}} \rinner_{\cM}.
\]

Due to the independence of $x_{r_{\cH}}$ and $x_{r_{\cH}}'$, 
we can decompose the expression into two independent parts that can be sampled separately, 
$Q_{\qua, \lr}(m) \equaldist \zeta_1 + \zeta_2$, where
\[
\zeta_1 := Q(\bar{m}) + \linner g, \cC^{1/2} (\Psi_{r_{\cH}} x_{r_{\cH}}) \rinner_{\cM} 
+ \frac{1}{2} \linner x_{r_{\cH}}, \Lambda_{r_{\cH}} x_{r_{\cH}} \rinner_{\cM},
\qquad \zeta_2 :=
\linner g, \cC^{1/2} (\Psi_{r_{\cH}}' x_{r_{\cH}}') \rinner_{\cM}.
\]
For the term $\zeta_2$, which involves $x_{r_{\cH}}'$, we can avoid computing $\Psi_{r_{\cH}}'$ by recognizing that 
\[
\Psi_{r_{\cH}}' x_{r_{\cH}}' \equaldist (\cI_{\bR^{n}} - \Psi_{r_{\cH}} \Psi_{r_{\cH}}^{\transpose}) z, \quad z \sim \cN(0, \cI_{\bR^{n}}).
\]
Thus, $\zeta_2$ becomes
\[
\zeta_2 \equaldist 
\linner g, \cC^{1/2}(\cI_{\bR^{n}} - \Psi_{r_{\cH}} \Psi_{r_{\cH}}^{\transpose}) z \rinner_{\cM} \sim \cN(0, \sigma_{r_{\cH}}'^2),
\]
which is normally distributed due to the linearity in $z$ and has variance
\[
(\sigma_{r_{\cH}}')^2 = \linner g, \cC^{1/2}(\cI_{\bR^{n}} - \Psi_{r_{\cH}} \Psi_{r_{\cH}}^{\transpose}) \cC^{1/2} g \rinner_{\cM}
= \linner g, (\cC - \Phi_{r_{\cH}} \Phi_{r_{\cH}}^{\transpose}) g \rinner_{\cM}
= \linner g, \cC g\rinner_{\cM} - \sum_{i=1}^{r_{\cH}} \linner g, \phi_j \rinner_{\cM}^2,
\]
where we have used the notation $\Phi_{r_{\cH}} := [\phi_1, \dots, \phi_{r_{\cH}}] = \cC^{1/2} \Psi_{r_{\cH}}$.
We can sample from its distribution by taking $\sigma_{r_{\cH}}' y_0$, where $y_0 \sim \cN(0,1)$.
On the other hand, $\zeta_1$ is simply
\[ \zeta_1  = 
Q(\bar{m}) + \sum_{i=1}^{r_{\cH}} \linner g, \phi_j \rinner_{\cM} y_j + \lambda_j y_j^2.
\]
Combining the two independent components yields 
\begin{equation}
Q_{\qua, \lr} \equaldist Q(\bar{m}) 
+ \left(\linner g, \cC g \rinner_{\cM} - \sum_{i=1}^{r_{\cH}} \linner g, \phi_j \rinner_{\cM} \right)^{\!\!1/2} \!\!\! y_0
+ \sum_{i=1}^{r_{\cH}} \linner g , \phi_j \rinner_{\cM} y_j + \lambda_j y_j^2,
\end{equation}
where $y_j \sim \cN(0,1)$ for $j = 0, \dots, r_{\cH}$ are i.i.d. standard Gaussian random variables.
\section{Proofs for error analysis} \label{appendix:proofs}
\subsection{Proof of Proposition \ref{thm:gm_1d}}
\begin{proof} (Proposition \ref{thm:gm_1d})
Our proof is constructive. We define mixture approximations using equally spaced Gaussians on an interval $(-L, L)$ for any $L >0$. That is, we consider the mixture approximation $\nu_{0,N} = \sum_{k=1}^{N} w_N^k \cN(\mu_N^k, \sigma_N^2)$, with means at the center of $N$ equally sized intervals from $-L$ to $L$, 
\begin{equation}
	\mu_N^{k} = -L - \frac{L}{N} + \frac{2L}{N} k,
\end{equation}
standard deviations $\sigma_N = N^{-p}$, and weights matching the shape of the reference density
\begin{equation}
	w_N^k = \frac{\pi_0(\mu_N^k)}{\sum_{k=1}^{N}\pi_0(\mu_N^k)},
\end{equation}
which have been normalized to maintain a sum to unity. 
The resulting mixture density is 
\begin{equation}
	\pi_N(x) = \frac{
		\sum_{k=1}^{N} \pi_0(\mu_N^k) \frac{1}{\sigma_N \sqrt{2 \pi}} \exp\left( - \frac{(x-\mu_N^k)^2}{2\sigma_N^2}\right)
	}{
		\sum_{k=1}^{N} \pi_0(\mu_N^k)
	}.
\end{equation}

From Lemma \ref{lemma:ratio_convergence}, we know the ratio $\pi_{N}(x) / \pi_0(x)$ converges to
\begin{equation}
	\lim_{N \rightarrow \infty}\frac{\pi_N(x)}{\pi_0(x)} = 
	\begin{cases}
		1/\nu_0(B_L) & |x| < L \\
		0 & |x| > L 
	\end{cases}.
\end{equation}
and is also uniformly bounded from above by $2/\nu_0(B_L)$ for sufficiently large $N$.

Recall that TV distance can be written as 
\[
\tv(\nu_0, \nu_N) 
= \frac{1}{2}\int_{\bR} \left| 1 - \frac{\pi_N(x)}{\pi_0(x)} \right| \pi_0(x) dx,
\]
which can be split into
\[
\tv(\nu_0, \nu_N) 
= \frac{1}{2}\int_{B_L} \left| 1 - \frac{\pi_N(x)}{\pi_0(x)} \right| \pi_0(x) dx 
+ \frac{1}{2}\int_{B_L^c} \left| 1 - \frac{\pi_N(x)}{\pi_0(x)} \right| \pi_0(x) dx.
\]
Since $L$ is arbitrary, we can choose $L$ sufficiently large so that 
\[
\left(1  + \frac{2}{\nu_0(B_L)} \right) (1 - \nu_0(B_L)) \leq \epsilon
\]
and simultaneously
\begin{equation}\label{eq:size_of_ball}
	\left|(1 - 1/\nu_0(B_L))\nu_0(B_L)\right| = (1 - \nu_0(B_L)) \leq \frac{\epsilon}{2}.
\end{equation}
We can then select $N$ sufficiently large such that the bound 
$ {\pi_N(x)}/{\pi_0(x)} \leq {2}/{\nu_0(B_L)} $
is satisfied and so the second integral is bounded, 
\begin{equation}\label{eq:outside_bound}
\frac{1}{2}\int_{B_L^c} \left| 1 - \frac{\pi_N(x)}{\pi_0(x)} \right| \pi_0(x) dx \leq
\frac{1}{2}\left(1  + \frac{2}{\nu_0(B_L)} \right) (1 - \nu_0(B_L)) 
\leq \frac{\epsilon}{2}.
\end{equation}
In the first integral, we know 
$$
1 - \frac{\pi_N(x)}{\pi_0(x)} \rightarrow 1 - \frac{1}{\nu_0(B_L)} \quad \forall x \in B_L,
$$
and we have a bound 
$$\left| 1 - \frac{\pi_N(x)}{\pi_0(x)} \right| \leq 1 + \frac{2}{\nu_0(B_L)}.$$
The Lebesgue dominated convergence theorem asserts that 
$$
\int_{B_L} \left| 1 - 	\frac{\pi_N(x)}{\pi_0(x)} \right| \pi_0(x) dx
\rightarrow \left|1 - \frac{1}{\nu_0(B_L)} \right| \nu_0(B_L)
= |1 - \nu_0(B_L)|.
$$
Thus, we can pick $N$ large such that 
\begin{equation}\label{eq:inside_bound}
	\left| \int_{B_L} \left| 1 - \frac{\pi_N(x)}{\pi_0(x)}\right| \pi_0(x) dx - (1 - \nu_0(B_L)) \right| \leq \frac{\epsilon}{2}.
\end{equation}
The overall bound follows from combining \eqref{eq:inside_bound} with \eqref{eq:size_of_ball} and \eqref{eq:outside_bound}.
\end{proof}

\begin{lemma} \label{lemma:ratio_convergence}
Let $N \in \bN$, $p \in (0,1)$, $L > 0$, and 
\begin{equation}
	\mu_N^{k} = -L -\frac{L}{N} + \frac{2L}{N} k, \quad
	w_N^k = \frac{\pi_0(\mu_N^k)}{\sum_{k=1}^{N}\pi_0(\mu_N^k)}, \quad
	\sigma_N = N^{-p}
\end{equation}
be the means, weights, and standard deviations corresponding to the Gaussian mixture $\nu_{0,N} = \sum_{i=1}^{k} w_N^k \cN(\mu_{N}^{k}, \sigma_N^2)$ with $\pi_{N}$ as its PDF. 
Then the ratio 
$\pi_{N}(x)/\pi_0(x)$, where $\pi_0(x)$ is the PDF for $\nu_0 = \cN(0, 1)$, converges pointwise with limits
\begin{equation}\label{eq:pointwise_limit_ratio}
	\lim_{N \rightarrow \infty}\frac{\pi_N(x)}{\pi_0(x)} = 
	\begin{cases}
		1/\nu_0(B_L) & |x| < L \\
		0 & |x| > L
	\end{cases}
\end{equation}
where 
$B_L = (-L, L) \subset \bR$. Moreover, given any $L$, there exists a uniform bound
\begin{equation}
	\frac{\pi_N(x)}{\pi_0(x)} \leq \frac{2}{\nu_0(B_L)},
\end{equation}
for $N$ sufficiently large. 
\end{lemma}

\begin{proof}

We can begin by writing the ratio of the mixture and reference densities as 
\begin{equation}
	\frac{\pi_N(x)}{\pi_0(x)} =
	\frac{
		\sum_{k=1}^{N} \sqrt{1 - \sigma_N^2} \exp\left( \frac{\squared{\sigma_N \mu_N^k}}{2(1-\sigma_N^2)}\right) 
			\frac{1}{s_N \sqrt{2 \pi}} 
			\exp\left( - \frac{(\mu_N^k - (1 - \sigma_N^2)x)^2}{2 s_N^2} \right)
	}{
		\sum_{k=1}^{N} \pi_0(\mu_N^k)
	},
\end{equation}
with $s_N = \sigma_{N} \sqrt{1-\sigma_N^2} \leq \sigma_N$,
which follows by completing the square for the quadratic forms.
We then multiply the numerator and denominator by $2L/N$
\begin{equation}
	\frac{\pi_N(x)}{\pi_0(x)} =
	\frac{\frac{2L}{N}
		\sum_{k=1}^{N} \sqrt{1 - \sigma_N^2}\exp\left( \frac{\squared{\sigma_N \mu_N^k}}{2(1-\sigma_N^2)}\right) 
			\frac{1}{s_N \sqrt{2 \pi}} 
			\exp\left( - \frac{(\mu_N^k - (1 - \sigma_N^2)x)^2}{2 s_N^2} \right)
	}{
		\frac{2L}{N}\sum_{k=1}^{N} \pi_0(\mu_N^k)
	},
\end{equation}
The denominator corresponds to a Riemann sum converging to $\nu_0(B_L)$ as $N \rightarrow \infty$.
We now focus on the numerator, which we denote by
\begin{equation}
a_N(x)
= \frac{2L}{N}\sum_{k=1}^{N} \sqrt{1-\sigma_N^2} \exp\left( \frac{\squared{\sigma_N \mu_N^k}}{2(1-\sigma_N^2)}\right) 
			\frac{1}{s_N \sqrt{2 \pi}} 
			\exp\left( - \frac{(\mu_N^k - (1 - \sigma_N^2)x)^2}{2s_N^2}\right).
\end{equation}
To prove the convergence of $a_N$, we derive upper and lower bounds of $a_N$ and apply the squeeze theorem. We begin with the upper bound, noting that we can bound the first part of the numerator by
$$ 
\sqrt{1 - \sigma_N^2} \exp\left(\frac{\squared{\sigma_N \mu_N^k}}{2(1-\sigma_N^2)}\right) 
\leq \exp \left(\frac{\sigma_N^2 L^2}{2(1-\sigma_N^2)} \right)
\rightarrow 1
$$
since $\mu_L^k \in [-L, L]$ and $\sigma_N \rightarrow 0$. In fact, the ratio $\sigma_N^2/(1-\sigma_N^2) = (N^{2p} - 1)^{-1}$ decreases monotonically with $N$. 
We now focus on the remaining part. For convenience, we let
\begin{equation}
b_N(x) := 
\frac{2L}{N} \sum_{k=1}^{N}
\frac{1}{s_N \sqrt{2 \pi}} 
\exp\left( \frac{(\mu_N^k - (1 - \sigma_N^2)x)^2}{2s_N^2}\right),
\end{equation}
which allows us us to write the upper bound as
\begin{equation}\label{eq:upper_bound_numerator}
a_N(x) \leq \exp \left(\frac{\sigma_N^2 L^2}{2(1-\sigma_N^2)} \right) b_N(x) =: u_N(x)
\end{equation}

We consider a change of variables, 
\begin{align}
z_N^k(x) = \frac{\mu_N^k - (1-\sigma_N^2)x}{s_N}.
\end{align}
The points $z_N^k(x)$ are linearly spaced at distance
\[
\Delta z_N = \frac{2L}{N s_N} = \frac{2L N^{p-1}}{\sqrt{1 - N^{-2p}}}
\]
with start and end points within the interval
\[
z_N^1(x) = \frac{-L + L/N - x}{s_N} + \frac{\sigma_N}{\sqrt{1 - \sigma_N^2}} x,
\quad
z_N^N(x) = \frac{L - L/N - x}{s_N} + \frac{\sigma_N}{\sqrt{1 - \sigma_N^2}} x,
\]
respectively. We can therefore rewrite $b_N$ as
\[
b_N(x) = \Delta z_N \sum_{k=1}^{N} \frac{1}{\sqrt{2\pi}} \exp \left(-\frac{(z_N^k(x))^2}{2} \right).
\]
Intuitively, this is another Riemann sum where the spacing $\Delta z_N \rightarrow 0 $ for $p < 1$, but on a moving window. To show convergence, we consider a fixed $R > 0$, and consider the decomposition 
\[
b_N(x) = \Delta z_N \sum_{k=1}^{N} \frac{1}{\sqrt{2\pi}} \exp \left(-\frac{(z_N^k(x))^2}{2} \right) 
\left(\ind_{B_R}(z_N^k(x)) +  \ind_{B_R^c}(z_N^k(x)) \right).
\]
To obtain an upper bound, we note that the Gaussian density $\exp(-z^2/2)$ is largest when $z^2$ is small. This allows us to obtain the bound
\[
\Delta z_N \sum_{k=1}^{N} \frac{1}{\sqrt{2\pi}} \exp \left(-\frac{(z_N^k(x))^2}{2} \right) \ind_{B_R^c}(z_k(x)) 
\leq 
\Delta z_N \sum_{k=0}^{N-1} \frac{2}{\sqrt{2\pi}} \exp \left(-\frac{(R + k \Delta z_N)^2}{2} \right),
\]
where essentially, we have conservatively replaced the equidistant points $z_k$ that are in $B_R^c$ by equidistant points starting at $z=R$, and doubled the sum to cover points extending out in both directions. We then apply another conservative bound, $\exp(-z^2/2) \leq 2 \exp(-|z|)$, so that

\[
\Delta z_N \sum_{k=0}^{N-1} \frac{2}{\sqrt{2\pi}} \exp \left(-\frac{(R + k \Delta z_N)^2}{2} \right)
\leq 
\Delta z_N \sum_{k=0}^{N-1} \frac{4}{\sqrt{2\pi}} \exp(-(R + k \Delta z_N))
\leq \frac{4 e^{-R}}{\sqrt{2\pi}} \frac{\Delta z_N}{1 - e^{-\Delta z_N}},
\]
where the final expression follows from the limit of the geometric series, and converges to $4e^{-R}/\sqrt{2\pi}$ as $N \rightarrow \infty$. We now return to $b_N$, where using the new bound, we have
\[
b_N(x) \leq \Delta z_N \sum_{k=1}^{N} \frac{1}{\sqrt{2\pi}} \exp \left(-\frac{(z_N^k(x))^2}{2} \right) \ind_{B_R}(z_N^k(x)) 
+ \frac{4 e^{-R}}{\sqrt{2\pi}} \frac{\Delta z_N }{1 - e^{-\Delta z_N}}.
\]
On the other hand, we also have the lower bound on $b_N$,
\[
b_N(x) \geq \Delta z_N \sum_{k=1}^{N} \frac{1}{\sqrt{2\pi}} \exp \left(-\frac{(z_N^k(x))^2}{2} \right) \ind_{B_R}(z_N^k(x)).
\]
The limiting behavior of
$
c_N(x) := \Delta z_N \sum_{k=1}^{N} {1}/{\sqrt{2\pi}} \exp \left(-{(z_N^k(x))^2}/{2} \right) \ind_{B_R}(z_N^k(x))
$
as $N \rightarrow \infty$ can be considered in three cases.
\begin{enumerate}
\item
$x \in (-L, L)$. In this case, $z_N^0 \rightarrow -\infty$ and $z_N^N \rightarrow \infty$, meaning for any $B_R$, we have $B_R \subset (z_N^0, z_N^N)$ eventually. The shrinking of the spacing $\Delta z_N$ implies we have a converging Riemann sum, $c_N \rightarrow \nu_0(B_R)$.

\item
$x < -L$ or $x > L$. Here, both $z_N^0, z_N^N < 0$ or $z_N^0, z_N^N > 0$ and so $z_N^0, z_N^N \rightarrow \pm\infty$ as $N \rightarrow \infty$. This implies eventually no points remain in $B_R$, and so $c_N \rightarrow 0$.

\item $x = \pm L$. Either $z_N^0 \rightarrow 0$ and $z_N^N \rightarrow \infty$ or $z_N^0 \rightarrow \infty$ and $z_N^N \rightarrow 0$, so eventually half of $B_R$ is covered by $(z_N^0, z_N^N)$. Thus, $c_N \rightarrow \nu_0(B_R)/2.$
\end{enumerate}
Moreover, we note that 
\[
\frac{\Delta z_N}{1 - e^{-\Delta z_N}} \rightarrow 1 \text{ as } \Delta z_N \rightarrow 0,
\]
which can be verified using L'H\^opital's rule.

For any fixed, $R$, we can now pass to the $\liminf$,
\begin{equation}
\begin{cases} 
\nu_0(B_R) \leq \liminf_{N \rightarrow \infty} b_N(x) \leq \nu_0(B_R)+ {4 e^{-R}}/{\sqrt{2\pi}} & |x|< L , \\
0 \leq \liminf_{N \rightarrow \infty} b_N(x) \leq {4e^{-R}}/{\sqrt{2\pi}} & |x| > L , \\
\end{cases}
\end{equation}
and since $R$ can be made arbitrarily large,
\begin{equation}
\liminf_{N \rightarrow \infty} b_N(x)  = 
\begin{cases} 
1 & |x| < L, \\
0 & |x| > L. \\
\end{cases}
\end{equation}
Analogous arguments can be made for the $\limsup$. Thus, 
\begin{equation}\label{eq:lim_bn}
\lim_{N \rightarrow \infty} b_N(x)  = 
\begin{cases} 
1 & |x| < L, \\
0 & |x| > L. \\
\end{cases}
\end{equation}
This gives an upper bound on the numerator with the limit 
\begin{equation}
u_N(x) = \exp \left(\frac{\sigma_N^2L^2}{2(1-\sigma_N^2)} \right) b_N(x) \rightarrow 
\begin{cases} 
1 & |x| < L, \\
0 & |x| > L, \\
\end{cases}
\end{equation}
as $N \rightarrow \infty$.

We now consider a lower bound. Recall 
\[
a_N(x)
= \frac{2L}{N}\sum_{k=0}^{N} \sqrt{1-\sigma_N^2} \exp\left( \frac{\squared{\sigma_N \mu_N^k}}{2(1-\sigma_N^2)}\right) 
			\frac{1}{s_N \sqrt{2 \pi}} 
			\exp\left( - \frac{(\mu_N^k - (1 - \sigma_N^2)x)^2}{2s_N^2}\right).
\]
We have immediately the lower bound
$a_N(x)
\geq \sqrt{1 - \sigma_N^2} b_N(x) =: l_N(x)$,
which, combined with \eqref{eq:lim_bn} yields
\begin{equation}
\lim_{N \rightarrow \infty} l_{N}(x) =
\begin{cases} 
1 & |x| < L, \\
0 & |x| > L. \\
\end{cases}
\end{equation}
Thus, we have $l_N(x) \leq a_N(x) \leq u_N(x)$, and so by the squeeze theorem
\begin{equation}
\lim_{N \rightarrow \infty} a_{N}(x) =
\begin{cases} 
1 & |x| < L, \\
0 & |x| > L. \\
\end{cases}
\end{equation}
This gives our overall pointwise convergence result \eqref{eq:pointwise_limit_ratio}.

Furthermore, to prove that the ratio of densities has a uniform bound independent of $x$ and $N$, we return to the bound 
$a_N(x) \leq \exp \left(\frac{\sigma_N^2 L^2}{2(1-\sigma_N^2)} \right) b_N(x)$.
In particular, we know 
$\exp\left( \frac{\sigma_N^2 L^2}{2(1-\sigma_N^2)} \right) \rightarrow 1$ as $N \rightarrow \infty$, and hence this term can be arbitrarily close to unity for large $N$.
We can obtain a bound for $b_N$ that is independent of $x$ using similar arguments as before, 
\begin{align*}
b_N(x) &= \Delta z_N \sum_{k=0}^{N} \frac{1}{\sqrt{2\pi}} \exp\left(- \frac{(z_N^k(x))^2}{2} \right) 
\leq \Delta z_N \sum_{k=0}^{N} \frac{2}{\sqrt{2\pi}} \exp\left(- \frac{(k \Delta z_N)^2}{2} \right)\\
&\leq \Delta z_N \sum_{k=0}^{N} \frac{4}{\sqrt{2\pi}} \exp\left(- k \Delta z_N \right)
\leq \frac{4}{\sqrt{2\pi}} \frac{\Delta z_N}{1 - e^{-\Delta z_N}}  
\end{align*}
which converges to $4/\sqrt{2\pi} \approx 1.6 < 2$. Finally, recall that the denominator 
converges to $\nu_0(B_L)$, also independently of $x$.
We can thus use the convergence results to obtain a uniform bound
\begin{equation}
	\frac{\pi_N(x)}{\pi_0(x)} \leq \frac{2}{\nu_0(B_L)} \quad \forall x \in \mathbb{R}, \quad \forall N \geq N^*,
\end{equation}
where $N^*$ may depend on $L$ but not $x$.
\end{proof}

\bibliographystyle{siamplain}
\bibliography{references,additional_references}

\begin{thebibliography}{10}

\bibitem{AlexanderianPetraStadlerEtAl17}
{\sc A.~Alexanderian, N.~Petra, G.~Stadler, and O.~Ghattas}, {\em Mean-variance
  risk-averse optimal control of systems governed by {PDEs} with random
  parameter fields using quadratic approximations}, SIAM/ASA Journal on
  Uncertainty Quantification, 5 (2017), pp.~1166--1192,
  \url{https://doi.org/10.1137/16M106306X}.

\bibitem{AlexanderianPetraStadlerEtAl17a}
{\sc A.~Alexanderian, N.~Petra, G.~Stadler, and O.~Ghattas}, {\em Mean-variance
  risk-averse optimal control of systems governed by {PDEs} with random
  parameter fields using quadratic approximations}, SIAM/ASA Journal on
  Uncertainty Quantification, 5 (2017), pp.~1166--1192.

\bibitem{AlghamdiHesseChenEtAl21}
{\sc A.~Alghamdi, M.~Hesse, J.~Chen, U.~Villa, and O.~Ghattas}, {\em Bayesian
  poroelastic aquifer characterization from {InSAR} surface deformation data.
  {P}art {II}: {Q}uantifying the uncertainty}, Water Resources Research, 57
  (2021), p.~e2021WR029775, \url{https://doi.org/10.1029/2021WR029775}.

\bibitem{AliUllmannHinze17}
{\sc A.~A. Ali, E.~Ullmann, and M.~Hinze}, {\em Multilevel {M}onte {C}arlo
  analysis for optimal control of elliptic {PDEs} with random coefficients},
  SIAM/ASA Journal on Uncertainty Quantification, 5 (2017), pp.~466--492.

\bibitem{ArtznerDelbaenEberEtAl99}
{\sc P.~Artzner, F.~Delbaen, J.-M. Eber, and D.~Heath}, {\em Coherent measures
  of risk}, jul 1999, \url{https://doi.org/10.1111/1467-9965.00068}.

\bibitem{Bacharoglou10}
{\sc A.~Bacharoglou}, {\em Approximation of probability distributions by convex
  mixtures of gaussian measures}, Proceedings of the American Mathematical
  Society, 138 (2010), pp.~2619--2628,
  \url{https://doi.org/10.1090/S0002-9939-10-10340-2}.

\bibitem{Bui-ThanhBursteddeGhattasEtAl12}
{\sc T.~Bui-Thanh, C.~Burstedde, O.~Ghattas, J.~Martin, G.~Stadler, and L.~C.
  Wilcox}, {\em {Extreme-scale UQ for Bayesian inverse problems governed by
  PDEs}}, in SC12: Proceedings of the International Conference for High
  Performance Computing, Networking, Storage and Analysis, 2012.

\bibitem{Bui-ThanhGhattas12a}
{\sc T.~Bui-Thanh and O.~Ghattas}, {\em Analysis of the {H}essian for inverse
  scattering problems. {P}art {I}: Inverse shape scattering of acoustic waves},
  Inverse Problems, 28 (2012), p.~055001,
  \url{https://doi.org/10.1088/0266-5611/28/5/055001}.

\bibitem{Bui-ThanhGhattas12}
{\sc T.~Bui-Thanh and O.~Ghattas}, {\em Analysis of the {H}essian for inverse
  scattering problems. {P}art {II}: Inverse medium scattering of acoustic
  waves}, Inverse Problems, 28 (2012), p.~055002,
  \url{https://doi.org/10.1088/0266-5611/28/5/055002}.

\bibitem{Bui-ThanhGhattas13}
{\sc T.~Bui-Thanh and O.~Ghattas}, {\em Randomized maximum likelihood sampling
  for large-scale {B}ayesian inverse problems}, In preparation,  (2013).

\bibitem{Bui-ThanhGhattasMartinEtAl13}
{\sc T.~Bui-Thanh, O.~Ghattas, J.~Martin, and G.~Stadler}, {\em A computational
  framework for infinite-dimensional {B}ayesian inverse problems {P}art {I}:
  {T}he linearized case, with application to global seismic inversion}, SIAM
  Journal on Scientific Computing, 35 (2013), pp.~A2494--A2523,
  \url{https://doi.org/10.1137/12089586X}.

\bibitem{BungartzGriebel04}
{\sc H.-J. Bungartz and M.~Griebel}, {\em Sparse grids}, Acta Numerica, 13
  (2004), pp.~1--123, \url{https://doi.org/10.1017/S0962492904000182}.

\bibitem{CastrillonCandasNobileTempone21}
{\sc J.~E. Castrill{\'{o}}n-Cand{\'{a}}s, F.~Nobile, and R.~F. Tempone}, {\em A
  hybrid collocation-perturbation approach for {PDEs} with random domains},
  Advances in Computational Mathematics, 47 (2021),
  \url{https://doi.org/10.1007/s10444-021-09859-6}.

\bibitem{ChaudhuriKramerNortonEtAl21}
{\sc A.~Chaudhuri, B.~Kramer, M.~Norton, J.~O. Royset, and K.~Willcox}, {\em
  Certifiable risk-based engineering design optimization}, arXiv preprint
  arXiv:2101.05129,  (2021).

\bibitem{ChenGhattas21}
{\sc P.~Chen and O.~Ghattas}, {\em Taylor approximation for chance constrained
  optimization problems governed by partial differential equations with
  high-dimensional random parameters}, SIAM/ASA Journal on Uncertainty
  Quantification, 9 (2021), pp.~1381--1410.

\bibitem{ChenHabermanGhattas21}
{\sc P.~Chen, M.~Haberman, and O.~Ghattas}, {\em Optimal design of acoustic
  metamaterial cloaks under uncertainty}, Journal of Computational Physics, 431
  (2021), p.~110114.

\bibitem{ChenVillaGhattas19}
{\sc P.~Chen, U.~Villa, and O.~Ghattas}, {\em Taylor approximation and variance
  reduction for {PDE}-constrained optimal control under uncertainty}, Journal
  of Computational Physics, 385 (2019), pp.~163--186,
  \url{https://arxiv.org/abs/1804.04301}.

\bibitem{ChiralaksanakulMahadevan04}
{\sc A.~Chiralaksanakul and S.~Mahadevan}, {\em First-order approximation
  methods in reliability-based design optimization}, Journal of Mechanical
  Design, 127 (2004), pp.~851--857, \url{https://doi.org/10.1115/1.1899691}.

\bibitem{DaonStadler18}
{\sc Y.~Daon and G.~Stadler}, {\em Mitigating the influence of boundary
  conditions on covariance operators derived from elliptic {PDEs}}, Inverse
  Problems and Imaging, 12 (2018), pp.~1083--1102,
  \url{https://doi.org/10.3934/ipi.2018045},
  \url{https://arxiv.org/abs/1610.05280}.

\bibitem{DeMarsBishopJah13}
{\sc K.~J. DeMars, R.~H. Bishop, and M.~K. Jah}, {\em Entropy-based approach
  for uncertainty propagation of nonlinear dynamical systems}, Journal of
  Guidance, Control, and Dynamics, 36 (2013), pp.~1047--1057,
  \url{https://doi.org/10.2514/1.58987}.

\bibitem{FossaArmellinDelandeEtAl22}
{\sc A.~Fossà, R.~Armellin, E.~Delande, M.~Losacco, and F.~Sanfedino}, {\em
  Multifidelity orbit uncertainty propagation using taylor polynomials}, in
  {AIAA} {SCITECH} 2022 Forum, American Institute of Aeronautics and
  Astronautics, Jan. 2022, \url{https://doi.org/10.2514/6.2022-0859}.

\bibitem{GeraciEldredIaccarino17}
{\sc G.~Geraci, M.~S. Eldred, and G.~Iaccarino}, {\em A multifidelity
  multilevel {M}onte {C}arlo method for uncertainty propagation in aerospace
  applications}, in 19th {AIAA} Non-Deterministic Approaches Conference,
  American Institute of Aeronautics and Astronautics, jan 2017,
  \url{https://doi.org/10.2514/6.2017-1951}.

\bibitem{GhattasWillcox21}
{\sc O.~Ghattas and K.~Willcox}, {\em Learning physics-based models from data:
  perspectives from inverse problems and model reduction}, Acta Numerica, 30
  (2021), pp.~445--554, \url{https://doi.org/doi:10.1017/S0962492921000064}.

\bibitem{Giles15}
{\sc M.~B. Giles}, {\em Multilevel monte carlo methods}, Acta Numerica, 24
  (2015), pp.~259--328, \url{https://doi.org/10.1017/s096249291500001x}.

\bibitem{GrahamKuoNuyensEtAl11}
{\sc I.~Graham, F.~Kuo, D.~Nuyens, R.~Scheichl, and I.~Sloan}, {\em Quasi-monte
  carlo methods for elliptic pdes with random coefficients and applications},
  Journal of Computational Physics, 230 (2011), pp.~3668--3694,
  \url{https://doi.org/https://doi.org/10.1016/j.jcp.2011.01.023},
  \url{https://www.sciencedirect.com/science/article/pii/S0021999111000489}.

\bibitem{HongHuLiu14}
{\sc L.~J. Hong, Z.~Hu, and G.~Liu}, {\em Monte carlo methods for value-at-risk
  and conditional value-at-risk}, {ACM} Transactions on Modeling and Computer
  Simulation, 24 (2014), pp.~1--37, \url{https://doi.org/10.1145/2661631}.

\bibitem{IsaacPetraStadlerEtAl15a}
{\sc T.~Isaac, N.~Petra, G.~Stadler, and O.~Ghattas}, {\em Scalable and
  efficient algorithms for the propagation of uncertainty from data through
  inference to prediction for large-scale problems, with application to flow of
  the {A}ntarctic ice sheet}, Journal of Computational Physics, 296 (2015),
  pp.~348--368, \url{https://doi.org/10.1016/j.jcp.2015.04.047}.
\newblock Winner, SIAM SIAG Computational Science \& Engineering Best Paper
  Award for papers published in 2015--2018.

\bibitem{LawStuartZygalakis15}
{\sc K.~Z. K.~Law, A.~Stuart}, {\em Data Assimilation: A Mathematical
  Introduction}, Springer Texts in Applied Mathematics, 2015.

\bibitem{KalmikovHeimbach14}
{\sc A.~G. Kalmikov and P.~Heimbach}, {\em A {H}essian-based method for
  uncertainty quantification in global ocean state estimation}, SIAM Journal on
  Scientific Computing, 36 (2014), pp.~S267--S295.

\bibitem{KiureghianLinHwang87}
{\sc A.~D. Kiureghian, H.~Lin, and S.~Hwang}, {\em Second‐order reliability
  approximations}, Journal of Engineering Mechanics, 113 (1987),
  pp.~1208--1225,
  \url{https://doi.org/10.1061/(ASCE)0733-9399(1987)113:8(1208)},
  \url{https://ascelibrary.org/doi/abs/10.1061/%28ASCE%290733-9399%281987%29113%3A8%281208%29},
  \url{https://arxiv.org/abs/https://ascelibrary.org/doi/pdf/10.1061/%28ASCE%290733-9399%281987%29113%3A8%281208%29}.

\bibitem{KouriSurowiec16}
{\sc D.~P. Kouri and T.~M. Surowiec}, {\em Risk-averse {PDE}-constrained
  optimization using the conditional value-at-risk}, SIAM Journal on
  Optimization, 26 (2016), pp.~365--396,
  \url{https://doi.org/10.1137/140954556}.

\bibitem{LindgrenRueLindstroem11}
{\sc F.~Lindgren, H.~Rue, and J.~Lindstr{\"o}m}, {\em An explicit link between
  {G}aussian fields and {G}aussian {M}arkov random fields: the stochastic
  partial differential equation approach}, Journal of the Royal Statistical
  Society: Series B (Statistical Methodology), 73 (2011), pp.~423--498,
  \url{https://doi.org/10.1111/j.1467-9868.2011.00777.x},
  \url{http://dx.doi.org/10.1111/j.1467-9868.2011.00777.x}.

\bibitem{MartinWilcoxBursteddeEtAl12}
{\sc J.~Martin, L.~C. Wilcox, C.~Burstedde, and O.~Ghattas}, {\em A stochastic
  {Newton MCMC} method for large-scale statistical inverse problems with
  application to seismic inversion}, SIAM Journal on Scientific Computing, 34
  (2012), pp.~A1460--A1487, \url{https://doi.org/10.1137/110845598}.

\bibitem{MelchersBeck18}
{\sc R.~E. Melchers and A.~T. Beck}, {\em Structural reliability analysis and
  prediction}, John wiley \& sons, third~ed., 2018,
  \url{https://doi.org/10.1002/9781119266105}.

\bibitem{NgWillcox14}
{\sc L.~Ng and K.~Willcox}, {\em Multifidelity approaches for optimization
  under uncertainty}, International Journal for Numerical Methods in
  Engineering, 100 (2014), pp.~746--772,
  \url{https://doi.org/10.1002/nme.4761}.

\bibitem{NobileTemponeWebster08a}
{\sc F.~Nobile, R.~Tempone, and C.~Webster}, {\em A sparse grid stochastic
  collocation method for partial differential equations with random input
  data}, SIAM Journal on Numerical Analysis, 46 (2008), pp.~2309--2345,
  \url{https://doi.org/10.1137/060663660}.

\bibitem{PeherstorferWillcoxGunzburger16}
{\sc B.~Peherstorfer, K.~Willcox, and M.~Gunzburger}, {\em Optimal model
  management for multifidelity monte carlo estimation}, {SIAM} Journal on
  Scientific Computing, 38 (2016), pp.~A3163--A3194,
  \url{https://doi.org/10.1137/15M1046472}.

\bibitem{Rackwitz01}
{\sc R.~Rackwitz}, {\em Reliability analysis{\textemdash}a review and some
  perspectives}, Structural Safety, 23 (2001), pp.~365--395,
  \url{https://doi.org/10.1016/s0167-4730(02)00009-7}.

\bibitem{RockafellarUryasev00}
{\sc R.~T. Rockafellar and S.~Uryasev}, {\em Optimization of conditional
  value-at-risk}, Journal of risk, 2 (2000), pp.~21--42.

\bibitem{SaibabaLeeKitanidis15}
{\sc A.~K. Saibaba, J.~Lee, and P.~K. Kitanidis}, {\em Randomized algorithms
  for generalized hermitian eigenvalue problems with application to computing
  karhunen{\textendash}lo{\`{e}}ve expansion}, Numerical Linear Algebra with
  Applications, 23 (2015), pp.~314--339,
  \url{https://doi.org/10.1002/nla.2026}.

\bibitem{ShapiroDentchevaRuszczynski09}
{\sc A.~Shapiro, D.~Dentcheva, and A.~Ruszczynski}, {\em Lectures on Stochastic
  Programming: Modeling and Theory}, Society for Industrial and Applied
  Mathematics, 2009.

\bibitem{TuChoiPark99}
{\sc J.~Tu, K.~K. Choi, and Y.~H. Park}, {\em {A New Study on Reliability-Based
  Design Optimization}}, Journal of Mechanical Design, 121 (1999),
  pp.~557--564, \url{https://doi.org/10.1115/1.2829499},
  \url{https://doi.org/10.1115/1.2829499},
  \url{https://arxiv.org/abs/https://asmedigitalcollection.asme.org/mechanicaldesign/article-pdf/121/4/557/5920868/557\_1.pdf}.

\bibitem{VittaldevRussell16}
{\sc V.~Vittaldev and R.~P. Russell}, {\em Multidirectional gaussian mixture
  models for nonlinear uncertainty propagation}, Computer Modeling in
  Engineering \& Sciences, 111 (2016), pp.~83--117,
  \url{https://doi.org/10.3970/cmes.2016.111.083}.

\bibitem{VittaldevRussellLinares16}
{\sc V.~Vittaldev, R.~P. Russell, and R.~Linares}, {\em Spacecraft uncertainty
  propagation using gaussian mixture models and polynomial chaos expansions},
  Journal of Guidance, Control, and Dynamics, 39 (2016), pp.~2615--2626,
  \url{https://doi.org/10.2514/1.g001571}.

\bibitem{YangStadlerMoserEtAl11}
{\sc S.~Yang, G.~Stadler, R.~Moser, and O.~Ghattas}, {\em A shape
  {H}essian-based boundary roughness analysis of {N}avier--{S}tokes flow}, SIAM
  Journal on Applied Mathematics, 71 (2011), pp.~333--355,
  \url{https://doi.org/10.1137/100796789},
  \url{https://doi.org/10.1137/100796789}.

\bibitem{YounChoi04}
{\sc B.~D. Youn and K.~K. Choi}, {\em Selecting probabilistic approaches for
  reliability-based design optimization}, {AIAA} Journal, 42 (2004-01),
  pp.~124--131, \url{https://doi.org/10.2514/1.9036}.

\end{thebibliography}

\end{document}